\documentclass[11pt]{article}
\usepackage{amssymb,amsmath,amsfonts,amsthm,mathtools,bm,enumerate,color}
\usepackage{mathrsfs}
\usepackage[toc,page,title,titletoc,header]{appendix}
\usepackage{graphicx}

\usepackage{paralist}

\usepackage{tikz}
\usetikzlibrary{arrows,positioning,shapes.geometric}

\usepackage{algorithm,algorithmic}

\graphicspath{ {./figures/} }

\usepackage{indentfirst}
\usepackage{multicol}
\usepackage{booktabs}
\usepackage{url}
\usepackage[outdir=./]{epstopdf} 

\usepackage{hyperref}
\hypersetup{
    colorlinks=true, 
    linktoc=all,     
    linkcolor=blue,  
}
\numberwithin{equation}{section}

\newtheorem{Theorem}{Theorem}[section]
\newtheorem{Lemma}[Theorem]{Lemma}
\newtheorem{Proposition}[Theorem]{Proposition}
\newtheorem{Assumption}{H.\!\!}

\newcommand*\diff{\mathop{}\!\mathrm{d}}

\theoremstyle{definition}
\newtheorem{Definition}{Definition}[section]

\theoremstyle{remark}
\newtheorem{Remark}{Remark}[section]

 \def\p{\partial} 
\def\to{\rightarrow}
 \def\ol{\overline}    
\def\Om{\Omega}  \def\om{\omega} 
 
\newcommand{\q}{\quad}

\def\l{\label}    \def\fa{\forall}
\def\b{\beta}  \def\a{\alpha} 
 
\def\eps{\varepsilon}

 \def\t{\times}  
\def\ms{\medskip}

\def \la{\langle} \def\ra{\rangle}

\def\cA{\mathcal{A}}

\def\cF{\mathcal{F}}

\def\cH{\mathcal{H}}
\def\cI{\mathcal{I}}

\def\cL{\mathcal{L}}

\def\cO{\mathcal{O}}
\def\cP{\mathcal{P}}

\def\cS{\mathcal{S}}

\def\cW{\mathcal{W}}
\def\cX{\mathcal{X}}
\def\cY{\mathcal{Y}}

\def\d{{\mathrm{d}}}

\def\bA{{\textbf{A}}}

\def\sB{\mathbb{B}}

\def\sE{{\mathbb{E}}}
\def\sF{{\mathbb{F}}}
\def\sI{{\mathbb{I}}}

\def\sN{{\mathbb{N}}}
\def\sP{\mathbb{P}}

\def\sR{{\mathbb R}}

\def\Lx{L^\mathrm{x}}
\def\Lxa{L^\mathrm{(x,a)}}

\DeclareMathOperator*{\argmin}{arg\,min}
\DeclareMathOperator*{\esssup}{ess\,sup}

\newcommand{\lc}
{\mathrel{\raise2pt\hbox{${\mathop<\limits_{\raise1pt\hbox
{\mbox{$\sim$}}}}$}}}

\newcommand{\gc}
{\mathrel{\raise2pt\hbox{${\mathop>\limits_{\raise1pt\hbox{\mbox{$\sim$}}}}$}}}

\newcommand{\ec}
{\mathrel{\raise2pt\hbox{${\mathop=\limits_{\raise1pt\hbox{\mbox{$\sim$}}}}$}}}

\def\bb{\begin{equation}} \def\ee{\end{equation}}
\def\bbn{\begin{equation*}} \def\een{\end{equation*}}

\def\beqn{\begin{eqnarray}}  \def\eqn{\end{eqnarray}}

\def\beqnx{\begin{eqnarray*}} \def\eqnx{\end{eqnarray*}}

\def\bn{\begin{enumerate}} \def\en{\end{enumerate}}

\def\bd{\begin{description}} \def\ed{\end{description}}



\usepackage{biblatex}
\renewbibmacro{in:}{}
\begin{document}

\title{Convergence Rates of 
Time Discretization in Extended Mean Field Control

}

\author{
Christoph Reisinger\thanks{
Mathematical Institute, University of Oxford, Oxford OX2 6GG, UK
 ({\tt \{christoph.reisinger, wolfgang.stockinger, maria.tsianni\}@maths.ox.ac.uk})}
\and
Wolfgang Stockinger\footnotemark[1]
\and Maria Olympia Tsianni\footnotemark[1]
\and
Yufei Zhang\thanks{Department of Mathematics, Imperial College London, Weeks Building, 16-18 Prince's Gardens, London, SW7 1NE, UK
({\tt yufei.zhang@imperial.ac.uk})}
}
\date{}

\maketitle

\noindent\textbf{Abstract.} 
Piecewise constant control approximation provides a practical framework for designing numerical schemes of continuous-time control problems. We analyze the accuracy of  such approximations for extended mean field control (MFC) problems, where the dynamics and costs depend on the joint distribution of states and controls. For linear–convex extended MFC problems, we show that the optimal control is $1/2$-H\"older continuous in time. Using this regularity, we prove that the optimal cost of the continuous-time problem can be approximated by piecewise constant controls with order $1/2$, while the optimal control itself can be approximated with order $1/4$. For general extended MFC problems, we further show that, under sufficient regularity of the value functions, the value functions converge with an  improved first-order rate,
matching the best-known rate  for classical control problems without mean field interaction,   and consistent with   the numerical observations for MFC of  Cucker–Smale models.


\ms


\medskip


\section{Introduction}\l{sec:intro}

\paragraph{Extended mean field control (MFC) problem.}

 Let 
 $T>0$ be a  given  terminal time,
$(\Om,\cF,\sP)$ be a complete  probability space
on which a  $d$-dimensional Brownian motion $(W_t)_{t\in [0,T]}$
is defined,
 $\sF = (\cF_t)_{t\in [0,T]}$ be the natural filtration of $W$
 augmented with an independent $\sigma$-algebra $\cF_0$,
and 
 $\cA$ be the set of 
 square--integrable 
$\sF$-progressively measurable processes
$\a=(\a_t)_{ t\in [0,T]}$ taking values in 
an   action set
$\bA\subset \sR^k$.
For any  initial 
state $\xi_0\in L^2(\cF_0;\sR^n)$
and control $\a\in \cA$,
  consider  the state process $X^\alpha$ governed by
the following controlled  McKean--Vlasov diffusion:
\bb\l{eq:mfcE_fwd}
\d X^\a_t=
b(t,X^\a_t,\a_t,\sP_{(X^\a_t,\a_t)})\, \d t
+\sigma(t,X^\a_t,\a_t, \sP_{(X^\a_t,\a_t)})\, \d W_t,
\q t\in [0,T];
\q X^\a_0=\xi_0, 
\ee
where $b$ and $\sigma$ are given measurable  functions
taking values in $\sR^n$ and $\sR^{n\times d}$, respectively,
and 
$\sP_{U}$ denotes   the  law of a given random variable $U$.  
The value  function of the extended MFC  problem is defined by
\bb\l{eq:mfcE}
V(\xi_0)=\inf_{\a\in \cA} J(\a;\xi_0)
\q
\textnormal{with}
\;
J(\a;\xi_0)=\sE\bigg[
\int_0^T f(t,X^\a_t,\a_t,\sP_{(X^\a_t,\a_t)})\, \d t+g(X^\a_T,\sP_{X^\a_T})
\bigg],
\ee
where 
  $f$ and   $g$
 are given real valued  cost functions.
Precise conditions on $\bA$, $b$, $\sigma$, $f$ and $g$ will be given in Section \ref{sec:main_result}.

In the above extended MFC problems,
 both the state coefficients and the cost function  involve mean field interactions through the joint distribution of the state and control processes. Such extended MFC problems characterize the large-population limit of control systems in which agents interact weakly through both their states and controls
\cite{
acciaio2019,  
djete2022extended,picarelli2025extended},
and have numerous applications in economics, biology, and social interactions  (see e.g.~\cite{
  carmona2018a,   acciaio2019,  gobet2019, burzoni2020}). 
When 
the state coefficients 
and cost functions 
depend only on  the marginal  distribution of the state process,
  \eqref{eq:mfcE} reduces to classical  MFC problems 
{(without control interactions)}
studied in 
\cite{chassagneux2014,  carmona2015, carmona2018a,carmona2019}.

\paragraph{Piecewise constant control approximation.}
To solve \eqref{eq:mfcE}, one typically relies on numerical schemes, as explicit solutions are rarely obtainable for (extended) MFC problems. 
A common strategy is to   approximate \eqref{eq:mfcE}  by discrete-time control problems  with  piecewise constant controls.
Specifically,
for a  given 
 partition  $\pi=\{0=t_0<\cdots<t_N=T\}$ of $[0,T]$,
consider the following discrete--time approximation of 
     \eqref{eq:mfcE}:
 \bb\l{eq:mfcE_constant}
V_\pi(\xi_0)=\inf_{\a\in \cA_\pi} J_\pi(\a;\xi_0),
\ee
where  
$\cA_{\pi}\subset \cA $ is the   subset of controls   that are constant on each subinterval $[t_i,t_{i+1})$ in $\pi$:
\begin{align}\l{eq:A_pi}
\cA_{\pi}\coloneqq
\Big\{\alpha\in\cA:  
{\; \forall \omega \in \Omega } \;  \exists a_i\in \bA, \, i=0,\ldots, N-1,  \;  \text{ s.t. }  \alpha_s{(\omega)} \equiv \sum^{N-1}_{i=0} a_i \bm{1}_{[t_i,t_{i+1})}(s)
\Big\},
\end{align}
and for each $\a\in \cA_{\pi}$, $J_\pi(\a;\xi_0)$ is the discretized cost functional defined by
\bb\l{eq:J_pi}
J_\pi(\a;\xi_0)\coloneqq\sE\bigg[\sum_{i=0}^{N-1}
\int_{t_i}^{t_{i+1}} f(t_i,X^{\a,\pi}_{t_i},\a_{t_i},\sP_{(X^{\a,\pi}_{t_i},\a_{t_i})})\, \d t+g(X^{\a,\pi}_T,\sP_{X^{\a,\pi}_T})
\bigg],
\ee
and $X^{\a,\pi}$ is governed by  the  Euler--Maruyama  approximation of  \eqref{eq:mfcE_fwd}:
$X^{\a,\pi}_0=\xi_0$ and  for all $i\in \{0,\ldots,N-1\}$, 
\begin{align}\l{eq:mfcE_fwd_euler}
\begin{split}
X^{\a,\pi}_{t_{i+1}}
&=X^{\a,\pi}_{t_{i}}
+
b(t_i,X^{\a,\pi}_{t_i},\a_{t_i},\sP_{(X^{\a,\pi}_{t_i},\a_{t_i})})\, (t_{i+1}-t_i)
\\
&\quad 
+\sigma(t_i,X^{\a,\pi}_{t_i},\a_{t_i},\sP_{(X^{\a,\pi}_{t_i},\a_{t_i})})\, (W_{t_{i+1}}-W_{t_{i}}).
\end{split}
\end{align}
Based on 
the discrete-time approximation \eqref{eq:mfcE_constant},  
various numerical schemes have been proposed to solve \eqref{eq:mfcE}.
For instance, one can approximate the marginal laws $(\sP_{(X^{\a,\pi}_{t_i},\a_{t_i})})_{i=0}^N$ in \eqref{eq:J_pi}-\eqref{eq:mfcE_fwd_euler} using particle methods and solve the resulting finite-dimensional control problem with neural networks;  
see \cite{carmona2019,   pham2022mean, picarelli2025extended, ReiStoZha1}  and the references therein.

Given the central role of \eqref{eq:mfcE_constant} in solving \eqref{eq:mfcE}, a natural question is how accurately the continuous-time problem \eqref{eq:mfcE} can be approximated by the discrete-time problem \eqref{eq:mfcE_constant} with piecewise constant controls. 
Specifically,  we would answer the following two questions:
\begin{description}
\item[Q1]: What is the convergence rate of $|V_\pi(\xi_0)-V(\xi_0)|$ in terms of the stepsize $|\pi|$?
\item[Q2]: How does a (sub-)optimal control  of \eqref{eq:mfcE_constant}
approximate the optimal control of
\eqref{eq:mfcE}?
\end{description}

To the best of our knowledge,
there is no published work on  
 the accuracy of piecewise
constant policy approximation for extended  MFC problems. 
Moreover,  even for classical control problems without mean field interactions,  convergence of the controls (as in  Q2)   has not been established.

The convergence of value functions has been studied by   \cite{krylov1999, jakobsen2019,carmona2019,legrand2023convergence}    in settings simpler than those considered here.   
\cite{krylov1999}
analyzed classical control problems  without mean field interactions,
and 
established an order $1/6$  convergence rate  of   the value functions, assuming the coefficients are Lipschitz continuous and uniformly bounded. \cite{jakobsen2019} improved the  rate  to order $1/4$  under the same assumptions, and further obtained   an optimal order $1$ convergence by assuming the coefficients are four times differentiable with bounded derivatives. \cite{legrand2023convergence} shows order 1/3 under additional semi-convexity and semi-concavity assumptions on the data.
\cite{carmona2019} studies classical MFC problems with only drift  controls, restricts the admissible controls to closed-loop controls induced by sufficiently regular feedback maps, and proves an order $1/2$ convergence of the value functions. The regularity assumptions implicitly require all cost functions to be three times differentiable with bounded, Lipschitz-continuous derivatives (see Remark \ref{rmk:regularity_comparison} for details). 

Note that the boundedness and high-order differentiability assumptions on the state coefficients and cost functions
imposed in \cite{krylov1999, jakobsen2019,carmona2019,legrand2023convergence} 
are restrictive, as they exclude linear dynamics and quadratic costs found in the popular linear-quadratic models  \cite{pham2018}, as well as more general convex costs studied in \cite{carmona2015, acciaio2019, gobet2019, 
szpruch2021exploration,
guo2023reinforcement}.

\paragraph{Our work.}
This paper consists of two parts. 
 
\begin{itemize}
 
\item
We first analyze a class of linear–convex extended MFC problems, where the drift coefficient depends linearly on the state and control, the diffusion coefficient is uncontrolled, and the cost functions are convex with Lipschitz continuous derivatives. 
In this setting, we establish the optimal control of \eqref{eq:mfcE}
is $1/2$-H\"older continuous in time. 
Leveraging this time regularity, we show that the optimal value
and control
of  \eqref{eq:mfcE_constant} converge to the optimal  value and control  of \eqref{eq:mfcE} with orders $1/2$ 
and $1/4$, respectively
(Theorems \ref{thm:PCPT_discrete_value} and \ref{thm:control_PCPT_discrete}).

These convergence results are novel even for standard control problems (without mean field interactions) and for MFC problems (without control interactions).

\item 
For general extended MFC problems, we show that, under sufficient regularity of the value functions, the value functions corresponding to piecewise constant controls converge with an improved first-order rate to the value function of \eqref{eq:mfcE} (Theorem \ref{improvedorder}).  This extends the optimal first-order convergence result in \cite{jakobsen2019} to the present setting with mean field dependence on both state and control.
We confirm this first-order convergence through numerical experiments on multi-dimensional MFC problems arising from the consensus control of nonlinear Cucker–Smale models (see Section \ref{numericalresults}). 

 To the best of our knowledge, this is the first time a first-order convergence result for piecewise constant control approximations has been established in MFC problems.
\end{itemize}

\paragraph{Our approaches and most related works.} 

Note that 
\cite{krylov1999, jakobsen2019,legrand2023convergence}
establish the time discretization errors of classical control problems  by 
applying  It\^{o}'s lemma and a dynamic programming principle.
To handle coefficients without sufficiently high regularity, \cite{krylov1999, jakobsen2019,legrand2023convergence} employ a sophisticated ``shaking the coefficients" technique, pioneered by Krylov, to mollify the coefficients and balance the regularization error with the time discretization error, resulting in the aforementioned fractional convergence orders.

This approach cannot be easily adapted to analyze MFC problems with unbounded   coefficients that are not twice differentiable,   in particular the linear–convex MFC problems studied in our paper. 
Indeed, due to the dependence of the coefficients on marginal distributions, it is well-known that one must lift the coefficients and value functions to mappings defined on the infinite-dimensional space of probability measures to recover a dynamic programming principle (see, e.g., \cite{pham2018}). For these infinite-dimensional spaces, there is currently no known regularization technique with a quantifiable regularization error analogous to the finite-dimensional setting as in \cite{krylov1999, jakobsen2019}.

Instead, we exploit the linear-convex structure of the control problems, and  characterize the optimal control  of \eqref{eq:mfcE}
via the    Pontryagin maximum principle \cite{acciaio2019}.
Leveraging this characterization, we use Malliavin calculus to establish the time regularity of the optimal control of of \eqref{eq:mfcE}.
This   allows  for quantifying  the time discretization error   of the optimal control without imposing strong regularity assumptions as in \cite{carmona2019, jakobsen2019}.

A key step is to analyze the solution regularity of the Pontryagin system, which is technically involved due to the mean field interactions in the controls. 
To see it, 
recall that
in the setting
with uncontrolled diffusion coefficients, 
by  \cite[Remark 3.4]{acciaio2019},
an optimal  control $\hat{\a}$  of \eqref{eq:mfcE}
satisfies 
\begin{equation}\label{eq:nonMarkovian_optimal}
\hat{\a}_t=\argmin\lbrace \mathbb{E}[H^{\textrm{re}}(t,X_t,\beta,\mathbb{P}_{(X_t,\beta)},Y_t)] \mid \ \beta \in L^2(\mathcal{F}_t;\bA) \rbrace,
\q \textnormal{$ \d \sP\otimes \d t$-a.e.},
\end{equation}
where $H^{\textrm{re}}$ is the associated   Hamiltonian of \eqref{eq:mfcE} (see \eqref{eq:hamiltonian_re}), 
$X$ is the state process controlled by $\hat \alpha$,
and 
and $Y$ is an adjoint process satisfying 
an McKean-Vlasov backward  stochastic differential equation,
which is 
generally non-Markovian since $\hat \alpha$ may be non-Markovian.  
The main issue is that   
for extended MFC problems,  
\eqref{eq:nonMarkovian_optimal}
in general does not imply 
$\hat{\a}_t=\psi(t,X_t,Y_t,\sP_{(X_t,Y_t)})$,
with $\psi$ being the pointwise minimizer 
 of   $H^{\textrm{re}}$, 
 as highlighted by the linear-quadratic  example
 in \cite[Remark 4.2]{acciaio2019}. 
 Consequently, unlike in classical MFC problems studied in \cite{carmona2015}, one cannot use the pointwise minimizer $\psi$ of $H^{\textrm{re}}$ to reduce the Pontryagin system to a Markovian system. 

We show that, under suitable conditions, the optimal control can be represented as a Lipschitz function of the state and adjoint processes together with their marginal laws.  
The desired feedback  function 
can be obtained
either by constructing it from a modified Hamiltonian (Propositions \ref{prop:mfc_hat_a} and \ref{prop:control_measure}) or by explicitly solving the first-order condition
(Proposition   \ref{prop:lq}).
This allows for reducing 
 the non--Markovian Pontryagin system
 to a 
 McKean-Vlasov backward  stochastic differential equation
 (MV-FBSDE).
We analyze the well-posedness and   stability   of the coupled  MV-FBSDE
by adapting the continuation method in \cite{bensoussan2015, carmona2015, guo2023reinforcement},
and further establish the H\"{o}lder regularity of the solutions 
by extending the path regularity results for decoupled FBSDEs in \cite{zhang2017}.

In the smooth setting, we extend the analysis
in   \cite{jakobsen2019}   to extended MFC  problems,
working under sufficient regularity assumptions on the data and value function such that no regularization is required.
Using an iterated It{\^o} expansion of the value function in the time and measure components,   we show that the local error on each subinterval is of second order in the step size. 
Moreover, we  
derive a sub-dynamic programming principle (Lemma \ref{lemma1}) for the extended MFC problem with piecewise constant controls, accurate to first order in the step size. With these  in hand, we aggregate the local errors over time and obtain the first order convergence.

\paragraph{Notation.}
For any given $n\in \sN$ and $x\in \sR^n$, we  denote by 
  $\sI_n$  the $n\t n$ identity matrix,
  by ${0}_{n}$ the  zero element
  of  $\sR^{n}$
   and
by
 $\bm{\delta}_{x}$
 the Dirac measure supported at $x$.
We  denote by $\la \cdot,\cdot\ra$
the usual inner product in a given Euclidean space
and by   $|\cdot|$ the norm induced by $\la \cdot,\cdot\ra$,
which in particular satisfies  for all 
$n,m,d\in \sN$
and
$\theta_1=(x_1,y_1,z_1),\theta_2=(x_2,y_2,z_2)\in \sR^n\t \sR^m\t \sR^{m\t d}$
that
$\la z_1,z_2\ra =\textrm{trace}(z^*_1z_2)$
and 
$\la \theta_1,\theta_2\ra =\la x_1,x_2\ra+\la y_1,y_2\ra+\la z_1,z_2\ra$.

We   introduce the following spaces:
for each $p \ge 1$, $k \in \sN$, $t\in [0,T]$
and Euclidean space $(E,|\cdot|)$,
$L^p(\Om; E)$ is the space  of 
 $E$-valued
$\cF$-measurable
random variables $X$ satisfying
$\|X\|_{L^p}=\sE[|X|^p]^{1/p}<\infty$;
$L^p(\cF_t; E)$ is the subspace  of $L^p( \Om;E)$
containing all 
$\cF_t$-measurable
random variables;
$\cS^p(t,T;E)$ is the space of 
$\sF$-progressively  measurable processes
$Y: \Om\t [t,T]\to E$  
satisfying $\|Y\|_{\cS^p}=\sE[\esssup_{s\in [t,T]}|Y_s|^p]^{1/p}<\infty$;
$\cH^p(t,T; E)$ is the space of 
  $\sF$-progressively measurable
 processes 
$Z: \Om\t [t,T]\to E$   
 satisfying $\|Z\|_{\cH^p}=\sE[(\int_t^T|Z_s|^2\,\d s)^{p/2}]^{1/p}<\infty$.
 To simplify the  notation,
 we 
denote 
  $\cS^p=\cS^p(0,T;E)$
  and $\cH^p=\cH^p(0,T;E)$.

Given an    
Euclidean space $(E, |\cdot|)$,
we denote by $\cP_2(E)$  the  space of  probability measures 
$\mu$
on   $E$ satisfying $\|\mu\|_2=(\int_E |x|^2\,\d \mu(x))^{1/2}<\infty$,
endowed with the   Wasserstein metric defined by: 
$$
\cW_2(\mu_1,\mu_2)
\coloneqq \inf_{\kappa\in \Pi(\mu_1,\mu_2)} \left(\int_{E\t E}|x-y|^2\d \kappa( x, y)\right)^{1/2},
\q \mu_1,\mu_2\in \cP_2(E),
$$
where $\Pi(\mu_1,\mu_2)$ is the set of all couplings of $\mu_1$ and $\mu_2$, i.e.,
$\kappa\in \Pi(\mu_1,\mu_2)$ is a probability measure on $E\t E$ such that $\kappa(\cdot\t E)=\mu_1$ 
and $\kappa(E\t \cdot)=\mu_2$.
For a given function $h:\cP_2(\sR^n\t \sR^k)\to \sR$
and a measure  $\eta\in \cP_2(\sR^n\t \sR^k)$
with marginals $\mu\in \cP_2(\sR^n)$, $\nu\in \cP_2(\sR^k)$,
 we  denote by $\p_\eta h(\eta)(\cdot):\sR^n\t \sR^k\to \sR^n\t \sR^k$
the  L-derivative of $h$ at $\eta$ 
and by 
$(\p_\mu h(\eta),\p_\nu h(\eta))(\cdot):\sR^n\t \sR^k\to \sR^n\t \sR^k$
the partial L-derivatives of 
$h$ with respect to  the marginals;
see e.g. \cite[Section 2.1]{acciaio2019} or \cite[Chapter 5]{carmona2018a}
for detailed definitions.

Finally,
we    denote by $C\in [0,\infty)$ a generic constant
throughout this paper, which is independent of the initial condition $\xi_0$,
though it may depend on the constants appearing in the assumptions  
  and may take a different value at each occurrence.


\section{Main results}
\label{sec:main_result}

 This section outlines the model assumptions and presents the main convergence rate results for the discrete-time approximation \eqref{eq:J_pi}–\eqref{eq:mfcE_fwd_euler}.
We establish two distinct sets of results: 
\begin{itemize}
    \item For   linear–convex   problems with only drift controls, 
    we prove a half-order convergence rate by deriving the optimal time regularity of the optimal control via the maximum principle.
    \item For general control problems with a sufficient   regular   value function, 
     we   prove the optimal first-order convergence rate using a dynamic programming approach.
\end{itemize}

\subsection{Error analysis for linear-convex extended MFC problems}
\label{sec:error_linear_convex}

\subsubsection{Stochastic maximum principle} 
In this section, 
 we impose the following structural assumptions on the coefficients of \eqref{eq:mfcE_fwd} and \eqref{eq:mfcE} to facilitate the application of the maximum principle.

\begin{Assumption}\l{assum:mfcE}
Let $\bA\subset \sR^k$ be a nonempty closed convex set, 
   $(b,\sigma):[0,T]\t \sR^n\t \sR^k\t  \cP_2(\sR^n\t \sR^k)\to \sR^n\times \sR^{n\t d} $,
$f:[0,T]\t \sR^n\t \sR^k\t  \cP_2(\sR^n\t \sR^k)\to\sR$ 
and $g:\sR^n\t  \cP_2(\sR^n)\to\sR$ 
be measurable functions satisfying the following properties:
\begin{enumerate}[(1)]
\item \label{item:mfcE_lin}
$b$ and  $\sigma$ are of the following affine form: for all $(t,x,a, \eta)\in [0,T]\t \sR^n\t \sR^k\t   \cP_2(\sR^n \times \sR^k)$,
\begin{align}
b(t,x,a,\eta)&=b_0(t)+b_1(t)x + b_2(t)a+b_3(t)\bar{\eta},
\label{eq:b_affine}\\
\sigma(t,x, a,\eta)&=\sigma_0(t)+\sigma_1(t)x+\sigma_2(t)\ol{\pi_1\sharp \eta},
\end{align}
where 
 $( b_0, b_1, b_2, b_3,\sigma_0,\sigma_1,\sigma_2)\in L^\infty(   [0,T];   \sR^n\times  \sR^{n\t n}\t  \sR^{n\t k} \t\sR^{n\t (n+k)} \t \sR^{n\t d}\t \sR^{(n\t d)\t n}\t \sR^{(n\t d)\t n})$,  
$\bar{\eta} $
is  the first moment of   $\eta$,
and 
$\pi_1\sharp \eta$ is the first marginal of   $\eta$.
With a slight abuse of notation, we identify $\sigma$ as a function $\sigma:
[0,T]\t \sR^n \t  \cP_2(\sR^n)\to  \sR^{n\t d}$.


\item \l{item:mfcE_growth}
$ f(\cdot,0,0,\bm{\delta}_{0_{n+k}})\in L^\infty(0,T)$,
$f$ and  $g$ are  differentiable with respect to $(x,a,\eta)$ and $(x,\mu)$, respectively,
and all derivatives are of linear growth  and Lipschitz continuous. 

That is, 
there exists  $\hat{L}\in [0,\infty)$ such that 
for all $R\ge 0$ and all $(t,x,a,\mu,\eta)$ with 
$|x|$, $ |a|$, $ \|\mu\|_2$, $ \| \eta \|_2\le R$,
$|\p_xf(t,x,a,\eta)|+|\p_a f(t,x,a,\eta)|+|\p_x g(x,\mu)|\le \hat{L}(1+R)$,
the $L^2(\sR^{n}\t\sR^k,\eta)$-norms
of the maps $(x',a') \mapsto \p_\mu f(t,x,a,\eta)(x',a')$, $(x',a') \mapsto \p_\nu f(t,x,a,\eta)(x',a')$  
are bounded by $\hat{L}(1+R)$,
and the $L^2(\sR^{n},\mu)$-norm of the map
  $x'\mapsto \p_\mu g(x,\mu)(x')$
is bounded by $\hat{L}(1+R)$.
\color{black}
%
Moreover, there exists $\tilde{L}\in [0,\infty)$ such that 
for all $t\in [0,T]$,
$\p_{x}f(t,\cdot)
:\sR^n\t \bA\t \cP_2(\sR^n \times \sR^{k}) \to \sR^n$,
$\p_{a}f(t,\cdot)
:\sR^n\t \bA\t \cP_2(\sR^n \times \sR^{k}) \to \sR^k$
and 
$\p_xg(\cdot):\sR^n\t \cP_2(\sR^n)\to \sR^n$
are 
$\tilde{L}$-Lipschitz continuous.
Moreover, for any $(t,x,a,\eta,\mu)\in [0,T]\t \sR^n\t \sR^k\t \cP_2(\sR^n \times \sR^{k})\t \cP_2(\sR^n)$,
there exist versions of 
$\p_\mu f(t,x,a,\eta)(\cdot)$,  $\p_\nu f(t,x,a,\eta)(\cdot)$ and  $\p_\mu g(x,\mu)(\cdot)$
such that 
\begin{align*}
& (x,a,\eta,\mu,x',a')\in \sR^n\t \bA \t \cP_2(\sR^n \times \sR^{k})\t \cP_2(\sR^n)\t \sR^n \t \bA   \\
& \quad \mapsto (\p_\mu f(t,x,a,\eta)(x',a'), \p_\nu f(t,x,a,\eta)(x',a'),\p_\mu g(x,\mu)(x'))\in \sR^n \t \sR^k \t \sR^n,
\end{align*}
is $\tilde{L}$-Lipschitz continuous.

\item\l{item:mfcE_convex}
 $f$ and $g$ are  convex 
in  $(x,a,\eta)$ 
and  $(x,\mu)$, respectively. That is, 
 there exists $\lambda_1,\lambda_2\ge 0$
such that   $\lambda_1+\lambda_2>0$ and 
for all $t\in [0,T]$, $(x,a,\eta),(x',a',\eta')\in \sR^n\t \bA \t \cP_2(\sR^n \t \sR^{k})$,
\begin{align*}
&f(t,x',a',\eta')-f(t,x,a,\eta)-
\la \p_{x}f(t,x,a,\eta), x'-x\ra 
-\la \p_{a}f(t,x,a,\eta), a'-a\ra 
\\
&\q 
-\tilde{\sE}[\la\p_\mu f(t,x,a,\eta)(\tilde{X},\tilde{\alpha}),\tilde{X}'-\tilde{X}\ra 
+ \la\p_\nu f(t,x,a,\eta)(\tilde{X},\tilde{\alpha}),\tilde{\alpha}'-\tilde{\alpha}\ra ]
\\
&\q
\ge \lambda_1|a'-a|^2+\lambda_2\tilde{\sE}[|\tilde{\alpha}'-\tilde{\alpha}|^2],
\end{align*}
whenever  $(\tilde{X},\tilde{\a}),(\tilde{X}',\tilde{\a}')\in L^2(\tilde{\Om},\tilde{\cF},\tilde{\sP};\sR^n\t\sR^k)$
with distributions $\eta$ and $\eta'$, respectively.
Moreover, for all $(x,\mu),(x',\mu')\in \sR^n \t \cP_2(\sR^n)$,
\begin{align*}
g(x',\mu')-g(x,\mu)-\la \p_{x}g(x,\mu), x'-x \ra 
-\tilde{\sE}[\la\p_\mu g(x,\mu)(\tilde{X}),\tilde{X}'-\tilde{X}\ra ]\ge 0,
\end{align*}
whenever  $\tilde{X},\tilde{X}'\in L^2(\tilde{\Om},\tilde{\cF},\tilde{\sP};\sR^n)$
with distributions $\mu$ and $\mu'$, respectively.
Above and hereafter, we denote by $\tilde{\sE}$  the expectation on  $(\tilde{\Om},\tilde{\cF},\tilde{\sP})$.
\end{enumerate}
\end{Assumption}

\begin{Remark}\l{rmk:mfcE_regularity}
(H.\ref{assum:mfcE})    extends
  Assumption ``Control of MKV Dynamics" in 
\cite{carmona2018a} 
to the present setting with   mean field interactions through controls.
Note that (H.\ref{assum:mfcE}) only requires 
  the cost function $f$ to be strongly convex
either in the control variable or in its law,
which allows for considering 
cost functions that do  not explicitly depend on the state of the controls 
(see e.g.~Proposition \ref{prop:control_measure}).
Condition (H.\ref{assum:mfcE}(\ref{item:mfcE_lin})) requires  that the diffusion coefficient is  uncontrolled, which is important for  analyzing  the regularity of the optimal control 
via a probabilistic approach
(see Theorem  \ref{TH:ControlRegularity}).

Under (H.\ref{assum:mfcE}),
$f$  and $g$ are   
  locally Lipschitz continuous
and  at most of quadratic growth 
with respect to 
$(x,a,\eta)$ and 
$(x,\mu)$, respectively.
Indeed, 
by  (H.\ref{assum:mfcE}(\ref{item:mfcE_growth})), 
there exists   $C\ge 0$ such that 
for all $(x,a,\eta),(x',a',\eta')\in \sR^n\t \bA \t \cP_2(\sR^n \t \sR^{k})$, 
\begin{align}
\label{eq:f_local_lipschitz}
\begin{split}
&|f(t,x,a,\eta)-f(t,x',a',\eta')|
\\
&\le C(1+|x|+|x'|+|a|+|a'|+\|\eta\|_2+\|\eta'\|_2)(|(x,a)-(x',a')|+\cW_2(\eta,\eta')),
\end{split}
\end{align}
which along with the  boundedness of $|f(t,0,0,\bm{\delta}_{0_{n+k}})|$ implies the quadratic growth of $f$.
Similar arguments yield the desired properties of   $g$.

\end{Remark}

Under Condition (H.\ref{assum:mfcE}),
for any given  
initial state $\xi_0\in L^2(\cF_0;\sR^n)$
and
   control $\a\in \cA$,
the  state process $X^\a\in \cS^2(\sR^n)$ is well-defined by \eqref{eq:mfcE_fwd},   
and the cost   $J(\a;\xi_0)$
in \eqref{eq:mfcE}
is finite.

We    analyze  the   regularity of the optimal control   of 
     \eqref{eq:mfcE},
     which will be used to    quantify the time discretization error.  
 We adopt   a probabilistic approach via the stochastic maximum principle.    
To this end, 
define   the Hamiltonian
$H:[0,T]\t \sR^n\t \sR^k\t  \cP_2(\sR^n\t \sR^k)\t \sR^n\t \sR^{n\t d}\to \sR$ by:
\bb\l{eq:mfcE_hamiltonian}
H(t,x,a,\eta,y,z)\coloneqq\la b(t,x,a,\eta), y\ra +\la \sigma(t,x,\pi_1\sharp\eta),z\ra+f(t,x,a,\eta),
\ee 
and the reduced Hamiltonian  
$H^{\textrm{re}}:[0,T]\t \sR^n\t \sR^k\t  \cP_2(\sR^n\t \sR^k)\t \sR^n\to \sR$
  by:
\bb\label{eq:hamiltonian_re}
H^{\textrm{re}}(t,x, a,\eta,y)\coloneqq \la b(t,x,a,\eta), y\ra +f(t,x,a,\eta).
\ee
The following lemma 
characterizes   
the optimal control using
a   coupled forward-backward system,
which is generally non-Markovian since the optimal control itself may be non-Markovian.

\begin{Lemma}
\label{lemma:maximum_principle}
    Suppose (H.\ref{assum:mfcE}) holds.  
   An $\a\in \cA$ is an optimal control of \eqref{eq:mfcE} if and only if 
 for all $ a \in \bA$ and   $\mathrm{d}\mathbb{P} \otimes \mathrm{d}t$\ -a.e., 
\begin{align}\l{eq:opti_re}
\la \p_a H^{\textrm{re}}(t, X^\a_t,\a_t, \sP_{(X^\a_t,\a_t)},Y^\a_t) 
+
\tilde{\mathbb{E}} [\p_\nu  H^{\textrm{re}}(t, \tilde{X}^\a_t,\tilde{\a}_t, {\sP}_{({X}^\a_t,{\a}_t)},\tilde{Y}^\a_t)(X^\a_t,\alpha_t)], \alpha_t -a \ra\leq 0, \quad 
\end{align}
where  $X^\alpha \in \cS^2(\sR^n)$ 
satisfies \eqref{eq:mfcE_fwd}, and 
 $(Y^\a,Z^\a)\in \cS^2(\sR^n)\t \cH^2(\sR^{n\t d})$  
satisfies  
 for all $t\in [0,T]$,  
 \begin{align}\l{eq:mfcE_bsde_nonMarkov}
\begin{split}
\d Y^{{\a}}_t&=-\big(\p_x H(\theta^\a_t,Y^{{\a}}_t,Z^{{\a}}_t)
+\tilde{\sE}[\p_\mu H(\tilde{\theta}^\a_t,\tilde{Y}^{{\a}}_t,\tilde{Z}^{{\a}}_t)(X^{{\a}}_t,{\a}_t)]\big)\,\d t
+Z^{{\a}}_t\,\d W_t,
\\
Y^{{\a}}_T&=\p_x g(X^{{\a}}_T,\sP_{X^{{\a}}_T})+\tilde{\sE}[\p_\mu g(\tilde{X}^{{\a}}_T,\sP_{X^{{\a}}_t})(X^{{\a}}_T)],
\end{split}
\end{align}
where   
$\theta^\a_t \coloneqq (t,X^{{\a}}_t,{\a}_t,\sP_{(X^{{\a}}_t,{\a}_t)})$ 
and  the tilde notation refers  to an independent copy. 
\end{Lemma}

To prove Lemma \ref{lemma:maximum_principle},
recall that by 
 the    maximum principle \cite[Theorem 3.5]{acciaio2019},
 $\alpha$ is optimal if and only if
\begin{align}\label{eq:opti}
\la \p_a H(\theta^\a_t,Y^\a_t,Z^\a_t) + \tilde{\mathbb{E}} [\p_\nu  H(\tilde{\theta}^\a_t,\tilde{Y}^\a_t,\tilde{Z}^\a_t)(X^\a_t,\alpha_t)], \alpha_t -a \ra\leq 0, \quad 
\textnormal{$\forall a \in \bA, \ \mathrm{d}\mathbb{P} \otimes \mathrm{d}t$\ -a.e.}
\end{align}
The equivalence between  
\eqref{eq:opti} and \eqref{eq:opti_re}
follows   from 
 the fact that $\sigma$ is uncontrolled. 
  
\subsubsection{Feedback representation of the optimality condition}
 
To reduce the non-Markovian system in Lemma  \ref{lemma:maximum_principle} into  a Markovian system, we assume that the optimality condition    \eqref{eq:opti_re} can be attained through a sufficiently regular feedback map from the state and adjoint processes to the action set.

\begin{Assumption}\phantomsection\l{assum:mfcE_hat}
  
Assume the notation of (H.\ref{assum:mfcE}). 
There exists a  measurable function $\hat{\a}:[0,T]\t \sR^n \t \sR^n \t \cP_2(\sR^{n} \t \sR^{n}) \to \bA$ 
and a constant $L_\a\in [0,\infty)$ satisfying the following properties:
\begin{enumerate}[(1)]
\item \l{item:ex} 
For all $t\in [0,T]$,
$|\hat{\a}(\cdot,0,0,\bm{\delta}_{0_{n + n}})| \leq L_\a$,
  $\hat{\a}(t,\cdot):\sR^n \t \sR^n \t \cP_2(\sR^{n} \t \sR^{n}) \to \bA$ is 
$L_\a$-Lipschitz continuous,
and satisfies  for all $(x,y,\eta,a)\in \sR^n \t \sR^n \t \cP_2(\sR^{n} \t \sR^{n})\t \bA$, 
\begin{align}\l{eq:opti_markov}
\begin{split}
&\la \p_a H^{\textrm{re}}(t,x,
\hat{\a}(t,x,y,\eta),
\phi(t,\eta),y)
\\
&\q
 +\int_{\sR^n\t \sR^n} \p_\nu  H^{\textrm{re}}(t,\tilde{x},\hat{\a}(t,\tilde{x},\tilde{y},\eta) ,\phi(t,\eta), \tilde{y})\big(x,\hat{\a}(t,x,y,\eta) \big)\,\d \eta(\tilde{x},\tilde{y}),
\\
&\q \hat{\a}(t,x,y,\eta) -a \ra\leq 0, 
\end{split}
\end{align}
where
$\phi(t,\eta) \coloneqq \eta \circ \big(\sR^n\t \sR^n\ni (x,y)\mapsto(x,\hat{\a}(t,x,y,\eta))\in \sR^n\t \bA\big)^{-1}$.

\item \l{item:HC} 
For all $t,t'\in [0,T]$ and $ (x,y,\eta)\in \sR^n \t \sR^n \t \cP_2(\sR^n \times \sR^{n})$,
$|\hat{\a}(t,x,y,\eta)-\hat{\a}(t',x,y,\eta)|\le L_{\a}(1+|x|+|y|+\|\eta\|_2)|t-t'|^{1/2}$.
\end{enumerate}
 
\end{Assumption}

\begin{Remark}
Equation \eqref{eq:opti_markov} implies that the    function $\hat{\a}$ attains 
the optimality condition  \eqref{eq:opti}  pointwise.
In this case,
  the  optimal control $\hat{\a}$ can be expressed as $\hat{\a}_t=\hat{\a}(t,X^{\hat{\a}}_t,Y^{\hat{\a}}_t,\sP_{(X^{\hat{\a}}_t,Y^{\hat{\a}}_t)})$, $t\in [0,T]$,
  and 
 the joint law of $(X^{\hat{\a}}_t,\hat{\a}_t)$ is given by $\phi(t,\sP_{(X^{\hat{\a}}_t,Y^{\hat{\a}}_t)})$. To see it,  observe that  
 for any $t\in [0,T]$, $X,Y \in L^2(\Om;\sR^n)$
and any Borel   set $A \subset \sR^n \t \sR^k$, 
\begin{align}\l{eq:PushF}
\begin{split}
\phi(t,\sP_{(X,Y)})(A) &= \sP_{(X,Y)} \left( (\text{id}_{\sR^{n}},\hat{\a}(t,\cdot,\cdot,\sP_{(X,Y)}))^{-1}(A) \right) \\
& = \sP \left( (X,Y) \in (\text{id}_{\sR^{n}},\hat{\a}(t,\cdot,\cdot,\sP_{(X,Y)}))^{-1}(A) \right) \\
& = \sP \left( (X,\hat{\a}(t,X,Y,\sP_{(X,Y)})) \in  A \right) = \mathbb{P}_{(X,\hat{\a}(t,X,Y,\sP_{(X,Y)}))}(A).
\end{split}
\end{align}

Constructing a function $\hat\alpha$   satisfying (H.\ref{assum:mfcE_hat})  for general coefficients appears challenging. The main difficulty stems from the fact that
due to
the nonlinear dependence on the law of the control, 
the function 
$\phi$ in \eqref{eq:opti_markov} 
  depends implicitly on $\hat \alpha $.  
 Hence  Condition \eqref{eq:opti_markov} cannot be simplified to a pointwise minimization of the Hamiltonian  $H^{\rm re}$.  In fact,  although
similar assumptions have been  adopted in  
\cite[Theorem 2.4]{gobet2019} and \cite[Assumption (M)]{lauriere2020}  
to study  extended MFC problems, 
 no explicit coefficient conditions were provided therein  to guarantee the existence of $\hat \alpha$.
  
\end{Remark}

Here we verify (H.\ref{assum:mfcE_hat}) in special cases, with  additional conditions on the running cost $f$.
These conditions only involve the first-order derivatives of $f$.
 The proofs are given in Section \ref{sec:proof_feedback_representation}.

The first example considers the case where the running cost  is independent of   the law of the controls, although the drift coefficient may depend on the expectation of the control.
This includes the classical MFC problem
as a special case,
for which both the   state  dynamics is also independent of
 the law of the controls
(see  e.g., \cite{carmona2018a}).

\begin{Proposition}\l{prop:mfc_hat_a}
Suppose (H.\ref{assum:mfcE}) holds,
and 
for all $(t,x,a)\in [0,T]\t\sR^n\t \sR^k$,
  $\cP_2(\sR^n\t \sR^k)\ni \eta\mapsto 
f(t,x,a,\eta)\in \sR$
 depends only on  
 the first  marginal $\pi_1\sharp \eta$ of   $\eta$.
Then 
 (H.\ref{assum:mfcE_hat}(\ref{item:ex})) holds.

Assume further that
there exists  $\tilde{K}\ge 0$ such that 
 for all $t,t'\in [0,T]$, $(x,a,\eta)\in \sR^n\t  \bA\t \cP_2(\sR^n\t \sR^n)$,  
$|b_2(t)-b_2(t')|+|b_3(t)-b_3(t')|\le \tilde{K}|t-t'|^{1/2}$ and $|\p_a f(t,x,a,\eta)-\p_a f(t',x,a,\eta)|\le \tilde{K}(1+|x|+|a|+\|\eta\|_2)|t-t'|^{1/2}$.
Then  
 (H.\ref{assum:mfcE_hat}) holds.

\end{Proposition}

The second example considers the case where   the dependence of   $f$ 
on $(X,\a)$ is separable, and 
the state dynamics \eqref{eq:mfcE_fwd}
depends on the  control process 
only through its expectation.
In this case, we show that the function $\hat \alpha$ in 
(H.\ref{assum:mfcE_hat})  can be chosen to depend solely on the law of the state and adjoint processes,
resulting a \emph{deterministic}  optimal control.

 \begin{Proposition}\l{prop:control_measure}
Suppose  (H.\ref{assum:mfcE}) holds,
\eqref{eq:b_affine} holds with $b_2\equiv 0$,
and
the function $f$ 
 is of the form 
$$f(t,x,a,\eta)=f_1(t,x,\pi_1\sharp \eta,\pi_2\sharp \eta)
+f_2(t,a,\pi_1\sharp \eta,\pi_2\sharp \eta),
$$
where 
 $f_1:[0,T]\t \sR^n\t \cP_2(\sR^n)\t \cP_2(\sR^k)\to \sR$
and $f_2:[0,T]\t \sR^k\t \cP_2(\sR^n)\t \cP_2(\sR^k)\to \sR$
satisfy (H.\ref{assum:mfcE}(\ref{item:mfcE_growth})),
and $\pi_1\sharp\eta$ (resp.~$\pi_2\sharp \eta$) is 
 the first (resp.~second) marginal of   $\eta$.
Then (H.\ref{assum:mfcE_hat}(\ref{item:ex})) holds with  a function $\hat{\a}:[0,T]\t \cP_2(\sR^n\t \sR^n)\to\bA$. 

Assume further that
there exists   $\tilde{K}\ge 0 $ such that 
 for all $t,t'\in [0,T]$ and  $(x,a,\mu)\in \sR^n\t  \bA\t \cP_2(\sR^n)$,  
$|b_3(t)-b_3(t')|\le \tilde{K}|t-t'|^{1/2}$ 
and 
\begin{align*}
&
|\p_\nu f_1(t,x,\mu,\bm{\delta}_{{a}})({a})
-\p_\nu f_1(t',x,\mu,\bm{\delta}_{{a}})({a})|
+
|\p_a f_2(t,a,\mu,\bm{\delta}_{{a}})
-\p_a f_2(t',a,\mu,\bm{\delta}_{{a}})|
\\
&\q+|\p_\nu f_2(t,a,\mu,\bm{\delta}_{{a}})({a})
-
\p_\nu f_2(t',a,\mu,\bm{\delta}_{{a}})({a})|
\le \tilde{K}(1+ |x| + |a| + \|\mu\|_2)|t-t'|^{1/2}.
\end{align*}
Then (H.\ref{assum:mfcE_hat}) holds with  a function $\hat{\a}:[0,T]\t \cP_2(\sR^n\t \sR^n)\to\bA$.
\end{Proposition}

The last example considers the setting where the running cost $f$ is 
   quadratic in the control variable, but it  can have a generic convex dependence on  the state.
 It includes   the commonly studied linear--quadratic MFC problems as   special cases.
To simplify the notation,  
 we consider a one-dimensional setting with  $n=k=d=1$,  $\bA=\sR$ 
 and 
\begin{align}\l{eq:f_quadratic}
&f(t,x,a,\eta) = \frac{1}{2} \left( f_1(t,x,\pi_1\sharp \eta) +q(t)a^2+
\bar{q}(t)\big(a-r(t)\bar{a}\big)^2
+2 c(t)xa \right),
\end{align}
where
 $\bar{a} = \int a \mathrm{d} \eta(x,a)$, 
$ q,\bar{q}, r, c\in L^\infty(0,T;\sR)$,
 $q\ge \lambda_1>0$, $\bar{q}\geq 0$
and $f_1:[0,T]\t \sR\t \cP_2(\sR)\to \sR$ is a given  function. 
The analysis extends naturally to  a
multi-dimensional setting.

\begin{Proposition}
\label{prop:lq}
  Suppose  (H.\ref{assum:mfcE}) holds, $n=k=d=1$,  $\bA=\sR$,
  $b_3(t)=(\beta(t),\gamma(t))^\top$
  in \eqref{eq:b_affine},
  and   $f$ is of the form \eqref{eq:f_quadratic}.
  Then 
(H.\ref{assum:mfcE_hat}(\ref{item:ex})) holds
with the function $\hat \alpha$ given by
\begin{align}
\label{eq:hat_alpha_lq}
\hat{\a}(t,x,y,\eta) \coloneqq  \frac{- c(t)x-b_2(t)y   + \psi(t) \int_\sR x\,\d\eta(x,y) +( -\gamma(t)  + \zeta(t)) \int_\sR y\,\d\eta(x,y)  }
{q(t) + \bar{q}(t)},
\end{align}
where
$\psi$ and $\zeta$ are defined by  
\begin{align*}
& \psi(t) \coloneqq \frac{c(t) \bar{q}(t) r(t)(r(t) -2)}{q(t) + \bar{q}(t)\big(r(t)-1\big)^2},
\q
 \zeta(t)\coloneqq \frac{(b_2(t) + \gamma(t))\bar{q}(t) r(t)(r(t) -2)}{q(t) + \bar{q}(t)\big(r(t)-1\big)^2}.
\end{align*}
If  $b_2,\gamma,q,\bar{q},r,c$ are $1/2$-H\"{o}lder continuous,
then  $\hat \alpha$  in \eqref{eq:hat_alpha_lq} satisfies 
(H.\ref{assum:mfcE_hat}).

\end{Proposition}

\subsubsection{Regularity 
and discrete-time approximations of optimal controls}

Leveraging Lemma \ref{lemma:maximum_principle} and  (H.\ref{assum:mfcE_hat}), 
$\hat{\a} \in \cA$ is 
an optimal control of \eqref{eq:mfcE} if there exists  
 a tuple of processes $(X^{\hat{\a}},Y^{\hat{\a}},Z^{\hat{\a}},\hat{\a})\in \cS^2(\sR^n)\t\cS^2(\sR^n)\t \cH^2(\sR^{n\t d})\t \cA$ such that  for all $t\in [0,T]$,  
\begin{align}\l{eq:mfc_fbsde_hat2}
\begin{split}
\mathrm{d}X^{\hat{\a}}_t&=
b\big(t,X^{\hat{\a}}_t,\hat{\a}_t,\sP_{(X^{\hat{\a}}_t,\hat{\a}_t)}\big)\, \d t
+\sigma(t,X^{\hat{\a}}_t,\sP_{X^{\hat{\a}}_t})\, \d W_t,
\q t\in (0,T],
\\
\mathrm{d}Y^{\hat{\a}}_t&=-\big(\p_x H(t,X^{\hat{\a}}_t,\hat{\a}_t,\sP_{(X^{\hat{\a}}_t,\hat{\a}_t)},Y^{\hat{\a}}_t,Z^{\hat{\a}}_t)
\\
&\q
+\tilde{\sE}[\p_\mu H(t,\tilde{X}^{\hat{\a}}_t,\tilde{{\hat{\a}}},\sP_{(X^{\hat{\a}}_t,\hat{\a}_t)},\tilde{Y}^{\hat{\a}}_t,\tilde{Z}^{\hat{\a}}_t)(X^{\hat{\a}}_t,\hat{\a}_t)]\big)\,\d t
+Z^{\hat{\a}}_t\,\d W_t,
\q  t\in [0,T),
\\
X^{\hat{\a}}_0&=\xi_0,\q Y^{\hat{\a}}_T=\p_x g(X^{\hat{\a}}_T,\sP_{X^{\hat{\a}}_T})+\tilde{\sE}[\p_\mu g(\tilde{X}^{\hat{\a}}_T,\sP_{X^{\hat{\a}}_t})(X^{\hat{\a}}_T)],
\\
\hat{\a}_t&= \hat{\a}(t,
X^{\hat{\a}}_t,Y^{\hat{\a}}_t,\sP_{(X^{\hat{\a}}_t,Y^{\hat{\a}}_t)}), \q t\in [0,T],
\end{split}
\end{align}
where
$(\tilde{X}^{\hat{\a}},\tilde{Y}^{\hat{\a}}, \tilde{Z}^{\hat{\a}},\tilde{{\hat{\a}}})$
is an  independent copy
 of $({X}^{\hat{\a}},{Y}^{\hat{\a}}, {Z}^{\hat{\a}},{{\hat{\a}}})$ defined on a space 
$L^2(\tilde{\Om},\tilde{\cF},\tilde{\sP})$.

Based on the above characterization, the following theorem shows that \eqref{eq:mfcE} admits a unique optimal control in $\cA$ that is $1/2$-H\"older continuous, which is the highest path regularity one can expect for an optimal control process.

\begin{Theorem}\label{TH:ControlRegularity}
Suppose (H.\ref{assum:mfcE}) and (H.\ref{assum:mfcE_hat}(\ref{item:ex})) hold,
and let   $\xi_0\in L^p(\cF_0;\sR^n)$ for some $p\ge 2$.
Then
(\ref{eq:mfcE}) admits a unique optimal control $\hat{\a}=(\hat{\a}_t)_{t\in [0,T]}\in \cA$ satisfying 
$\|\hat{\a}\|_{\cS^p}\le C(1+\|\xi_0\|_{L^p})$.

Assume further  that
(H.\ref{assum:mfcE_hat}(\ref{item:HC})) holds. Then there exists a constant $C\ge 0$ such that   for all 
 $0\le s\le t\le T$,  
$ \sE\left[\sup_{s\le r\le t}|\hat{\a}_r-\hat{\a}_s|^p\right]^{1/p}
 \le
C
(1+\|\xi_0\|_{L^{p}})
|t-s|^{{1}/{2}}.
$ 
\end{Theorem}

The proof of Theorem \ref{TH:ControlRegularity} is given in Section \ref{sec:regularity_mfcE}.
We  shall  prove in Theorem 
\ref{TH:fbsde_Regularity} that   \eqref{eq:mfc_fbsde_hat2} admits a   
$1/2$-H\"{o}lder continuous solution in 
$ \cS^2(\sR^n)\t\cS^2(\sR^n)\t \cH^2(\sR^{n\t d})$,
which along with   the time regularity of $\hat \alpha $
in (H.\ref{assum:mfcE_hat}(\ref{item:HC})) yields
  the desired path regularity of the optimal control.

Using  Theorem \ref{TH:ControlRegularity},
we  analyze  the convergence rate of 
the discrete-time   problem \eqref{eq:mfcE_constant}
to the continuous-time problem \eqref{eq:mfcE}.
The following H\"older  regularity assumption of the coefficients will be needed
to  quantify 
the  time discretization error of the controlled dynamics and the running cost in \eqref{eq:mfcE_constant}.

\begin{Assumption}\l{assum:mfc_Holder_t}
Assume  the notation of (H.\ref{assum:mfcE}). 
The functions $b_0,b_1,b_2,b_3,\sigma_0,\sigma_1,\sigma_2$ 
in (H.\ref{assum:mfcE}(\ref{item:mfcE_lin}))
are $1/2$-H\"{o}lder continuous,
and 
there exists   $\hat{K}\ge 0$ such that 
 for all $t,t'\in [0,T]$, $(x,a,\eta)\in \sR^n\t  \bA\t \cP_2(\sR^n\t \sR^k)$, 
$|f(t,x,a,\eta)- f(t',x,a,\eta)|\le \hat{K}(1+|x|^2+|a|^2+\|\eta\|^2_2)|t-t'|^{1/2}$.

\end{Assumption}

 Now we present the main result of this section, which  
  proves that the value function 
 $V_\pi(\xi_0)$ in \eqref{eq:mfcE_constant}
converges to the value function $V(\xi_0)$
in \eqref{eq:mfcE}
 with  order $1/2$ as the stepsize $|\pi|$ tends to zero. 
The proof  
is given in Section \ref{sec:conv_PCPT}.

\begin{Theorem}\l{thm:PCPT_discrete_value}
Suppose (H.\ref{assum:mfcE}), (H.\ref{assum:mfcE_hat})
and (H.\ref{assum:mfc_Holder_t}) hold.
Then there exists  $C>0$  such that   for all 
$\xi_0\in L^2(\cF_0;\sR^n)$ and for every partition
$\pi$  of $[0,T]$ with  stepsize $|\pi|$,
$$
V_\pi(\xi_0)- V(\xi_0)\le C(1+\|\xi_0\|^2_{L^{2}})|\pi|^{1/2}.
$$
Assume further that  $\bA$ is   compact.
Then 
$$
|V_\pi(\xi_0)- V(\xi_0)|\le C(1+\|\xi_0\|^2_{L^{2}})|\pi|^{1/2}.
$$
\end{Theorem}

\begin{Remark}
\label{rmk:regularity_comparison}
   To the best of our knowledge, this is the first result on  a half-order convergence rate for the discrete-time approximation of general linear–convex extended MFC problems.

This result is new even for classical MFC problems  without mean field interactions in controls.
In this setting, \cite[Proposition 12]{carmona2019} established the same convergence rate, but     under stronger assumptions that restrict to closed loop controls and require higher order differentiability of the cost functions. 
   Specifically,   \cite{carmona2019} 
restricts the analysis to   controls of the form $\a_t=\phi(t,X_t)$ with $\phi\in C^{1,2}_b([0,T]\t\sR^n)$,
and assumes that 
the decoupling field of \eqref{eq:mfc_fbsde_hat2} and the function $\hat{\a}$ in (H.\ref{assum:mfcE_hat})
to be  twice differentiable   with uniformly Lipschitz continuous derivatives.
    These conditions   require the  cost functions $f$ and $g$ in \eqref{eq:mfcE} 
to be three--times differentiable in $(x,a,\mu)$ with bounded and Lipschitz continuous derivatives.
 In contrast, we optimize over general open-loop controls  
 and require only first-order differentiability of the cost functions.

\end{Remark}

\begin{Remark}
\l{rmk:non_compact}
The compactness of $\bA$ 
in Theorem \ref{thm:PCPT_discrete_value}
is used to  ensure  that 
the $\cH^2$-norms of 
discrete-time optimal  controls for \eqref{eq:mfcE_constant}  are   uniformly bounded  with respect to the partitions $\pi$.
Similar half-order convergence rates   can be established if one restricts \eqref{eq:mfcE_constant}
to controls whose  $\cH^2$-norms are uniformly bounded.  
\end{Remark}

 We further  show that an  
 optimal control  of the discrete-time control problem \eqref{eq:mfcE_constant}
 converges strongly to the optimal control of \eqref{eq:mfcE} with rate $\cO(|\pi|^{1/4} )$.
 The proof 
is given in Section \ref{sec:conv_PCPT}.

\begin{Theorem}\l{thm:control_PCPT_discrete}
Suppose (H.\ref{assum:mfcE}), (H.\ref{assum:mfcE_hat})
and (H.\ref{assum:mfc_Holder_t}) hold,
and $\bA$ is  compact.
Then 
 there exists a constant $C>0$  such that  
for all 
$\xi_0\in L^2(\cF_0;\sR^n)$, all partitions
$\pi$  of $[0,T]$ with  stepsize $|\pi|$,
all  $\eps\ge 0$,    
and  all $\a \in \cA_{\pi}$ with 
$J_\pi(\a;\xi_0)\le V_{\pi}(\xi_0)+\eps$, 
$$
\|\hat{\a}-{\a}\|_{\cH^2}\le C\big((1+\|\xi_0\|_{L^{2}})|\pi|^{1/4} + \sqrt{\eps}\big),
$$
where  $\hat{\a}\in\cA$ is the optimal control of \eqref{eq:mfcE}.  
\end{Theorem}

To the best of our knowledge, 
  Theorem \ref{thm:control_PCPT_discrete} is the first strong convergence result on discretization  errors  of optimal control processes.

\subsection{Error analysis for general  extended MFC problems}

In this section, we consider a more generic set-up where both the drift and  diffusion coefficients of the state dynamics   depend nonlinearly on the  joint law of the state and   control processes, but make certain stronger regularity   assumptions.
In this setting, we show that the piecewise constant control approximation converges with the improved order of 1 in the time step.
This is the 
\emph{maximum achievable rate} for such an approximation, and it matches the convergence rate observed in practice; see Section \ref{numericalresults}.

We impose the following standing assumption on the coefficients of \eqref{eq:mfcE_fwd} and \eqref{eq:mfcE}. 

\begin{Assumption}\l{assum:A1} 
Let $\bA\subset \sR^k$ be a nonempty   set,
 $(b,\sigma):[0,T]\t \sR^n\t \sR^k\t  \cP_2(\sR^n\t \sR^k)\to \sR^n\times \sR^{n\t d} $,
$f:[0,T]\t \sR^n\t \sR^k\t  \cP_2(\sR^n\t \sR^k)\to\sR$ 
and $g:\sR^n\t  \cP_2(\sR^n)\to\sR$ 
be continuous functions. There exists a constant $C> 0$ such that for all 
$  t \in [0,T], x, x' \in \mathbb{R}^n$, $\mu , \mu' \in \mathcal{P}_2(\mathbb{R}^n \times \mathbb{R}^k)$, $a,a' \in \textbf{A}$ and $\nu\in \cP_2(\sR^n)$,
\begin{equation*}
\begin{split}
&|b(t,x,a,\mu) - b(t,x',a',\mu')| 
\le C\left(|x-x'|+ |a-a'| + \mathcal{W}_2\left(\mu,\mu'\right)\right),\\
& |\sigma(t,x,a,\mu) - \sigma(t,x',a',\mu')| \le C \left(|x-x'| + |a-a'| + \mathcal{W}_2\left(\mu,\mu'\right)\right),\\
& \int_0^T (|b(t,0,0,\bm{\delta}_{0_{n+k}})|^2 + |\sigma(t,0,0,\bm{\delta}_{0_{n+k}})|^2) \diff t \le C,\\
&|f(t,x,a,\mu)|\le C \left(1 + |x|^2 + |a|^2 + \|\mu\|_2^2\right),
\q
 |g(x,\nu)|\le C \left(1 + |x|^2 + \|\nu\|_2^2\right).
\end{split}
\end{equation*}
\end{Assumption}

Condition (H.\ref{assum:A1})
is the standard assumption for studying  (extended) MFC problems (see e.g., \cite{pham2018}).
It ensures 
that  for any  $\alpha\in \cA$, \eqref{eq:mfcE_fwd} has a unique solution $X^{\alpha}\in \cS^2(\sR^n)$, and the control problem  \eqref{eq:mfcE} is well-defined. 

We now introduce a dynamic version of 
\eqref{eq:mfcE}, which will be used to 
  analyze the control discretization error. 
For each $(t,\mu)\in [0,T] \times \mathcal{P}_2(\mathbb{R}^n)$,
let 
$\xi\in L^2(\cF_t; \sR^n)$
be such that $\cL_{\xi}=\mu$,
and 
consider  the following   control problem 
\bb\l{eq:mfcE2}
V(t,\mu)\coloneqq \inf_{\a\in \cA} J(\alpha; t,\xi)
\q
\textnormal{with}
\;
J(\alpha; t,\xi)\coloneqq \sE\bigg[
\int_t^T f(s,X^\a_s,\a_s,\sP_{(X^\a_s,\a_s)})\, \d s+g(X^\a_T,\sP_{X^\a_T})
\bigg],
\ee
where $X^\a$ is governed by 
\begin{equation}\label{stateprocessX}
 \d X^\a_s=
b(s,X^\a_s,\a_s,\sP_{(X^\a_s,\a_s)})\, \d s
+\sigma(s,X^\a_s,\a_s, \sP_{(X^\a_s,\a_s)})\, \d W_s,
\q s\in [t,T];
\q X^\a_t=\xi.
 \end{equation}
Note that by \cite[Proposition 2.4]{djete2022mckean}
and the continuity of $f$ and $g$, 
 $J(\alpha; t,\xi)$ satisfies the law invariance property, i.e., 
 it depends on the law of $
 \xi$  instead of the specific choice of
the random variable  itself.
Thus the   function $V$
  in 
\eqref{eq:mfcE2} can be identified as a function on $[0,T]\times \cP_2(\sR^n)$.
   
In the sequel, we analyze the piecewise constant control approximation of \eqref{eq:mfcE2},
and restrict  attention, for simplicity, to uniform partitions.
Let $\pi=\{t=t_0<\cdots<t_N \}$ 
be a uniform partition of $[t,T]$ with stepsize $h$ and $N=\lceil T/h \rceil $,  
 we 
 consider the following 
 approximate control problem:
\bb\label{valuefunction_h}
    V_\pi^c(t,\mu) \coloneqq \inf_{\alpha \in \cA_h} J( 
    \alpha; t,\xi),
\ee
where 
$J( \alpha; t,\xi)$ is defined in \eqref{eq:mfcE2},
    and $\cA_{h}$ is the  set of  piecewise constant controls  on $\pi$ defined as in \eqref{eq:A_pi}. 
  For clarity of exposition, we retain the continuous-time state dynamics \eqref{stateprocessX} in \eqref{valuefunction_h}. Similar results and analyzes naturally extend to the setting where the state dynamics are also discretized in time   as in  
\eqref{eq:mfcE_fwd_euler}.

To derive the improved convergence rate
of 
\eqref{valuefunction_h}, we extend techniques developed in  \cite[Theorem 2.1]{jakobsen2019} for   stochastic optimal control problems without mean field interactions. 
This requires performing local expansions of the value function over a small time interval using It\^{o}'s formula.
Compared with 
\cite{jakobsen2019},
the presence of the law of the state and control processes in the state dynamics, cost functional, and value function makes the analysis more involved, as it requires applying It\^o's formula along a flow of probability measures.
This involves  differentiating functions defined on the space of probability measures, for which 
we employ the notion of the $L$--derivative. We review its definition and the related function spaces $C^{1,2}$ and $C^{1,2}_2$  in Appendix 
\ref{sec:ito_formula}.


Before presenting the regularity conditions that facilitate the analysis, we define two  operators that arise from applying It\^{o}'s formula over  time intervals on which the control process remains constant.
For each $u
\in C^{1,2}_2([0,T]\times \mathcal{P}_2(\mathbb{R}^n))$ and 
$\nu \in \mathcal{P}_2 (\mathbb{R}^n \times \mathbb{R}^k )$, let 
$\mu = \pi_1\sharp \nu$ the first marginal of $\nu$,
and define 
\begin{equation}
\l{defLx}
\begin{split}
  (\Lx u)(t,\nu)  &\coloneqq \partial_t u(t,\mu)+ \int_{\mathbb{R}^n \times \mathbb{R}^k} \Big(
    \langle b(t, x, a, \nu), \partial_{\mu} u(t,\mu)(x) \rangle \\
  & \quad + \frac{1}{2} \operatorname{Tr}\left[
    (\sigma \sigma^{\intercal})(t, x, a, \nu)
    \partial_x \partial_{\mu} u(t,\mu)(x)
    \right]
  \Big) \nu(\mathrm{d}x, \mathrm{d}a).
\end{split}
\end{equation}
For each 
$u
\in C^{1,2}_2([0,T]\times \mathcal{P}_2(\mathbb{R}^n
\times \mathbb{R}^k))$,
and $
\nu \in \mathcal{P}_2(\mathbb{R}^n
\times \mathbb{R}^k)$, define 
\begin{equation}
\l{defLxa}
\begin{split}
    (\Lxa u)(t,\nu)& \coloneqq  \partial_t u(t,\nu)+ \int_{\mathbb{R}^n \times \mathbb{R}^k} \bigg(
    \left\langle 
    \begin{pmatrix} b(t, x, a, \nu)
    \\
    0_k\end{pmatrix}, \partial_{\nu} u(t,\nu)(x,a) 
    \right\rangle \\
  & \quad + \frac{1}{2} \operatorname{Tr}\left[
   \begin{pmatrix} (\sigma \sigma^{\intercal})(t, x, a, \nu) & 0_{n    \times k}
    \\
    0_{k
    \times n} & 0_{k    \times k}\end{pmatrix}
    \partial_{(x,a)} \partial_{\nu} u(t,\nu)(x,a)
    \right]
  \bigg) \nu(\mathrm{d}x, \mathrm{d}a),
\end{split}
\end{equation}
where $0_{m\times n}$ denotes the $m\times n$ 
zero matrix.
Note that  on an time interval  where   the control $\alpha$ is constant,
the pair $ 
 (X^\alpha ,\alpha)$ can be treated   as  an It\^{o} process with a degenerate second component,  which corresponds to the zeros in \eqref{defLxa}.

We   now formulate the additional regularity assumptions required for the higher order analysis.
For brevity, let $F:[0,T] \times \mathcal{P}_2(\mathbb{R}^n \times \mathbb{R}^k) \to \mathbb{R}$, be $F(t,\nu)\coloneqq  \int_{\mathbb{R}^n \times \mathbb{R}^k} f(t,x,a, \nu)\nu(\d x, \d a)$.

\begin{Assumption}\l{assum:B1}
\l{assum:B2}
 There exists $C>0$ such that
 \begin{enumerate}[(1)]
\item 
\l{enum_i}
$F \in C_2^{1,2}([0,T] \times \mathcal{P}_2(\mathbb{R}^n \times \mathbb{R}^k))$ and for all $(t,\nu) \in [0,T] \times \mathcal{P}_2(\mathbb{R}^n \times \mathbb{R}^k)$,
$|(\Lxa F) (t, \nu)| \le C $.
\item 
\l{enum_ii}
$V_\pi^c \in C_2^{1,2}([0,T] \times \mathcal{P}_2(\mathbb{R}^n))$,  
$[0,T] \times \mathcal{P}_2(\mathbb{R}^n \times \mathbb{R}^k) \ni (t,\nu) \mapsto (\Lx V^c_{\pi})(t,\nu)$ is in $ C_2^{1,2}([0,T] \times \mathcal{P}_2(\mathbb{R}^n \times \mathbb{R}^k))$ and for all $(t,\nu) \in [0,T]
\times \mathcal{P}_2(\mathbb{R}^n \times \mathbb{R}^k)$,
 $|(\Lxa \Lx V^c_{\pi})(t,\nu)| \le C.$ 
 \end{enumerate}
\end{Assumption}

\begin{Remark}
    Assumption (H.\ref{assum:B1}) is formulated in terms of the operators $\Lx$ and $\Lxa$ for compactness of notation.
    Sufficient conditions for (H.\ref{assum:B1}(\ref{enum_i})) can be given in terms of    bounds of $\partial_t F$, $\partial_{\nu} F$, $\partial_{(x,a)} \partial_{\nu} F$, $\sigma$ and $b$. For  (H.\ref{assum:B1}(\ref{enum_ii})), application of product differentiation rules to terms in $\Lxa \Lx V^c_{\pi}$ leads to a lengthy expansion, which can be bounded as long as derivatives of $V^c_{\pi}$ and of the the data up to a sufficiently high order are bounded. A similar requirement on bounded derivatives is already present in \cite[Proposition 2.4]{jakobsen2019} (see also the proof of Theorem 2.1 there) without mean-field interactions, and here extends to derivatives with respect to the measure argument.
    These boundedness assumptions on the data are restrictive and do not include, for instance, the linear-convex setting analyzed in Section 
    \ref{sec:error_linear_convex}. Also, to the best of  our knowledge, the existence of derivatives of the value function $V^c_{\pi}$ with piecewise constant controls has not established, even in the case without mean field interaction.
\end{Remark}

 The following theorem establishes the optimal first-order convergence of the piecewise constant control approximation 
 \eqref{valuefunction_h} to \eqref{eq:mfcE}. 
 The proof is given in Section \ref{MFC1_general}.
We confirm this first-order convergence through numerical examples in Section \ref{numericalresults}.
  
  \begin{Theorem}\label{improvedorder}
 Suppose   (H.\ref{assum:A1}) and (H.\ref{assum:B2}) hold. 
 There exists a constant $C\ge 0$
 such that 
 for all $s\in [0,T], \mu_s \in \mathcal{P}_2(\mathbb{R}^n),$ and $h>0$, 
\begin{equation}\label{mainresult1}
    0 \le V_\pi^c(s,\mu_s)-V(s,\mu_s) \le Ch.
\end{equation}
\end{Theorem}

\section{Proofs of main results}
\label{sec:proofs}
\subsection{Proofs of Propositions \ref{prop:mfc_hat_a},
\ref{prop:control_measure}
and \ref{prop:lq}}
\label{sec:proof_feedback_representation}

\begin{proof}[Proof of Proposition \ref{prop:mfc_hat_a}]
Observe that   
the reduced Hamiltonian ${H}^{\textrm{re}}$  satisfies
for all 
$ (t,x, a,\eta,y)\in [0,T]\t\sR^n\t \sR^k\t \cP_2(\sR^n\t \sR^k)\t \sR^n$,
\begin{align*}
{H}^{\textrm{re}}(t,x, a,\eta,y)
&= \la {b}(t,x,a,\eta), y\ra +\tilde{f}(t,x,a,\pi_1\sharp\eta)
\\
&= \psi_1(t,x,a,y)+\psi_2(t,\eta,y) +\tilde{f}(t,x,a,\pi_1\sharp\eta),
\end{align*}
where 
$\psi_1(t,a,x,y)\coloneqq \la b_0(t)+b_1(t)x + b_2(t)a, y\ra$,
$ \psi_2(t,\eta,y)\coloneqq \la b_3(t)\bar{\eta},y\ra $
and 
 $\tilde{f}(t,x,a,\mu)\coloneqq {f}(t,x,a,\mu\t \bm{\delta}_{0_k})$.
Moreover,  
$\p_a {H}^{\textrm{re}}(t,x, a,\eta,y)=b_2^*(t)y+\p_a\tilde{f}(t,x,a,\pi_1\sharp\eta)$
and $\p_\nu {H}^{\textrm{re}}(t,x, a,\eta,y)(\cdot)=\p_\nu\psi_2(t,\eta,y) (\cdot)=\b(t)y$, where 
 $\b(t)\in \sR^{k\t n}$ is the  submatrix  formed by deleting the first $n$ rows of $b^*_3(t)\in \sR^{(n+k)\t n}$.
 
Define the function $G:[0,T]\t \sR^n\t \sR^k\t  \cP_2(\sR^n)\t \cP_2(\sR^n)\t \sR^n\to \sR $
such that for all 
$(t,x,a,\mu,\rho,y)\in [0,T]\t \sR^n\t \sR^k\t  \cP_2(\sR^n)\t\cP_2(\sR^k)\t \sR^n $, 
\begin{align*}
G(t,x,a,\mu,\rho,y)\coloneqq 
 \psi_1(t,x,a,y)+\psi_2(t,\mu\times \bm{\delta}_a,\bar{\rho}) +\tilde{f}(t,x,a,\mu)
\end{align*}
with $\bar{\rho}=\int_{\sR^n}y\,\d \rho(y)$.
 We further  define the map 
$ \hat{\a}:[0,T]\t \sR^n\t  \sR^n\t \cP_2(\sR^n\t \sR^n)\to \bA$
satisfying for all $(t,x,y,\chi)\in [0,T]\t \sR^n\t  \sR^n\t \cP_2(\sR^n\t \sR^n)$ that
\bb\l{eq:mfc_a_hat}
\hat{\a}(t,x,y,\chi)= \argmin_{\a\in \bA}G(t,x,a,\pi_1\sharp \chi,\pi_2\sharp \chi,y).
\ee
Since the function $f$   depends only on  
 the first  marginal $\pi_1\sharp \eta$,
 we see from (H.\ref{assum:mfcE}(\ref{item:mfcE_convex}))
 that
 $\lambda_1>0$ and 
  the map
  $\bA\ni a\mapsto  \tilde{f}(t,x,a,\mu)\in \sR$ is $\lambda_1$-strongly convex,
  which along with the linearity of the map
$\bA\ni a\mapsto  \psi_1(t,x,a,y)+\psi_2(t,\mu\times \bm{\delta}_a,\bar{\rho})\in \sR$
shows $\bA\ni a\mapsto G(t,x,a,\pi_1\sharp \chi,\pi_2\sharp \chi,y)\in \sR$ is $\lambda_1$-strongly convex.
Then by following the same argument as in  \cite[Lemma 3.3]{carmona2018a},
one  can show
the above function $\hat{\a}$ is 
well-defined, measurable,
 locally bounded and Lipschitz continuous with respect to $(x,y,\chi)$ uniformly in $t$.

To show  that    $\hat{\a}$ satisfies  (H.\ref{assum:mfcE_hat}(\ref{item:ex})),
  it remains to verify  \eqref{eq:opti_markov}.
Since  $\hat{\a}$ is a minimizer of $G$ over $\bA$,
for all $(t,x,y,\chi)\in [0,T]\t \sR^n\t  \sR^n\t \cP_2(\sR^n\t \sR^n)$,
$a\in \bA$ that 
\begin{align}\l{eq:hat_a_optimal_G}
\begin{split}
0&
\ge \la \p_a G(t,x,\hat{\a}(t,x,y,\chi),\pi_1\sharp \chi,\pi_2\sharp \chi,y),
\hat{\a}(t,x,y,\chi)-a\ra 
\\
&
= \la b_2^*(t)y +\b(t)\ol{\pi_2\sharp \chi}+\p_a\tilde{f}(t,x,\hat{\a}(t,x,y,\chi),\pi_1\sharp\chi),
\hat{\a}(t,x,y,\chi)-a\ra,
\end{split}
\end{align}
where $\ol{\pi_2\sharp \chi}=\int_{\sR^n} y\,\d {\pi_2\sharp \chi}(y)$.
Since $f(t,x,a,\eta)$ depends on $\eta$  only via its   first marginal,  
$\p_a\tilde{f}(t,x,a,\pi_1\sharp\chi)=\p_a\tilde{f}(t,x,a,\pi_1\sharp\phi(t,\chi))$, with  $\phi$ defined by \eqref{eq:opti_markov}.
Hence, by the expression of $\p_a{H}^{\textrm{re}}$, 
\begin{align*}
0&
\ge \la b_2^*(t)y +\b(t)\ol{\pi_2\sharp \chi}+\p_a\tilde{f}(t,x,\hat{\a}(t,x,y,\chi),\pi_1\sharp\phi(t,\chi)),
\hat{\a}(t,x,y,\chi)-a\ra
\\
&= \la \p_a {H}^{\textrm{re}}(t,x, \hat{\a}(t,x,y,\chi),\phi(t,\chi),y)
+\b(t)\ol{\pi_2\sharp \chi},
\hat{\a}(t,x,y,\chi)-a\ra,
\end{align*}
which is the optimality condition \eqref{eq:opti_markov} 
since for all $ (t,a,\eta)$,
$$
\int_{\sR^n\t \sR^n}\p_\nu {H}^{\textrm{re}}(t,\tilde{x}, a,\eta,\tilde{y})(\cdot)\,\d \chi(\tilde{x},\tilde{y})=\int_{\sR^n\t \sR^n} \b(t)\tilde{y}\,\d \chi(\tilde{x},\tilde{y})=\b(t)\ol{\pi_2\sharp \chi}.
$$

 We now prove the time regularity of $\hat{\a}$ using  the   H\"{o}lder regularity of $b_2, b_3$ and $\p_a f$.
Let $t,t'\in [0,T], (x,y,\chi)\in \sR^n\t \sR^n\t \cP_2(\sR^n\t \sR^n)$,
 $\hat{a}=\hat{\a}(t,x,y,\chi)$ and $\hat{a}'=\hat{\a}(t',x,y,\chi)$. 
 By 
the optimal condition \eqref{eq:hat_a_optimal_G}, 
$\la  \p_a G(t,x,\hat{a},\pi_1\sharp \chi,\pi_2\sharp \chi,y)
,\hat{a}'-\hat{a}\ra\ge 0
 \ge 
\la  \p_a G(t',x,\hat{a}',\pi_1\sharp \chi,\pi_2\sharp \chi,y),\hat{a}'-\hat{a}\ra
$.
Moreover, by the
$\lambda_1$-strong convexity of 
 $\bA\ni a\mapsto G(t,x,a,\pi_1\sharp \chi,\pi_2\sharp \chi,y)\in \sR$,  
\begin{align*}
&G(t,x,\hat{a}',\pi_1\sharp \chi,\pi_2\sharp \chi,y)-G(t,x,\hat{a},\pi_1\sharp \chi,\pi_2\sharp \chi,y)
\\
&\q
-\la \p_{a}G(t,x,\hat{a},\pi_1\sharp \chi,\pi_2\sharp \chi,y), \hat{a}'-\hat{a}\ra \ge \lambda_1|\hat{a}'-\hat{a}|^2,
\end{align*}
from which, by exchanging the role of $ \hat{a}'$ and $ \hat{a}$ 
in the above inequality and summing the resulting estimates, one  can deduce that  
\begin{align*}
2\lambda_1|\hat{a}'-\hat{a}|^2
&\le
\la\hat{a}'-\hat{a}, \p_{a}G(t,x,\hat{a}',\pi_1\sharp \chi,\pi_2\sharp \chi,y)-\p_{a}G(t,x,\hat{a},\pi_1\sharp \chi,\pi_2\sharp \chi,y)\ra
\\
&\le
\la\hat{a}'-\hat{a}, \p_{a}G(t,x,\hat{a}',\pi_1\sharp \chi,\pi_2\sharp \chi,y)-\p_{a}G(t',x,\hat{a}',\pi_1\sharp \chi,\pi_2\sharp \chi,y)\ra.
\end{align*}
Hence, by the expression of $\p_a G$, 
\begin{align*}
|\hat{a}'-\hat{a}|
&\le
C |\p_{a}G(t,x,\hat{a}',\pi_1\sharp \chi,\pi_2\sharp \chi,y)-\p_{a}G(t,x,\hat{a}',\pi_1\sharp \chi,\pi_2\sharp \chi,y)|
\\
&\le C
\Big(|
\p_a\tilde{f}(t,x,\hat{\a}',\pi_1\sharp\chi)-\p_a\tilde{f}(t',x,\hat{\a}',\pi_1\sharp\chi)|
+|b_2^*(t)-b_2^*(t')||y|
\\
&\q +|\b(t)-\b(t')|\|\chi\|_2\Big),
\end{align*}
for a constant $C$ independent of $(t,t',x,y,\chi)$.
Then, by applying the H\"{o}lder regularity assumption of the coefficients,  
\begin{align*}
|\hat{a}'-\hat{a}|
&\le
C(1+|x|+\|\chi\|_2+|\hat{a}'|+|y|)|t-t'|^{1/2},
\end{align*}
which, together with the fact that the function $\hat{a}$ is locally bounded and  of linear growth in $(x,y,\chi)$, 
leads to the desired H\"{o}lder continuity of $\hat{\a}$.
\end{proof}

\begin{proof}[Proof of Proposition \ref{prop:control_measure}]
Consider the function
$\hat{\a}:[0,T]\t  \cP_2(\sR^{n} \t \sR^{n}) \to \bA$ satisfying for all 
$(t,\chi)\in [0,T]\t  \cP_2(\sR^n\t \sR^n)$ that 
$$
\hat{\a}(t,\chi)=\argmin_{a\in \bA} h(t,\chi,a),
\q
\textnormal{
with
$h(t,\chi,a)\coloneqq \tilde{\sE}[H^{\textrm{re}}(t,\tilde{X},a,\tilde{\sP}_{(\tilde{X},a)},\tilde{Y})]$,
 }
$$
where  $(\tilde{X},\tilde{Y})\in L^2(\tilde{\Om},\tilde{\cF},\tilde{\sP};\sR^n\t \sR^n)$
has distribution $\chi$.

We first show the function $\hat{\a}$ is well-defined.
By   the linearity of $b$ and the convexity of $f$ in (H.\ref{assum:mfcE}),  
the map
$\bA\ni a\mapsto h(t,\chi,a)\in \sR$ is strongly convex with factor $\lambda_1+\lambda_2>0$,
which admits a unique minimizer on the nonempty closed convex set $\bA$.
The measurability of $\hat{\a}$ follows from \cite[Lemma 3.3]{carmona2018a}.

Then, we prove that the function $\hat{\a}$ satisfies 
the optimality condition \eqref{eq:opti_markov}.
By using (H.\ref{assum:mfcE}),  for almost all $(t,\om)\in [0,T]\t \tilde{\Om}$, 
the mapping $\sR^n\ni a\mapsto  H^{\textrm{re}}(t,\tilde{X}(\om),a,\tilde{\sP}_{(\tilde{X},a)},\tilde{Y}(\om))$
is differentiable with the derivative being   at most of linear growth in $(\tilde{X}(\om),\tilde{Y}(\om))$.
Hence, Lebesgue's differentiation theorem shows that 
$h$ is differentiable  with respect to $a$ with 
the derivative 
\begin{align}\l{eq:p_ah}
\p_a h(t,\chi,a)=
 \tilde{\sE}[\p_aH^{\textrm{re}}(t,\tilde{X},a,\tilde{\sP}_{(\tilde{X},a)},\tilde{Y})]
 + \tilde{\sE}\big[
 \bar{\sE}[
 \p_\nu H^{\textrm{re}}(t,\tilde{X},a,\tilde{\sP}_{(\tilde{X},a)},\tilde{Y})(\bar{X},a)
 ]\big],
\end{align}
where $\bar{X}\in L^2(\bar{\Om},\bar{\cF},\bar{\sP};\sR^n)$ has distribution $\tilde{\sP}_{\tilde{X}}$.

Using  $b_2\equiv 0$ and 
the   condition on  $f$, 
$H^{\textrm{re}}(t,x,a,\eta,y)=\la b_0(t)+b_1(t)x +b_3(t)\bar{\eta}, y\ra +f_1(t,x,\pi_1\sharp \eta,\pi_2\sharp \eta)
+f_2(t,a,\pi_1\sharp \eta,\pi_2\sharp \eta)$.
Hence, 
for all $(t,x,y,a)$,
$\p_aH^{\textrm{re}}(t,x,a,\tilde{\sP}_{(\tilde{X},a)},y)
=\p_a f_2(t,a,\tilde{\sP}_{\tilde{X}},\bm{\delta}_a)$,
and 
$v\mapsto \p_\nu H^{\textrm{re}}(t,x,a,\eta,y)(v)$ can be chosen as a function
defined only on $\sR^k$ (not on $\sR^n\t \sR^k$ as in the general setting). This 
  simplifies \eqref{eq:p_ah} into:
\bb\l{eq:p_ah_2}
\p_a h(t,\chi,a)=
\p_aH^{\textrm{re}}(t,x,a,\tilde{\sP}_{(\tilde{X},a)},y)
 + \tilde{\sE}\big[
 \p_\nu H^{\textrm{re}}(t,\tilde{X},a,\tilde{\sP}_{(\tilde{X},a)},\tilde{Y})(a)
\big]
\ee
 for all $(t,\chi,a)\in[0,T]\t  \cP_2(\sR^{n} \t \sR^{n}) \t \bA$. 
Consequently, 
by 
the fact that $\hat{\a}(t,\chi)$ is a minimizer
and  the identity that $\tilde{\sP}_{(\tilde{X},\hat{\a}(t,\chi))}=\phi(t,\chi)$
(see \eqref{eq:PushF}),
the function $\hat{\a}$ satisfies 
the optimality condition \eqref{eq:opti_markov}:
for all $(t,\chi,a)\in [0,T]\t  \cP_2(\sR^n\t \sR^n)\t \bA$,
\begin{align}\l{eq:mfcE_optimal_a}
\begin{split}
0
&\ge \la \p_a h(t,\chi,\hat{\a}(t,\chi)),\hat{\a}(t,\chi)-a\ra
\\
&=\la 
\p_aH^{\textrm{re}}(t,x,\hat{\a}(t,\chi),\phi(t,\chi),y)
\\
&\q 
+ \tilde{\sE}[\p_\nu H^{\textrm{re}}(t,\tilde{X},\hat{\a}(t,\chi),\phi(t,\chi),\tilde{Y})(\hat{\a}(t,\chi))]
 ,\hat{\a}(t,\chi)-a\ra,
\end{split}
\end{align}
whenever  $(\tilde{X},\tilde{Y})\in L^2(\tilde{\Om},\tilde{\cF},\tilde{\sP};\sR^n\t \sR^n)$
has  distribution $\chi$.
Note that
in the present setting
 \eqref{eq:opti_markov}
is independent of $(x,y)$
since $\p_aH^{\textrm{re}}(t,x,\hat{\a}(t,\chi),\phi(t,\chi),y)
=\p_a f_2(t,a,\tilde{\sP}_{\tilde{X}},\bm{\delta}_{\hat{\a}(t,\chi)})$.

Finally, we establish 
the spatial and time regularity of $\hat{\a}$.
Similar to \cite[Lemma 3.3]{carmona2018a},
for all $t\in[0,T]$, 
 the $(\lambda_1+\lambda_2)$-strong convexity of $a\mapsto h(t,\chi,a)$
 implies 
$|\hat{\a}(t,\bm{\delta}_{0_{n+n}})-a_0|\le 
(\lambda_1+\lambda_2)^{-1}
|\p_a h(t,\bm{\delta}_{0_{n+n}},a_0)|$,
where $a_0$ an arbitrary element in $\bA$.
Then by \eqref{eq:p_ah_2} and (H.\ref{assum:mfcE}(\ref{item:mfcE_growth})),  
$\|\hat{\a}(\cdot,\bm{\delta}_{0_{n+n}})\|_{L^\infty(0,T)}<\infty$. 
Now let $(t,\chi),(t',\chi')\in [0,T]\t \cP_2(\sR^n\t \sR^n)$,
$\hat{a}=\hat{\a}(t,\chi)$, $\hat{a}'=\hat{\a}(t',\chi')$
and $(\tilde{X},\tilde{Y}),(\tilde{X}',\tilde{Y}')\in L^2(\tilde{\Om},\tilde{\cF},\tilde{\sP};\sR^n\t \sR^n)$
have  distributions $\chi$ and $\chi'$, respectively.
By following 
a  similar argument as that for 
 Proposition \ref{prop:mfc_hat_a},
one can deduce from 
the  $(\lambda_1+\lambda_2)$-strong convexity of $a\mapsto h(t,\chi,a)$,
the expression of $\p_{a}h$ in \eqref{eq:p_ah_2} 
and the Lipschitz continuity of $(\p_a f_2,\p_\nu f_1,\p_\nu f_2)$
in (H.\ref{assum:mfcE}(\ref{item:mfcE_growth}))
 that 
\begin{align*}
&|\hat{a}'-\hat{a}|
\le
C|\p_{a}h(t,\chi,\hat{a}')-  \p_{a}h(t',\chi',\hat{a}')|
\\
&\le
C\big(|\p_a f_2(t,\hat{a}',\tilde{\sP}_{\tilde{X}},\bm{\delta}_{\hat{a}'})
-\p_a f_2(t',\hat{a}',\tilde{\sP}_{\tilde{X}'},\bm{\delta}_{\hat{a}'})|
\\
&\q
+|\tilde{\sE}\big[
 \p_\nu H^{\textrm{re}}(t,\tilde{X},\hat{a}',\tilde{\sP}_{(\tilde{X},\hat{a}')},\tilde{Y})(\hat{a}')
 \big]
 -\tilde{\sE}\big[
 \p_\nu H^{\textrm{re}}(t',\tilde{X}',\hat{a}',\tilde{\sP}_{(\tilde{X}',\hat{a}')},\tilde{Y}')(\hat{a}')
 \big]|
 \big)
 \\
&\le
C\Big(
\cW_2(\chi,\chi')+
|\p_a f_2(t,\hat{a}',\tilde{\sP}_{\tilde{X}},\bm{\delta}_{\hat{a}'})
-\p_a f_2(t',\hat{a}',\tilde{\sP}_{\tilde{X}},\bm{\delta}_{\hat{a}'})|
+|b_3(t)-b_3(t')|\tilde{\sE}[|\tilde{Y}|]
\\
&\q
\big|
\tilde{\sE}\big[
 \p_\nu f_1(t,\tilde{X},\tilde{\sP}_{\tilde{X}},\bm{\delta}_{\hat{a}'})(\hat{a}')
 \big]
 -\tilde{\sE}\big[
 \p_\nu f_1(t',\tilde{X},\tilde{\sP}_{\tilde{X}},\bm{\delta}_{\hat{a}'})(\hat{a}')
 \big]
 \big|
 \\
&\q +
 \big|
\p_\nu f_2(t,\hat{a}',\tilde{\sP}_{\tilde{X}},\bm{\delta}_{\hat{a}'})(\hat{a}')
-\p_\nu f_2(t',\hat{a}',\tilde{\sP}_{\tilde{X}},\bm{\delta}_{\hat{a}'})(\hat{a}')
\big|
 \Big),
\end{align*}
where the constant $C$ is independent of $t,t',\chi,\chi'$.
Setting $t'=t$ in  the above estimate gives     
$|\hat{\a}(t,\chi)-\hat{\a}(t,\chi')|\le C\cW_2(\chi,\chi')$,
which along with 
$\|\hat{\a}(\cdot,\bm{\delta}_{0_{n+n}})\|_{L^\infty(0,T)}<\infty$
 implies 
$|\hat{\a}(t,\chi)|\le C(1+\|\chi\|_2)$.
The desired time regularity of $\hat{\a}$ then follows from the additional assumptions on the time regularity of coefficients.
\end{proof}

\begin{proof}[Proof of Proposition \ref{prop:lq}]

Note that the drift coefficient of \eqref{eq:mfcE_fwd} reads as
\begin{align*}
b(t,x,a,\eta)&=b_0(t)+b_1(t)x + b_2(t)a +  \beta(t) \bar{x} + \gamma(t) \bar{a},
\end{align*}
where 
$\bar{x} =  \int x \, \mathrm{d}\eta(x,a)$.
The definition of the reduced Hamiltonian \eqref{eq:hamiltonian_re} and the openness of the set $\bA$ imply that 
it suffices to find
a function $\hat{\a}:[0,T]\t \sR\t \sR\t \cP_2(\sR\t\sR)\to \bA$
such that 
for all $t\in [0,T]$, $X_t,Y_t\in L^2(\Om;\sR)$,
  $\a_t=\hat{\a}(t,X_t,Y_t,\sP_{(X_t,Y_t)})$ satisfies 
\begin{align}\label{eq:opti_cond}
b_2(t)Y_t +\gamma(t) \sE[Y_t] + \big(q(t) + \bar{q}(t)\big)\a_t + \bar{q}(t) r(t)(r(t) -2)\mathbb{E}[\a_t] + c(t) X_t = 0.
\end{align}
Taking expectations on both sides of \eqref{eq:opti_cond} yields
\begin{align}\label{eq:ex_alpha}
\sE[\a_t] = \frac{-(b_2(t) + \gamma(t) ) \sE[Y_t] - c(t)\sE[X_t]}{q(t) + \bar{q}(t)\big(r(t)-1\big)^2},
\end{align}
which is well-defined since
 $q(t)\ge \lambda_1>0$
 and $\bar{q}(t)\ge 0$.
Substituting \eqref{eq:ex_alpha}   into  \eqref{eq:opti_cond} and using the definition \eqref{eq:hat_alpha_lq} of $\hat \alpha$ shows that  $\hat \alpha$
satisfies 
(H.\ref{assum:mfcE_hat}(\ref{item:ex})).
Condition (H.\ref{assum:mfcE_hat}(\ref{item:HC})) holds if  
$b_2,\gamma,q,\bar{q},r,c$ are $1/2$-H\"{o}lder continuous on $[0,T]$.
 \end{proof}

\subsection{Proof of Theorem \ref{TH:ControlRegularity}}
\label{sec:regularity_mfcE}

To prove Theorem \ref{TH:ControlRegularity}, we analyze the path regularity of  solutions to  
the MV-FBSDE
 \eqref{eq:mfc_fbsde_hat2}.
Note that \eqref{eq:mfc_fbsde_hat2} can be equivalently formulated 
as 
\begin{subequations}\label{eq:mfc_fbsde2}
\begin{alignat}{2}
\mathrm{d} X_t&=\hat{b}(t,X_t,Y_t,\sP_{(X_t,Y_t)})\,\d t +\sigma (t,X_t, \sP_{X_t})\, \d W_t, \q t\in (0,T];
\q
&& X_0=\xi_0,
\l{eq:mfc_fwd_fb2}
\\
\mathrm{d} Y_t&=-\hat{f}(t,X_t,Y_t,Z_t, \sP_{(X_t,Y_t,Z_t)})\,\d t+Z_t\,\d W_t,
 \q t\in [0,T);
\q && Y_T=\hat{g}(X_T,\sP_{X_T})
\l{eq:mfc_bwd_fb2}
\end{alignat}
\end{subequations}
with  coefficients defined by:  
for all $(t,x,y,z,a,\mu, \chi, \rho)\in [0,T]\t \sR^n\t \sR^n \t \sR^{n \times d} \t \bA \t \cP_2(\sR^n)\t \cP_2(\sR^n \t \sR^{n}) \t \cP_2(\sR^n \t \sR^{n} \t  \sR^{n \t d})$,
\begin{align}\l{eq:mfc_coefficients2}
\begin{split}
&\hat{b}(t,x,y,\chi) \coloneqq b(t,x,\hat{\a}(t,x,y,\chi),\phi(t,\chi)),
\\
&\hat{f}(t,x,y,z,\rho)
\\\q& \coloneqq \p_x H(t,x,\hat{\a}(t,x,y,\pi_{1,2} \sharp \rho),\phi(t,\pi_{1,2} \sharp \rho),y,z)
\\
&\q
 +\int_{E}
\p_\mu H(t,\tilde{x},\hat{\a}(t,\tilde{x},\tilde{y},\pi_{1,2} \sharp \rho),\phi(t,\pi_{1,2} \sharp \rho),\tilde{y},\tilde{z})
(x,\hat{\a}(t,{x},{y},\pi_{1,2} \sharp \rho))\,\d \rho(\tilde{x},\tilde{y},\tilde{z}),
\\
&\hat{g}(x,\mu) \coloneqq 
\p_x g(x,\mu)+\int_{\sR^n}\p_\mu g(\tilde{x},\mu)(x)\,\d \mu(\tilde{x}),
\end{split}
\end{align}
where 
$E\coloneqq \sR^n\t \sR^n\t \sR^{n\t d}$,
$\phi(t,\chi)$ is defined as in (H.\ref{assum:mfcE_hat}(\ref{item:ex})) and
$\pi_{1,2} \sharp\rho\coloneqq\rho(\cdot \t \sR^{n \t d})$ is the marginal of   $\rho$ on $\sR^n\t\sR^n$.
In the sequel, we  
prove  that \eqref{eq:mfc_fbsde2} admits a unique 
$1/2$-H\"{o}lder continuous solution in 
$ \cS^2(\sR^n)\t\cS^2(\sR^n)\t \cH^2(\sR^{n\t d})$.

We first  show  that  the coefficients of the MV-FBSDE \eqref{eq:mfc_fbsde2}
are  Lipschitz continuous with respect to the spatial variables,
and satisfy     a   monotonicity condition. 
{
Proofs of these    propositions are given in Appendix \ref{appendix:hat_a_existence}.}

\begin{Proposition}\label{prop:mcfE1}
Suppose (H.\ref{assum:mfcE}) and (H.\ref{assum:mfcE_hat}(\ref{item:ex})) hold,
and let the functions 
$(\hat{b},\hat{f},\hat{g})$ 
be defined as in \eqref{eq:mfc_coefficients2}.
Then there exists a constant $C\ge 0$
satisfying 
for all $t\in [0,T]$ that
 the functions 
$(\hat{b}(t,\cdot),\sigma(t,\cdot),\hat{f}(t,\cdot),\hat{g}(t,\cdot))$ 
are  $C$-Lipschitz continuous in all  variables
and satisfy the estimate
 $
 \|\hat{b}(\cdot,{0},{0},\bm{\delta}_{{0}_{n+n}})\|_{L^2(0,T)}
 +\|\sigma(\cdot,0,\bm{\delta}_{0_{n}})\|_{L^\infty(0,T)}
 +
 \|\hat{f}(\cdot,{0},{0}, {0},\bm{\delta}_{{0}_{n+n+nd}})\|_{L^\infty(0,T)}\le C$.

\end{Proposition}

\begin{Proposition}\label{prop:mcfE2}
Suppose (H.\ref{assum:mfcE}) and (H.\ref{assum:mfcE_hat}(\ref{item:ex})) hold,
and let the functions 
$(\hat{b},\hat{f},\hat{g})$ 
be defined as in \eqref{eq:mfc_coefficients2}.
Then the functions $(\hat{b},\sigma,\hat{f},\hat{g})$ 
satisfy 
for all 
$t\in [0,T]$, $i\in \{1,2\}$, 
 $\Theta_i\coloneqq (X_i,Y_i,Z_i)\in L^2(\Om; \sR^n\t \sR^m\t \sR^{m\t d})$
 that
$ \sE[\la \hat{g}(X_1,\sP_{X_1})-\hat{g}(X_2,\sP_{X_2}), X_1-X_2\ra]
\ge  0$ and
\begin{align}\l{eq:monotonicity}
\begin{split}
&\sE[\la \hat{b}(t,X_1,Y_1,\sP_{(X_1,Y_1)})-\hat{b}(t,X_2,Y_2,\sP_{(X_2,Y_2)}), Y_1-Y_2\ra]
 \\
&\quad 
+\sE[\la \sigma(t,X_1,\sP_{X_1})-\sigma(t,X_2,\sP_{X_2}),  Z_1-Z_2\ra]
\\
&\quad 
+ \sE[\la -\hat{f}(t,\Theta_1,\sP_{\Theta_1})+\hat{f}(t,\Theta_2,\sP_{\Theta_2}), X_1-X_2\ra] 
\\
&
 \le 
 -2(\lambda_1 + \lambda_2)
  \|\hat{\a}(t,X_1,Y_1,\sP_{(X_1,Y_1)})-\hat{\a}(t,X_2,Y_2,\sP_{(X_2,Y_2)})\|^2_{L^2},
\end{split}
\end{align}
with the constants $\lambda_1,\lambda_2$ in  (H.\ref{assum:mfcE}(\ref{item:mfcE_convex})).
\end{Proposition}

Based on Propositions \ref{prop:mcfE1} and \ref{prop:mcfE2},
we     establish the well-posedness  and stability   of \eqref{eq:mfc_fbsde2}.
The proof 
adapts the method of continuation in \cite{bensoussan2015, carmona2015, guo2023reinforcement}
 to the present setting,
 and 
 is given in Appendix \ref{appendix:hat_a_existence}.

\begin{Proposition}
    \label{prop:fbsde_wellposedness}
{
Suppose (H.\ref{assum:mfcE}) and (H.\ref{assum:mfcE_hat}(\ref{item:ex})) hold. Then, for all  
$t\in [0,T]$ and 
 $\xi\in L^2(\cF_t;\sR^n)$,
there exists a unique triple
$(X^{t,\xi},Y^{t,\xi},Z^{t,\xi}) \in  \cS^2(t,T;\sR^n) \t \cS^2(t,T;\sR^n) \t \cH^2(t,T;\sR^{n \t d})$ 
satisfying \eqref{eq:mfc_fbsde2} on $[t,T]$ with  
the initial condition $X^{t,\xi}_t=\xi$.
Moreover, there exists a constant $C>0$ such that 
it holds for all $t\in [0,T]$  and $\xi,\xi'\in L^2(\cF_t;\sR^n)$ that 
$\|Y^{t,\xi}_t-Y^{t,\xi'}_t\|_{L^2}\le C\|\xi-\xi'\|_{L^2}$,
and
$\|X^{t,\xi}\|_{\cS^2(t,T;\sR^n)}+\|Y^{t,\xi}\|_{ \cS^2(t,T;\sR^n)}+\|Z^{t,\xi}\|_{\cH^2(t,T;\sR^{n\t d})}\le C(1+\|\xi\|_{L^2})$.}
\end{Proposition}

{ 
 We now give our result concerning the H\"{o}lder regularity of the solutions to \eqref{eq:mfc_fbsde2}.}
\begin{Theorem}\label{TH:fbsde_Regularity}
{Suppose (H.\ref{assum:mfcE}) and (H.\ref{assum:mfcE_hat}(\ref{item:ex})) hold,
and let $(X,Y,Z)\in \cS^2(\sR^n)\t\cS^2(\sR^n)\t \cH^2(\sR^{n\t d})$ be the unique solution to \eqref{eq:mfc_fbsde2} with initial condition $X_0 = \xi_0\in L^2(\cF_0;\sR^n)$.
Moreover, 
for all  $p\ge 2$,
there exists
a constant $C>0$, depending only on $p$ and the data in 
(H.\ref{assum:mfcE}) and (H.\ref{assum:mfcE_hat}(\ref{item:ex})),
such that
$
\|X\|_{\cS^p}+\|Y\|_{\cS^p}+ \|Z\|_{\cS^p}
\le 
 C
\big(1+
\|\xi_0\|_{L^{p}}
\big)$
and 
$ \sE\left[\sup_{s\le r\le t}|X_r-X_s|^p\right]^{1/p}
 +\sE\left[\sup_{s\le r\le t}|Y_r-Y_s|^p\right]^{1/p}
 \le
C
(1+\|\xi_0\|_{L^{p}})
|t-s|^{{1}/{2}}
$ 
for all $0\le s\le t\le T$.}
\end{Theorem}
\begin{proof}
Let $\xi_0\in L^2(\cF_0;\sR^n)$ be a given initial condition and 
$(X,Y,Z)\in \cS^2(\sR^n)\t\cS^2(\sR^n)\t \cH^2(\sR^{n\t d})$
be the solution to \eqref{eq:mfc_fbsde2}.
By using the pathwise uniqueness and  the Lipschitz stability  of \eqref{eq:mfc_fbsde2}
in Proposition \ref{prop:fbsde_wellposedness},
we can follow the arguments in \cite[Proposition 5.7]{carmona2015}
and deduce that
there exists  a measurable function $v:[0,T]\t \sR^n\to \sR^n$ (depending on $\xi_0$) 
and a constant $C>0$ (independent of $\xi_0$)
such that 
$\sP(\fa t\in [0,T], Y_t=v(t,X_t))=1$
and it holds for all $t\in [0,T]$ and $x,x'\in \sR^n$ that $|v(t,x)-v(t,x')|\le C|x-x'|$ and $|v(t,0)|\le C(1+\|\xi_0\|_{L^2})$.

By substituting the  relation $Y_t=v(t,X_t)$ into \eqref{eq:mfc_fwd_fb2}, 
we can rewrite \eqref{eq:mfc_fbsde2} into  the following decoupled FBSDE:
\begin{subequations}\label{eq:decoupled_fbsde}
\begin{align}
\d X_t&=\bar{b}(t,X_t)\,\d t +\bar{\sigma} (t,X_t)\, \d W_t, 
\q X_0=\xi_0,
\label{eq:decoupled_sde}
\\
\d Y_t&=-\bar{f}(t,X_t,Y_t,Z_t)\,\d t+Z_t\,\d W_t,
\q
 Y_T=\bar{g}(X_T).
\end{align}
\end{subequations}
with coefficients $\bar{b}, \bar{\sigma}, \bar{f}$ and $\bar{g}$ defined as follows:
\begin{alignat*}{2}
\bar{b}(t,x)&\coloneqq 
\hat{b}(t,x,v(t,x),\sP_{(X_t,Y_t)}),
&&
\q \bar{\sigma}(t,x)\coloneqq \sigma (t,x,\sP_{X_t}),
\\
\bar{f}(t,x,y,z)&\coloneqq
\hat{f}(t,x,y,z, \sP_{(X_t,Y_t,Z_t)}),
&&\q 
\bar{g}(x)\coloneqq \hat{g}(x,\sP_{X_T}).
\end{alignat*}
By Propositions \ref{prop:mcfE1} and   \ref{prop:fbsde_wellposedness},
 these coefficients are $C$-Lipschitz continuous in the state variable
 with a constant $C$ independent of  $\xi_0$,
and satisfy the estimates $\int_0^T|\bar{b}(t,0)|^2\,\d t< \infty$, $\sup_{t\in [0,T]}|\bar{\sigma}(t,0)|< \infty$
 and $\int_0^T|\bar{f}(t,0,0,0)|^2\, \d t<\infty$.
Hence, by applying  \cite[Theorem 5.2.2 (i)]{zhang2017} to \eqref{eq:decoupled_fbsde},
we see there exists a constant $C>0$
such that $|Z_t|\le C|\sigma(t,X_t,\sP_{X_t})|$ $\d \sP\otimes \d t$-a.e.. We remark that in \cite{zhang2017} 
the initial state $\xi_0$ is assumed to be deterministic
and 
the coefficients of the FBSDE are assumed to be H\"{o}lder continuous in time. However, the proof 
relies on expressing the process $Z$ in terms of the Malliavin derivatives  of    $X$ and $Y$,
and hence can be extended to the present setting 
where  $\xi_0$ is $\cF_0$-measurable
and the coefficients are   measurable in time
and satisfy the above estimates.

By  the Lipschitz continuity of $v$, the estimate $\sup_{t\in[0,T]}|v(t,0)|\le C(1+\|\xi_0\|_{L^2})$
and  standard moment estimates of \eqref{eq:decoupled_sde},  $\|X\|_{\cS^p}\le C_{(p)}(1+\|\xi_0\|_{L^p})$,
which
along with the relations $Y_t=v(t,X_t)$ and $|Z_t|\le C|\sigma(t,X_t,\sP_{X_t})|$
leads to $\|Y\|_{\cS^p}+\|Z\|_{\cS^p}\le C_{(p)}(1+\|\xi_0\|_{L^p})$.
Moreover,
by \eqref{eq:mfc_fbsde2}, 
H\"{o}lder's inequality
and the  Burkholder-Davis-Gundy  inequality,
the process $X$ satisfies 
for each $p\ge 2$, $t,s\in [0,T]$,
\begin{align*}
&\sE\left[\sup_{s\le r\le t}|X_r-X_s|^p\right]
\\
&\le \sE\bigg[ \bigg(\int_s^t|\hat{b}(r,X_r,Y_r,\sP_{(X_r,Y_r)})|\, \d r \bigg)^p\bigg] 
+\sE\bigg[\sup_{s\le r\le t}\bigg|\int_s^r \sigma(u,X_u,\sP_{X_u})\,\d W_u\bigg|^p\bigg]
\\
&\le
C_{(p)}\bigg\{
(\|\hat{b}(\cdot,0,0,\bm{\delta}_{{0}_{n+n}})\|^p_{L^2(0,T)}+\|(X,Y)\|^p_{\cS^p})|t-s|^{\frac{p}{2}}
+\sE\bigg[\bigg(\int_s^t |\sigma(r,X_r,\sP_{X_r})|^2\,\d r\bigg)^{\frac{p}{2}}\bigg]
\bigg\}
\\
&\le
C_{(p)}\bigg\{
(\|\hat{b}(\cdot,0,0,\bm{\delta}_{{0}_{n+n}})\|^p_{L^2(0,T)}
+\|\sigma(\cdot,0,\bm{\delta}_{{0}_{n}})\|^p_{L^\infty(0,T)}
+\|(X,Y)\|^p_{\cS^p})|t-s|^{\frac{p}{2}}
\bigg\},
\end{align*}
and the process $Y$ satisfies 
for each $p\ge 2$, $t,s\in [0,T]$,
\begin{align*}
&\sE\left[\sup_{s\le r\le t}|Y_r-Y_s|^p\right]
\\
&\le \sE\bigg[ \bigg(\int_s^t|\hat{f}(r,X_r,Y_r,Z_r,\sP_{(X_r,Y_r,Z_r)})|\, \d r \bigg)^p\bigg] 
+\sE\bigg[\sup_{s\le r\le t}\bigg|\int_s^r Z_u\,\d W_u\bigg|^p\bigg]
\\
&\le
C_{(p)}\bigg\{
(\|\hat{f}(\cdot,0,0,0,\bm{\delta}_{{0}_{n+n+nd}})\|^p_{L^2(0,T)}+\|(X,Y,Z)\|^p_{\cS^p})|t-s|^{\frac{p}{2}}
+\sE\bigg[\bigg(\int_s^t|Z_r|^2\,\d r\bigg)^{\frac{p}{2}}\bigg]
\bigg\},
\end{align*}
which together with  Proposition \ref{prop:mcfE1},
 the inequality that 
$\sE[(\int_s^t|Z_r|^2\,\d r)^{\frac{p}{2}}]
\le \|Z\|^p_{\cS^p} (t-s)^{\frac{p}{2}}$ and
 the estimate of $\|(X,Y,Z)\|_{\cS^p}$
 leads to 
the desired H\"{o}lder continuity of the processes $X$ and $Y$. 
\end{proof}

We are now ready to prove Theorem \ref{TH:ControlRegularity} 
based on   Theorem 
\ref{TH:fbsde_Regularity}.

\begin{proof}[Proof of Theorem \ref{TH:ControlRegularity}]
Let 
$\hat{\a}:[0,T]\t \sR^n \t \sR^n \t \cP_2(\sR^{n} \t \sR^{n}) \to \bA$ be
the function in (H.\ref{assum:mfcE_hat}(\ref{item:ex})).
Define
 for each $t\in [0,T]$  that 
$\hat{\a}_t=\hat{\a}(t,X_t,Y_t, \sP_{(X_t,Y_t)})$,
and write $\hat{\a}=(\hat{\a}_t)_{t\in [0,T]}$
with a slight abuse of notation.
The local boundedness 
and Lipschitz continuity 
of the function $\hat{\a}$ (see (H.\ref{assum:mfcE_hat}(\ref{item:ex}))) show that 
$\|\hat{\a}\|_{\cS^p}\le \|\hat{\a}(\cdot,0,0,\bm{\delta}_{0_{n + n}})\|_{{\cS^p}}+C(\|X\|_{\cS^p}+\|Y\|_{\cS^p})\le C(1+\|\xi_0\|_{L^p})$ for all $p\ge 2$. 
Then, the assumption that  $\xi_0\in L^2(\cF_0;\sR^n)$ and the definition of 
the function
$\hat{\a}$ in (H.\ref{assum:mfcE_hat}(\ref{item:ex}))
imply that the control $\hat{\a}$ is admissible (i.e.,~$\hat{\a}\in \cA$)
and
satisfies 
\eqref{eq:opti_re} (equivalently  \eqref{eq:opti}),
which shows that 
$\hat{\a}$ is an optimal control of \eqref{eq:mfcE}.
The uniqueness of optimal controls of \eqref{eq:mfcE} follows from the strong convexity of 
the cost functional $J:\cA\to \sR$, which will be shown in Lemma \ref{lemma:J}.

Finally, for any given $0\le s\le r\le t\le T$, by (H.\ref{assum:mfcE_hat}),
\begin{align*}
&|\hat{\a}_r-\hat{\a}_s|=|\hat{\a}(r,X_r,Y_r,\sP_{(X_r,Y_r)})-\hat{\a}(s,X_s,Y_s,\sP_{(X_s,Y_s)})|
\\
&\le 
C\big\{(1+
|X_r|+|Y_r|+\|\sP_{(X_r,Y_r)}\|_2)|r-s|^{\frac{1}{2}}
+
|X_r-X_s|
\\
&\q +|Y_r-Y_s|+\cW_2(\sP_{(X_r,Y_r)},\sP_{(X_s,Y_s)})
\big\},
\end{align*}
which together with the moment estimates and regularity of the processes $X,Y$ leads to 
\begin{align*}
&\sE\left[\sup_{s\le r\le t}|\hat{\a}_r-\hat{\a}_s|^p\right]^{\frac{1}{p}}
\\
&\le 
C\bigg((1+
\|X\|_{\cS^p}+\|Y\|_{\cS^p})|t-s|^{\frac{1}{2}}
+
\sE\left[\sup_{s\le r\le t}|X_r-X_s|^p\right]^{\frac{1}{p}}
 +\sE\left[\sup_{s\le r\le t}|Y_r-Y_s|^p\right]^{\frac{1}{p}}
\bigg)
\\
&\le 
C(1+\|\xi_0\|_{L^p})|t-s|^{\frac{1}{2}}.
\end{align*}
This completes the proof of Theorem \ref{TH:ControlRegularity}.
\end{proof}

\subsection{Proofs of Theorems 
\ref{thm:PCPT_discrete_value}
and \ref{thm:control_PCPT_discrete}}\l{sec:conv_PCPT}

\begin{proof}[Proof of Theorem \ref{thm:PCPT_discrete_value}]
Throughout this proof,
let $\xi_0\in L^2(\cF_0;\sR^n)$ be a given initial state,
let  $\pi=\{0=t_0<\cdots<t_N=T\}$ be a given
 partition of $[0,T]$
with  stepsize $|\pi|=\max_{i=0,\ldots, N-1}(t_{i+1}-t_i)$,
let $\cA_{\pi}\subset \cA$ be the associated piecewise constant controls
and  
let $C$ be a generic constant, 
which is independent of   $\xi_0$,   $\pi$
and controls $\a\in \cA$, and may take a different value at each occurrence.

\textbf{Step 1: Estimate an upper bound of $V_\pi(\xi_0)-V(\xi_0)$.}
Let  $\hat{\a}\in \cA$ be an optimal control of \eqref{eq:mfcE}
satisfying 
that
$ \|\hat{\a}\|_{\cS^2}^2
 \le
C
(1+\|\xi_0\|^2_{L^{2}})$
and 
$ \|\hat{\a}_t-\hat{\a}_s\|_{L^2}^2
 \le
C
(1+\|\xi_0\|^2_{L^{2}})
|t-s|$
for all $ s,t\in [0,T]$,
and 
${X}^{\hat{\a}}$ be the solution to \eqref{eq:mfcE_fwd}  with the control ${\hat{\a}}$
satisfying
 $ \|X^{\hat{\a}}\|_{\cS^2}^2
 \le
C
(1+\|\xi_0\|^2_{L^{2}})$
and 
$
\|X^{\hat{\a}}_t-X^{\hat{\a}}_s\|_{L^2}^2
  \le
C
(1+\|\xi_0\|^2_{L^{2}})
|t-s|$
for all
$ s,t\in [0,T]$
(see Theorem \ref{TH:ControlRegularity}).
Define 
 the   piecewise constant approximation  $\hat{\a}^{\pi}$
of the process $\hat{\a}$ on $\pi$
such that  
$\hat{\a}^\pi_{t}=\sum_{i=0}^{N-1}\hat{\a}_{t_i} \bm{1}_{[t_i,t_{i+1})}(t)$ for all $t\in [0,T)$.
Then  $\hat{\a}^\pi\in \cA_\pi$ and it holds for all $t\in [0,T)$ that 
$t\in [t_i,t_{i+1})$ for some $i\in \{0,\ldots, N-1\}$ and 
$$
\|\hat{\a}^\pi_{t}-\hat{\a}_{t}\|_{L^2}=\|\hat{\a}_{t_i}-\hat{\a}_{t}\|_{L^2}
\le C(1+\|\xi_0\|_{L^{2}})|\pi|^{{1}/{2}}.
$$
Let
$\hat{X}^\pi$ be the solution to \eqref{eq:mfcE_fwd_euler}  with the control $\hat{\a}^\pi$.
By   the Lipschitz continuity of  $(b,\sigma)$
and Gronwall's inequality, 
$$ 
\max_{t_i\in \pi}\|\hat{X}^{\pi}_{t_i}\|_{L^2}^2
 \le
C
\big(1+\|\xi_0\|^2_{L^{2}}+\max_{t_i\in \pi}\|\hat{\a}^\pi_{t_i}\|^2_{L^2}\big)
\le 
C(1+\|\xi_0\|^2_{L^{2}}).
$$

For each $i\in \{0,\ldots,N-1\}$,
\begin{align}\l{eq:X^hat_a-X^pi}
\begin{split}
&\sE[|X^{\hat{\a}}_{t_{i+1}}-\hat{X}^{\pi}_{t_{i+1}}|^2]
\\
&
\le C\bigg\{
\sE\bigg[\bigg|\sum_{j=0}^{ i}\int_{t_j}^{t_{j+1}}
\big(b(t,X^{\hat{\a}}_t,\hat{\a}_t,\sP_{(X^{\hat{\a}}_t,\hat{\a}_t)})-
b(t_j,\hat{X}^{\pi}_{t_j},\hat{\a}^\pi_{t_j},\sP_{(\hat{X}^{\pi}_{t_j},\hat{\a}^\pi_{t_j})})\big)
\,\d t\bigg|^2\bigg]
\\
&\q 
+\sE\bigg[\sum_{j=0}^{ i}\int_{t_j}^{t_{j+1}}
|\sigma(t,X^{\hat{\a}}_t,\sP_{X^{\hat{\a}}_t})-
\sigma(t_j,\hat{X}^{\pi}_{t_j},\sP_{\hat{X}^{\pi}_{t_j}})|^2
\,\d t\bigg]
\bigg\}
\\
&\le
 C\bigg\{
R_1+
\sE\bigg[\bigg(
\sum_{j=0}^{ i}\int_{t_j}^{t_{j+1}}
\big|b({t_j},X^{\hat{\a}}_{t_j},\hat{\a}_{t_j},\sP_{(X^{\hat{\a}}_{t_j},\hat{\a}_{t_j})})-
b(t_j,\hat{X}^{\pi}_{t_j},\hat{\a}^\pi_{t_j},\sP_{(\hat{X}^{\pi}_{t_j},\hat{\a}^\pi_{t_j})})\big|
\,\d t\bigg)^2\bigg]
\\
&\q 
+\sE\bigg[\sum_{j=0}^{ i}\int_{t_j}^{t_{j+1}}
|\sigma(t_j,X^{\hat{\a}}_{t_j},\sP_{X^{\hat{\a}}_{t_j}})-
\sigma(t_j,\hat{X}^{\pi}_{t_j},\sP_{\hat{X}^{\pi}_{t_j}})|^2
\,\d t\bigg]
\bigg\},
\end{split}
\end{align}
with the residual term $R_1$ defined as 
\begin{align}\l{eq:R1}
\begin{split}
R_1
&\coloneqq
 \sE\bigg[\bigg(\sum_{i=0}^{N-1}\int_{t_i}^{t_{i+1}}
\big|b(t,X^{\hat{\a}}_t,\hat{\a}_t,\sP_{(X^{\hat{\a}}_t,\hat{\a}_t)})-
b(t_i,X^{\hat{\a}}_{t_i},\hat{\a}_{t_i},\sP_{(X^{\hat{\a}}_{t_i},\hat{\a}_{t_i})})\big|
\,\d t\bigg)^2\bigg]
\\
&\q 
+\sE\bigg[\sum_{i=0}^{ N-1}\int_{t_i}^{t_{i+1}}
|\sigma(t,X^{\hat{\a}}_t,\sP_{X^{\hat{\a}}_t})-
\sigma(t_i,X^{\hat{\a}}_{t_i},\sP_{X^{\hat{\a}}_{t_i}})|^2
\,\d t\bigg]
\\
&\le
 T
 \sE\bigg[\sum_{i=0}^{N-1}\int_{t_i}^{t_{i+1}}
\big|b(t,X^{\hat{\a}}_t,\hat{\a}_t,\sP_{(X^{\hat{\a}}_t,\hat{\a}_t)})-
b(t_i,X^{\hat{\a}}_{t_i},\hat{\a}_{t_i},\sP_{(X^{\hat{\a}}_{t_i},\hat{\a}_{t_i})})\big|^2
\,\d t\bigg]
\\
&\q 
+\sE\bigg[\sum_{i=0}^{ N-1}\int_{t_i}^{t_{i+1}}
|\sigma(t,X^{\hat{\a}}_t,\sP_{X^{\hat{\a}}_t})-
\sigma(t_i,X^{\hat{\a}}_{t_i},\sP_{X^{\hat{\a}}_{t_i}})|^2
\,\d t\bigg],
\end{split}
\end{align}
where the last inequality used   the Cauchy-Schwarz inequality.
Hence, by applying the Lipschitz continuity of $b$ and $\sigma$ and Gronwall's inequality, 
\begin{align}\l{eq:X_hat-X_a,pi}
\begin{split}
\max_{t_i\in \pi}\sE[|X^{\hat{\a}}_{t_i}-\hat{X}^{\pi}_{t_i}|^2]
&\le
 C\Big(
R_1+\max_{t_i\in \pi}\|\hat{\a}_{t_i}-\hat{\a}^\pi_{t_i}\|_{L^2}^2
\Big),
\end{split}
\end{align}
which, together with  the definition of $V_\pi(\xi_0)$ and 
the optimality of 
$\hat{\a}$  for \eqref{eq:mfcE},
gives that
\begin{align*}
&V_\pi(\xi_0)-V(\xi_0)
\le J_\pi(\hat{\a}^\pi;\xi_0)-J(\hat{\a};\xi_0)
\\
&\le\sE\bigg[\sum_{i=0}^{N-1}
\int_{t_i}^{t_{i+1}}
\big|
 f(t_i,\hat{X}^{\pi}_{t_i},\hat{\a}^\pi_{t_i},\sP_{(\hat{X}^{\pi}_{t_i},\hat{\a}^\pi_{t_i})})
 -f(t,X^{\hat{\a}}_t,\hat{\a}_t,\sP_{(X^{\hat{\a}}_t,\hat{\a}_t)})
 \big|\,\d t
\\
&\q +
|g(\hat{X}^{\pi}_T,\sP_{\hat{X}^{\pi}_T})-g(X^{\hat{\a}}_T,\sP_{X^{\hat{\a}}_T})|
\bigg]
\\
&\le
C\bigg(R_2+\sE\bigg[\sum_{i=0}^{N-1}
\int_{t_i}^{t_{i+1}}
\big|
f(t_i,\hat{X}^{\pi}_{t_i},\hat{\a}^\pi_{t_i},\sP_{(\hat{X}^{\pi}_{t_i},\hat{\a}^\pi_{t_i})})
-f({t_i},X^{\hat{\a}}_{t_i},\hat{\a}_{t_i},\sP_{(X^{\hat{\a}}_{t_i},\hat{\a}_{t_i})})
 \big|\,\d t
\\
& \q +
  |g(\hat{X}^{\pi}_T,\sP_{\hat{X}^{\pi}_T})-g(X^{\hat{\a}}_T,\sP_{X^{\hat{\a}}_T})|
\bigg]
\bigg)
\end{align*}
with the residual term defined by 
\begin{align}\l{eq:R2}
\begin{split}
R_2
&\coloneqq
\sE\bigg[\sum_{i=0}^{N-1}
\int_{t_i}^{t_{i+1}}
\big|
 f(t,X^{\hat{\a}}_t,\hat{\a}_t,\sP_{(X^{\hat{\a}}_t,\hat{\a}_t)})
 -f({t_i},X^{\hat{\a}}_{t_i},\hat{\a}_{t_i},\sP_{(X^{\hat{\a}}_{t_i},\hat{\a}_{t_i})})
 \big|\,\d t
\bigg].
\end{split}
\end{align}
Then, by using  \eqref{eq:f_local_lipschitz},
 the Cauchy-Schwarz inequality,
\eqref{eq:X_hat-X_a,pi}
and  the fact that $\|\hat{\a}^\pi_{t}-\hat{\a}_{t}\|_{L^2}\le C(1+\|\xi_0\|_{L^{2}})|\pi|^{{1}/{2}}$, 
\begin{align}\l{eq:V_pi-V}
\begin{split}
&V_\pi(\xi_0)-V(\xi_0)
\\
&\le
C\bigg(R_2+
\max_{t_i\in \pi}\big(1+\|(X^{\hat{\a}}_{t_i},\hat{\a}_{t_i},\hat{X}^{\pi}_{t_i},\hat{\a}^\pi_{t_i})\|_{L^2}\big)
\big(\|X^{\hat{\a}}_{t_i}-\hat{X}^{\pi}_{t_i}\|_{L^2}+\|\hat{\a}_{t_i}-\hat{\a}^\pi_{t_i}\|_{L^2}\big)
\bigg)
\\
&\le
C\Big(R_2+
(1+\|\xi_0\|_{L^{2}})
\big(
R^{1/2}_1+(1+\|\xi_0\|_{L^{2}})|\pi|^{{1}/{2}}
\big)
\Big).
\end{split}
\end{align}
Hence it remains to estimate the residual terms $R_1$ and $R_2$ defined as in \eqref{eq:R1} and \eqref{eq:R2}, respectively.
Note that 
by 
  (H.\ref{assum:mfcE}(\ref{item:mfcE_lin})) and  (H.\ref{assum:mfc_Holder_t}),
for all $(x,a,\mu,\eta), (x',a',\mu',\eta')\in \sR^k\t \bA\t \cP_2(\sR^n)\t\cP_2(\sR^n\t \sR^k)$,
\begin{align}
\label{eq:local_holder_b_sigma}
\begin{split}
&|b(r,x,a,\eta)-b(s,x',a',\eta')|
\\
&\le C
\Big((1+|x|+|a|+\|\eta\|_2)|r-s|^{1/2}+|x-x'|+|a-a'|+\cW_2(\eta,\eta')\Big),
\\
&|\sigma(r,x,\mu)-\sigma(s,x',\mu')|
\le C\Big((1+|x|+\|\mu\|_2)|r-s|^{1/2}+|x-x'|+\cW_2(\mu,\mu')\Big).
\end{split}
\end{align}
This along with the H\"{o}lder regularity of $(X^{\hat{\a}},\hat{\a})$
implies that 
\begin{align*}
R_1
&\le
C
\bigg((1+\|X^{\hat{\a}}\|^2_{\cS^2}+\|\hat{\a}\|^2_{\cS^2})|\pi|+
\sup_{t_i\in \pi,r\in [t_i,t_{i+1})}
\Big(
\|X^{\hat{\a}}_r-X^{\hat{\a}}_{t_i}\|_{L^2}^2
+\|{\hat{\a}}_r-{\hat{\a}}_{t_i}\|_{L^2}^2
\Big)
\bigg)
\\
&\le
C(1+\|\xi_0\|^2_{L^2})|\pi|,
\end{align*}
while  \eqref{eq:f_local_lipschitz}
and 
(H.\ref{assum:mfc_Holder_t}) 
gives 
\begin{align*}
R_2
&\le
C\bigg[
(1+\|X^{\hat{\a}}\|^2_{\cS^2}+\|\hat{\a}\|^2_{\cS^2})|\pi|^{1/2}
\\
&\q 
+
(1+\|X^{\hat{\a}}\|_{\cS^2}+\|\hat{\a}\|_{\cS^2})
\sup_{t_i\in \pi,r\in [t_i,t_{i+1})}
\Big(
\|X^{\hat{\a}}_r-X^{\hat{\a}}_{t_i}\|_{L^2}
+\|{\hat{\a}}_r-{\hat{\a}}_{t_i}\|_{L^2}
\Big)
\bigg]
\\
&\le 
C(1+\|\xi_0\|^2_{L^2})|\pi|^{1/2}.
\end{align*}
These estimates and \eqref{eq:V_pi-V} yields 
the upper bound
  $V_\pi(\xi_0)-V(\xi_0)\le C(1+\|\xi_0\|^2_{L^2})|\pi|^{1/2}$.

\textbf{Step 2: Estimate an upper bound of $V(\xi_0)-V_\pi(\xi_0)$.}
Note that the additional compactness assumption of $\bA$ implies that there exists $C>0$ such that
$\|\a\|_{\cH^2}\le C$ for all $\a\in \cA_\pi$.
Then  standard moment estimates for MV-SDEs (see e.g.~\cite[Theorem 3.3]{reis2019})
shows that there exists $C>0$ such that for all $\a\in \cA_\pi$, the solution to 
\eqref{eq:mfcE_fwd} with the control $\a$ satisfies 
$\|X^\a\|_{\cS^2}\le C(1+\|\xi_0\|_{L^2})$.
Moreover, for any $0\le s\le r\le t\le T$, by  the  Burkholder-Davis-Gundy  inequality,
H\"{o}lder's inequality
and (H.\ref{assum:mfcE}(\ref{item:mfcE_lin})), 
\begin{align}\l{eq:X^a_Holder_discrete}
\begin{split}
&\sE\left[\sup_{s\le r\le t}|X^\a_r-X^\a_s|^2\right]
\\
&\le 2
\sE\bigg[\bigg(
\int_s^t
|b(u,X^\a_u,\a_u,\sP_{(X^\a_u,\a_u)})|^2\, \d u\bigg)(t-s)
+
\int_s^t
|\sigma(u,X^\a_u,\sP_{X^\a_u})|^2\, \d u
\bigg]
\\
&\le C
(\|b_0\|^2_{L^2(0,T)}+\|\sigma_0\|^2_{L^\infty(0,T)}+\|X^\a\|^2_{\cS^2}+\|\a\|^2_{\cH^2})(t-s)
\\
&\le C(1+\|\xi_0\|^2_{L^2})(t-s).
\end{split}
\end{align}
Similarly,
for each $\a\in \cA_\pi$, 
by   the Lipschitz continuity of  $b,\sigma$
and Gronwall's inequality,
\begin{equation}\label{eq:apriori_1}
\max_{t_i\in \pi}\|{X}^{\a,\pi}_{t_i}\|_{L^2}^2
 \le
C
\big(1+\|\xi_0\|^2_{L^{2}}+\max_{t_i\in \pi}\|{\a}_{t_i}\|^2_{L^2}\big)
\le 
C(1+\|\xi_0\|^2_{L^{2}}).
\end{equation}
Let $\a\in \cA_\pi$ be fixed,
and let $X^\a$ and ${X}^{\a,\pi}$ be the solution to \eqref{eq:mfcE_fwd} and 
\eqref{eq:mfcE_fwd_euler} with the control $\a$, respectively. 
Then by following similar arguments as those for \eqref{eq:X^hat_a-X^pi} and \eqref{eq:R1},
we have
for each
$i\in \{0,\ldots,N-1\}$ that
\begin{align*}
&\sE[|X^{{\a}}_{t_{i+1}}-{X}^{\a,\pi}_{t_{i+1}}|^2]
\\
&\le
 C\bigg\{
R_1^\a+
\sE\bigg[\bigg(
\sum_{j=0}^{ i}\int_{t_j}^{t_{j+1}}
\big|b({t_j},X^{{\a}}_{t_j},{\a}_{t_j},\sP_{(X^{{\a}}_{t_j},{\a}_{t_j})})-
b(t_j,{X}^{\a,\pi}_{t_j},{\a}_{t_j},\sP_{({X}^{\a,\pi}_{t_j},{\a}_{t_j})})\big|
\,\d t\bigg)^2\bigg]
\\
&\q 
+\sE\bigg[\sum_{j=0}^{ i}\int_{t_j}^{t_{j+1}}
|\sigma(t_j,X^{{\a}}_{t_j},\sP_{X^{{\a}}_{t_j}})-
\sigma(t_j,{X}^{\a,\pi}_{t_j},\sP_{{X}^{\a,\pi}_{t_j}})|^2
\,\d t\bigg]
\bigg\},
\end{align*}
with the residual term $R^\a_1$ defined as 
\begin{align}\l{eq:R1_a}
\begin{split}
R^\a_1
&\coloneqq
 \sE\bigg[\sum_{i=0}^{N-1}\int_{t_i}^{t_{i+1}}
\big|b(t,X^{{\a}}_t,{\a}_t,\sP_{(X^{{\a}}_t,{\a}_t)})-
b(t_i,X^{{\a}}_{t_i},{\a}_{t_i},\sP_{(X^{{\a}}_{t_i},{\a}_{t_i})})\big|^2
\,\d t\bigg]
\\
&\q 
+\sE\bigg[\sum_{i=0}^{ N-1}\int_{t_i}^{t_{i+1}}
|\sigma(t,X^{{\a}}_t,\sP_{X^{{\a}}_t})-
\sigma(t_i,X^{{\a}}_{t_i},\sP_{X^{{\a}}_{t_i}})|^2
\,\d t\bigg],
\end{split}
\end{align}
which, along with  the Lipschitz continuity of $b$ and $\sigma$ and Gronwall's inequality, gives that
\begin{align}\l{eq:X_a-X_a,pi_discrete_a}
\begin{split}
\max_{t_i\in \pi}\sE[|X^{{\a}}_{t_i}-{X}^{\a,\pi}_{t_i}|^2]
&\le
 CR^\a_1.
\end{split}
\end{align}
Hence, by the local Lipschitz continuiy \eqref{eq:f_local_lipschitz} of $f$,
H\"{o}lder's inequality,
the \textit{a priori} estimate for $\|X^\a\|_{\cS^2}$, and the estimates  \eqref{eq:apriori_1}
 and \eqref{eq:X_a-X_a,pi_discrete_a},  
 \begin{align}\l{eq:V-V_pi}
 \begin{split}
&
V(\xi_0)-V_\pi(\xi_0)\le \sup_{\a\in \cA_\pi}
\big|J({\a};\xi_0)-J_\pi({\a};\xi_0)\big|
\\
&\le
\sup_{\a\in \cA_\pi}
\sE\bigg[\sum_{i=0}^{N-1}
\int_{t_i}^{t_{i+1}}
\big|
f(t,X^{{\a}}_t,{\a}_t,\sP_{(X^{{\a}}_t,{\a}_t)})
- f(t_i,{X}^{\a,\pi}_{t_i},{\a}_{t_i},\sP_{({X}^{\a,\pi}_{t_i},{\a}_{t_i})})
 \big|\,\d t
\\
&\q +
|g(X^{{\a}}_T,\sP_{X^{{\a}}_T})-g({X}^{\a,\pi}_T,\sP_{{X}^{\a,\pi}_T})|
\bigg]
\\
&\le
C\sup_{\a\in \cA_\pi}\bigg(R^\a_2+\sE\bigg[\sum_{i=0}^{N-1}
\int_{t_i}^{t_{i+1}}
\big|
f({t_i},X^{{\a}}_{t_i},{\a}_{t_i},\sP_{(X^{{\a}}_{t_i},{\a}_{t_i})})
-f(t_i,{X}^{\a,\pi}_{t_i},{\a}_{t_i},\sP_{({X}^{\a,\pi}_{t_i},{\a}_{t_i})})
 \big|\,\d t
\\
& \q +
  |g(X^{{\a}}_T,\sP_{X^{{\a}}_T})-g({X}^{\a,\pi}_T,\sP_{{X}^{\a,\pi}_T})|
\bigg]
\bigg)
\\
&\le
C\sup_{\a\in \cA_\pi}\Big(R^\a_2+ (1+\|\xi_0\|_{L^2})
\max_{t_i\in \pi}\|X^{{\a}}_{t_i}-{X}^{\a,\pi}_{t_i}\|_{L^2}\Big)
\\
&
\le
C\sup_{\a\in \cA_\pi}\Big(R^\a_2+ (1+\|\xi_0\|_{L^2})(R^\a_1)^{1/2}\Big)
 \end{split}
\end{align}
with the residual term $R^\a_1$ defined as in \eqref{eq:R1_a} and 
 the residual term   $R^\a_2$ defined by:
\begin{align}\l{eq:R2_a}
\begin{split}
R_2^\a
&\coloneqq
\sE\bigg[\sum_{i=0}^{N-1}
\int_{t_i}^{t_{i+1}}
\big|
 f(t,X^{{\a}}_t,{\a}_t,\sP_{(X^{{\a}}_t,{\a}_t)})
 -f({t_i},X^{{\a}}_{t_i},{\a}_{t_i},\sP_{(X^{{\a}}_{t_i},{\a}_{t_i})})
 \big|\,\d t
\bigg].
\end{split}
\end{align}
Note that for each $\a\in \cA_\pi$, we have ${\a}_t={\a}_{t_i}$ for all $t\in [t_i,t_{i+1})$, $i\in \{0,\ldots,N-1\}$,
which 
together with 
(H.\ref{assum:mfc_Holder_t}),
the local Lipschitz continuity of coefficients 
\eqref{eq:f_local_lipschitz}  and \eqref{eq:local_holder_b_sigma},
and the estimate \eqref{eq:X^a_Holder_discrete}
implies that
\begin{align*}
R^\a_1
&\le
C
\Big((1+\|X^{{\a}}\|^2_{\cS^2}+\|{\a}\|^2_{\cH^2})|\pi|+
\sup_{t_i\in \pi,r\in [t_i,t_{i+1})}
\|X^{{\a}}_r-X^{{\a}}_{t_i}\|_{L^2}^2
\Big)
\\
&\le
C(1+\|\xi_0\|^2_{L^2})|\pi|,
\\
R^\a_2
&\le
C\Big(
(1+\|X^{{\a}}\|^2_{\cS^2}+\|{\a}\|^2_{\cH^2})|\pi|^{1/2}
\\
&\q 
+
(1+\|X^{{\a}}\|_{\cS^2}+\|{\a}\|_{\cH^2})
\sup_{t_i\in \pi,r\in [t_i,t_{i+1})}
\|X^{{\a}}_r-X^{{\a}}_{t_i}\|_{L^2}
\Big)
\\
&\le 
C(1+\|\xi_0\|^2_{L^2})|\pi|^{1/2}.
\end{align*}
These estimates lead to the desired upper bound
 $V(\xi_0)-V_\pi(\xi_0)\le C(1+\|\xi_0\|^2_{L^2})|\pi|^{1/2}$.
\end{proof}

To prove Theorem \ref{thm:control_PCPT_discrete}, we first show  the cost functional $J(\cdot;\xi_0):\cA\to \sR$
defined by \eqref{eq:mfcE}
is strongly convex and coercive. 

\begin{Lemma}\l{lemma:J}
Suppose (H.\ref{assum:mfcE})
and (H.\ref{assum:mfcE_hat})
 hold, and  $\xi_0\in L^2(\cF_0;\sR^n)$.
 The functional  
 let $J(\cdot;\xi_0):\cA\to \sR$   defined by  \eqref{eq:mfcE}  is continuous and satisfies for all $\a,\b\in \cA$ and $\tau\in [0,1]$,
 \begin{align*}
&\tau J(\a;\xi_0)+(1-\tau) J(\b;\xi_0)-J(\tau \a+(1-\tau)\b;\xi_0)\ge 
\tau(1-\tau)(\lambda_1+\lambda_2)
\|\a-\b\|^2_{\cH^2},
\end{align*}
where $\lambda_1, \lambda_2$ are the constants appearing in (H.\ref{assum:mfcE}(\ref{item:mfcE_convex})).
Moreover,   for all $\a\in \cA$, 
\bb\l{eq:coercive}
J(\hat{\a};\xi_0)-J(\a;\xi_0)\le -(\lambda_1+\lambda_2)\|\hat{\a}-{\a}\|_{\cH^2}^2,
\ee
where $\hat{\a}$ is the unique minimizer of \eqref{eq:mfcE} defined in  Theorem \ref{TH:ControlRegularity}.
\end{Lemma}
\begin{proof}
The continuity of $J$ follows directly from  stability results of \eqref{eq:mfcE} and the local Lipschitz continuity of functions $(f,g)$.

We now show the strong convexity of the cost functional $J$.  
Let $\a,\b\in \cA$, $\tau\in [0,1]$, 
and let $X^\a$ (resp.~$X^\b$) be the solution to \eqref{eq:mfcE} with control $\a$ (resp.~$\b$).
Let $\gamma=\tau\a+(1-\tau)\b$
and
let $X\coloneqq\tau X^\a+(1-\tau) X^\b$.
We first show $X=X^\gamma$,
where $X^\gamma$  be the solution to \eqref{eq:mfcE} with control $\gamma$.
 It is clear that $X_0=\tau X^\a_0+(1-\tau) X^\b_0=\xi_0=X^\gamma_0$.
For  each $t\in[0,T]$,
we see 
 that 
 $\sE[(X_t,\gamma_t)]=\tau \sE[(X^\a_t,\a_t)]+(1-\tau)\sE[(X^\b_t,\b_t)]$,
 which together with the linearity of the functions $b,\sigma$  in $(x,a,\eta)$ (see (H.\ref{assum:mfcE}(\ref{item:mfcE_lin})))
gives that
\begin{align*}
 b(t,X_t,\gamma_t,\sP_{(X_t,\gamma_t)})
& =
 \tau b(t,X^\a_t,\a_t,\sP_{(X^\a_t,\a_t)})
+(1-\tau)b(t,X^\b_t,\b_t,\sP_{(X^\b_t,\b_t)}),
\\
 \sigma(t,X_t,\sP_{X_t})&=
 \tau \sigma(t,X^\a_t,\sP_{X^\a_t})
+(1-\tau)\sigma(t,X^\b_t,\sP_{X^\b_t}).
\end{align*}
Hence, we can show by using   It\^{o}'s formula that
$X$ satisfies the same MV-SDE as $X^\gamma$, 
which along with the uniqueness of strong solutions
shows that
$X^\gamma=X=\tau X^\a+(1-\tau) X^\b$. 

Let $\tilde{X}^\a_T$ and  $\tilde{X}^\b_T$ be independent copies of 
${X}^\a_T$ and  ${X}^\b_T$, respectively, defined on $L^2(\tilde{\Om},\tilde{\cF},\tilde{\sP};\sR^n)$.
We see that $\tilde{X}^\gamma_T\coloneqq \tau \tilde{X}^\a_T+(1-\tau) \tilde{X}^\b_T$ is an independent copy of $X^\gamma_T$ with distribution $\sP_{X^\gamma_T}$.
Hence
we can obtain from the convexity of $g$ in  (H.\ref{assum:mfcE}(\ref{item:mfcE_convex}))
and $X^\gamma=\tau X^\a+(1-\tau) X^\b$
 that
\begin{align*}
&g(X^\a_T,\sP_{X^\a_T})-g(X^\gamma_T,\sP_{X^\gamma_T})
\\
&\ge \la \p_{x}g(X^\gamma_T,\sP_{X^\gamma_T}), X^\a_T-X^\gamma_T \ra
+\tilde{\sE}[\la\p_\mu g(X^\gamma_T,\mu)(\tilde{X}^\gamma_T),\tilde{X}^\a_T-\tilde{X}^\gamma_T\ra ]
\\
&=(1-\tau)
\Big(
\la \p_{x}g(X^\gamma_T,\sP_{X^\gamma_T}), X^\a_T-X^\b_T \ra
+\tilde{\sE}[\la\p_\mu g(X^\gamma_T,\mu)(\tilde{X}^\gamma_T),\tilde{X}^\a_T-\tilde{X}^\b_T\ra ]
\Big).
\end{align*}
Similarly, we can show that 
\begin{align*}
&g(X^\b_T,\sP_{X^\b_T})-g(X^\gamma_T,\sP_{X^\gamma_T})
\\
&\ge 
\tau
\Big(
\la \p_{x}g(X^\gamma_T,\sP_{X^\gamma_T}), X^\b_T-X^\a_T \ra
+\tilde{\sE}[\la\p_\mu g(X^\gamma_T,\mu)(\tilde{X}^\gamma_T),\tilde{X}^\b_T-\tilde{X}^\a_T\ra ]
\Big),
\end{align*}
which implies that 
$$
\tau \sE[g(X^\a_T,\sP_{X^\a_T})]+(1-\tau)\sE[g(X^\b_T,\sP_{X^\b_T})]
\ge
\sE[g(X^\gamma_T,\sP_{X^\gamma_T})].
$$
Now 
for each $t\in [0,T]$, let 
$(\tilde{X}^\a_t,\tilde{\a}_t)$ and  $(\tilde{X}^\b_t,\tilde{\b}_t)$ be independent copies of 
$({X}^\a_t,\a_t)$ and  $({X}^\b_t,\b_t)$ defined on $L^2(\tilde{\Om},\tilde{\cF},\tilde{\sP};\sR^n\t \sR^k)$, respectively.
We see that $(\tilde{X}^\gamma_t,\tilde{\gamma}_t)\coloneqq \tau (\tilde{X}^\a_t,\tilde{\a}_t)+(1-\tau) (\tilde{X}^\b_t,\tilde{\b}_t)$ is an independent copy of $(X^\gamma_t,\gamma_t)$ with distribution $\sP_{(X^\gamma_t,\gamma_t)}$.
Then we can obtain from the convexity of $f$ in   (H.\ref{assum:mfcE}(\ref{item:mfcE_convex}))
and $(X^\gamma,\gamma)=\tau( X^\a,\a)+(1-\tau) (X^\b,\b)$ that 
\begin{align*}
&f(t,X^\a_t,\a_t,\sP_{(X^\a_t,\a_t)})-f(t,X^\gamma_t,\gamma_t,\sP_{(X^\gamma_t,\gamma_t)})
\\
&\ge 
(1-\tau)\Big(
\la \p_{(x,a)}f(t,X^\gamma_t,\gamma_t,\sP_{(X^\gamma_t,\gamma_t)}), (X^\a_t-X^\b_t,\a_t-\b_t)\ra 
\\
&\q 
+\tilde{\sE}[\la\p_\mu f(t,X^\gamma_t,\gamma_t,\sP_{(X^\gamma_t,\gamma_t)})(\tilde{X}^\gamma_t,\tilde{\gamma}_t),\tilde{X}^\a_t-\tilde{X}^\b_t\ra] 
\\
&\q
+ \tilde{\sE}[\la\p_\nu f(t,X^\gamma_t,\gamma_t,\sP_{(X^\gamma_t,\gamma_t)}),\tilde{\alpha}_t-\tilde{\b}_t\ra ]
\Big)
+(1-\tau)^2
\Big(\lambda_1|\a_t-\b_t|^2+\lambda_2\tilde{\sE}[|\tilde{\alpha}_t-\tilde{\b}_t|^2]\Big),
\end{align*}
Similarly, we can derive a lower bound of 
$f(t,X^\b_t,\b_t,\sP_{(X^\b_t,\b_t)})-f(t,X^\gamma_t,\gamma_t,\sP_{(X^\gamma_t,\gamma_t)})$,
which subsequently leads to the estimate that
\begin{align*}
&\tau\sE[f(t,X^\a_t,\a_t,\sP_{(X^\a_t,\a_t)})]+(1-\tau)\sE[f(t,X^\b_t,\b_t,\sP_{(X^\b_t,\b_t)})]
-\sE[f(t,X^\gamma_t,\gamma_t,\sP_{(X^\gamma_t,\gamma_t)})]
\\
&\ge 
\Big(\tau(1-\tau)^2+\tau^2(1-\tau)\Big)
\Big(\lambda_1\sE[|\a_t-\b_t|^2]+\lambda_2\tilde{\sE}[|\tilde{\alpha}_t-\tilde{\b}_t|^2]
\Big)
\\
&=\tau(1-\tau)(\lambda_1+\lambda_2)
\sE[|\a_t-\b_t|^2].
\end{align*}
Hence, we can conclude from   \eqref{eq:mfcE} the desired strong convexity estimate.

We proceed to show the estimate \eqref{eq:coercive}.
The linearity of $(b,\sigma)$
and  the convexity of $f$  
in (H.\ref{assum:mfcE}) imply 
that the Hamiltonian $H$ defined as in \eqref{eq:mfcE_hamiltonian} is convex, i.e., 
for all $(t,y,z)\in [0,T]\in \sR^n\t \sR^{n\t d}$, $(x,\eta,a),(x',\eta',a')\in \sR^n\t \cP_2(\sR^n \t \sR^{k})\t \bA$,
\begin{align}\l{eq:mfcE_H_convex}
\begin{split}
&H(t,x,a,\eta,y,z)-H(t,x',a',\eta',y,z)-\la \p_{(x,\a)}H(t,x,a,\eta,y,z), (x-x',a-a')\ra 
\\
&\q 
-\tilde{\sE}[\la\p_\mu H(t,x,a,\eta,y,z)(\tilde{X},\tilde{\a}),\tilde{X}-\tilde{X}'\ra  + \la\p_\nu H(t,x,a,\eta,y,z)(\tilde{X},\tilde{\a}),\tilde{\a}-\tilde{\a}'\ra ] \\
& \le -\lambda_1 |a'-a|^2 - \lambda_2 \tilde{\mathbb{E}}[|\tilde{\a}'-\tilde{\a}|^2],
\end{split}
\end{align}
whenever  $(\tilde{X},\tilde{\a}),(\tilde{X}',\tilde{\a}')\in L^2(\tilde{\Om},\tilde{\cF},\tilde{\sP};\sR^n\t\sR^k)$
with distributions $\eta$ and $\eta'$, respectively.
Moreover, 
 the same arguments as in 
\cite[Theorem 3.5]{acciaio2019} give us that 
\begin{align*}
J(\hat{\a};\xi_0)-J(\a;\xi_0) 
& \leq
\sE\bigg[\int_0^T
\Big(H(t,X^{\hat{\a}}_t,\hat{\alpha}_t,\sP_{(X^{\hat{\a}}_t,\hat{\alpha}_t)},Y^{\hat{\a}}_t,Z^{\hat{\a}}_t)
-H(\a_t, X^{{\a}}_t,{\alpha}_t,\sP_{(X^{{\a}}_t,{\alpha}_t)},Y^{\hat{\a}}_t,Z^{\hat{\a}}_t)\,\d t
\Big)
\bigg]
\\
&\quad 
-\sE\bigg[\int_0^T
\la
\p_x H(t,X^{\hat{\a}}_t,\hat{\alpha}_t,\sP_{(X^{\hat{\a}}_t,\hat{\alpha}_t)},Y^{\hat{\a}}_t,Z^{\hat{\a}}_t),
X^{\hat{\a}}_t-X^{{\a}}_t
\ra \,\d t
\bigg]
\\
&\quad 
-\sE\bigg[\int_0^T
\tilde{\sE}[\la
\p_\mu H(t,X^{\hat{\a}}_t,\hat{\alpha}_t,\sP_{(X^{\hat{\a}}_t,\hat{\alpha}_t)},Y^{\hat{\a}}_t,Z^{\hat{\a}}_t)
(\tilde{X}^{\hat{\a}}_t,\tilde{\hat{\alpha}}_t),
\tilde{X}^{\hat{\a}}_t-\tilde{X}^{{\a}}_t
\ra] \,\d t
\bigg],
\end{align*} 
which along with  \eqref{eq:mfcE_H_convex}  and 
the fact that $\hat{\alpha}$ satisfies  the optimality condition \eqref{eq:opti}
leads to 
\begin{align*}
J(\hat{\a};\xi_0)-J(\a;\xi_0) 
& \leq
\mathbb{E}
\bigg[
 \int_{0}^{T} \la \p_a H(\theta^{\hat{\a}}_t,Y^{\hat{\a}}_t,Z^{\hat{\a}}_t) + \tilde{\mathbb{E}} [\p_\nu  H(\tilde{\theta}^{\hat{\a}}_t,\tilde{Y}^{\hat{\a}}_t,\tilde{Z}^{\hat{\a}}_t)(X^{\hat{\a}}_t,\hat{\alpha}_t)], \hat{\alpha}_t-a_t \ra \, \d t
 \bigg]
\\
&\quad 
- (\lambda_1+\lambda_2)\|\hat{\a}-{\a}\|_{\cH^2}^2 
\\
& \leq -(\lambda_1+\lambda_2)\|\hat{\a}-{\a}\|_{\cH^2}^2.
\end{align*} 
This finishes the proof of the estimate  \eqref{eq:coercive}.
\end{proof}

\begin{proof}[Proof of Theorem \ref{thm:control_PCPT_discrete}]
Recall that 
Step 2 of the proof of Theorem \ref{thm:PCPT_discrete_value} (see \eqref{eq:V-V_pi})
proves that there exists a constant $C>0$, independent of $\xi_0$ and $\pi$, such that for all $\a\in \cA_{\pi}$,
$|J({\a};\xi_0)-J_{\pi}({\a};\xi_0)|\le C(1+\|\xi_0\|^2_{L^2})|\pi|^{1/2}$.
Hence, by  the estimates \eqref{eq:coercive}, for all $\a\in \cA_\pi$ with $J_{\pi}(\a;\xi_0)\le  V_{\pi}(\xi_0)+\eps$,
\begin{align*}
(\lambda_1+\lambda_2)\|\hat{\a}-{\a}\|_{\cH^2}^2 &\le J(\a;\xi_0)-J(\hat{\a};\xi_0)  - J_{\pi}(\a;\xi_0) + J_{\pi}(\a;\xi_0)  \\
& \leq  J(\a;\xi_0)-J(\hat{\a};\xi_0) - J_{\pi}(\a;\xi_0) + V_{\pi}(\xi_0)+\eps  \\
& \leq C(1+\|\xi_0\|^2_{L^{2}})|\pi|^{1/2} + V_{\pi}(\xi_0)-  J(\hat{\a};\xi_0) + \eps\\
& \leq C(1+\|\xi_0\|^2_{L^{2}})|\pi|^{1/2} + \eps,
\end{align*}
where  the last estimate follows from Theorem \ref{thm:PCPT_discrete_value}.
Taking the square root of both sides of the inequality yields the claim.
\end{proof}

\subsection{Proof of Theorem \ref{improvedorder}}\label{MFC1_general}

We start by showing   an approximate sub-dynamic programming principle that will be a key ingredient in the proof of  Theorem \ref{improvedorder}.

\begin{Lemma}\label{lemma1}
Suppose   
(H.\ref{assum:A1}) and (H.\ref{assum:B2}) hold. For all $s\in [0,T], \nu_s \in \mathcal{P}_2(\mathbb{R}^n\times \mathbb{R}^k)$, and $h>0$,
\begin{equation}\label{bound1}
- \big( (\Lx V_\pi^c)(s,\nu_s) +\mathbb{E}\big[f(s,X^{\alpha}_s,\alpha_s,\nu_s)\big] \big) \le Ch,
\end{equation}
where $(X^{\alpha}_s,\alpha_s) \sim \nu_s$, and 
$C>0$ is a    constant     depending   only on the constants in (H.\ref{assum:B2}).
\end{Lemma}
\begin{proof} 
For some $s \in [0,T)$, let $\alpha \in \mathcal{A}_h$ be a piecewise constant policy over $[s,T]$ such that $\alpha_r = \beta_r$ for $r \in [s,s+h)$ and $\alpha_r = \gamma_r$ for $r \in [s+h,T]$, where $(\beta_r)_{r \in [s, s+h)} \in \mathcal{A}_h$, i.e.\ constant, and $\beta_r \sim \pi_2\sharp \nu_r$, the second marginal of $\nu_r$, while we choose $\gamma$ such that it minimizes 
\begin{equation}\label{remark1}
\mathbb{E}\left[\int_{s+h}^Tf(r,X^{\alpha}_r,\alpha_r,\mathbb{P}_{(X^{\alpha}_r,\alpha_r)})\diff r + g(X^{\alpha}_T,\mathbb{P}_{X^{\alpha}_T}) \bigg| \mathcal{F}_{s+h}  \right].
\end{equation}
Notice that $\alpha$ is admissible over $[s,T]$ and the flow property holds: $X^{\alpha}_r = X^{\gamma,X^{\beta}_{s+h} }_r $ for all $r \in [s+h, T]$, where $X^{\gamma,X^{\beta}_{s+h} }_r$ follows the dynamics \eqref{stateprocessX} with initial condition $X^{\beta}_{s+h}$ at time $s+h$, as similarly argued, for example, in \cite[Proposition 6.29]{carmona2018a}. 

We will occasionally write for compactness $\mu_t = \mathbb{P}_{X^{\alpha}_t}$ and $\nu_t = \mathbb{P}_{(X^{\alpha}_t, \alpha_t)}$ for any $t\in [s,T]$.
Now, by definition,
\begin{equation}
\label{dppinequ}
\begin{split}
V_\pi^c(s,\mu_s) &\le \mathbb{E}\left[\int_s^{T}f(r,X^{\alpha}_r,\alpha_r,\mathbb{P}_{(X^{\alpha}_r,\alpha_r)})\diff r +g(X^{\alpha}_T,\mathbb{P}_{X^{\alpha}_T})  \right]\\
& = \mathbb{E}\left[\int_s^{s+h}f(r,X^{\alpha}_r,\alpha_r,\mathbb{P}_{(X^{\alpha}_r,\alpha_r)})\diff r + \mathbb{E}\left[\int_{s+h}^Tf(r,X^{\alpha}_r,\alpha_r,\mathbb{P}_{(X^{\alpha}_r,\alpha_r)})\diff r + g(X^{\alpha}_T,\mathbb{P}_{X^{\alpha}_T}) \big| \mathcal{F}_{s+h} \right]  \right]\\
& = \mathbb{E}\left[\int_s^{s+h}f(r,X^{\beta}_r,\beta_r, \nu_r)\diff r\right] + V_\pi^c(s+h,\mu_{s+h}),
\end{split}
\end{equation}
since for $[s+h, T]$ we choose $\alpha$ such that it minimizes the expression in \eqref{remark1}.  

Recalling that $F(r,\nu_r):= \mathbb{E}[f(r,X^{\alpha}_r,\alpha_r,\nu_r)]$, \eqref{dppinequ} can then be rewritten as
\begin{equation}\label{DPPinequality}
        V_\pi^c(s,\mu_s) \le \int_s^{s+h} F(r,\nu_r)\diff r+ V_\pi^c(s+h, \mu_{s+h}).
\end{equation}
Applying It\^{o}'s formula twice to $V_\pi^c$, we have 
\begin{equation}
    \begin{split}
V_\pi^c(s+h, \mu_{s+h}) &= V_\pi^c(s,\mu_s) + \int_s^{s+h} (\Lx V^c_{\pi})(r,\nu_r) \diff r \\
&=  V_\pi^c(s,\mu_s) + h \cdot  (\Lx V^c_{\pi})(s,\nu_s) + 
\int_s^{s+h}\int_s^r (\Lxa \Lx V^c_{\pi})(\tau,\nu_{\tau}) \diff \tau \diff r,
    \end{split}
\end{equation}
noting that in the first line $\Lx$ is applied to $V^c_{\pi}$, which is a function only of the marginal law $\mu_r = \pi_1\sharp \nu_r$ in the $x$-component, while $\Lx V^c_{\pi}$ is a general function of $\nu_r$ (and not solely of its marginal) by the dependence of $b$ and $\sigma$ on $\nu_r$ in \eqref{defLx};
for the second line, we applied  It{\^o}'s formula  to $(\Lx V^c_{\pi})(s,\nu_s)$, where $\nu_s$ is the joint law of the partly degenerate It{\^o} process $(X,\alpha)$ for which $\alpha$ is (random but) constant over $(s,s+h)$ by the assumption of piecewise constant controls.

Using this and applying It\^{o}'s formula once to $F$, we get by insertion in \eqref{dppinequ},
\begin{equation}
    \begin{split}
     V_\pi^c(s,\mu_s) \le & 
     \ V_\pi^c(s,\mu_s) + h \cdot \left( F(s,\nu_s)  +  (\Lx V^c_{\pi})(s,\nu_s) \right) \\ & + \int_s^{s+h}\int_s^r 
     \left((\Lxa F)(\tau,\nu_{\tau}) 
      + (\Lxa \Lx V^c_{\pi})(\tau,\nu_{\tau}) \right) \diff \tau \diff r.
    \end{split}
\end{equation}
By (H.\ref{assum:B2}), the terms in the second line are of order $O(h^2)$. Rearranging the terms yields the desired conclusion. 
\end{proof}

 \begin{proof}[Proof of Theorem \ref{improvedorder}]
Let $\alpha \in \mathcal{A}$ be an arbitrary control, $s\in [0,T],$ and $\mu_{.}=\mathbb{P}_{X^{\alpha}_{.}}$ the measure flow with $\mu_s \in \mathcal{P}_2(\mathbb{R}^n)$ for all $s$, and similarly  $\nu_s= \mathbb{P}_{(X^{\alpha}_{s}, \alpha_{s})} \in \mathcal{P}_2(\mathbb{R}^n \times \mathbb{R}^k)$. 
 Let $C>0$ be a generic constant that depends only on  the constants in (H.\ref{assum:A1}) and (H.\ref{assum:B2}), and    may vary throughout the analysis.

Applying It\^{o}'s formula to $V_\pi^c$ and using Lemma \ref{lemma1} yield
\begin{equation*}
\begin{split}
    V_\pi^c(T,\mu_{T}) &= V_\pi^c(s,\mu_s)+
    \int_s^{T} (\Lx V_\pi^c)(r,\nu_r) \diff r\\
    &\ge V_\pi^c(s,\mu_s)-\int_s^T\mathbb{E}\big[f(r,X^{\alpha}_r,\alpha_r,\nu_r)\big]\diff r-CTh,    
\end{split}
\end{equation*}
which implies that 
\begin{equation*}
\begin{split}
    J( \alpha; s, \xi) &= \mathbb{E}\Big[\int_s^T f(r,X^{\alpha}_r,\alpha_r,\nu_r)\diff r+g(X^{\alpha}_T,\mu_{T})\Big]\\
    &= \mathbb{E}\Big[\int_s^T f(r,X^{\alpha}_r,\alpha_r,\nu_r)\diff r\Big]+V_\pi^c(T,\mu_{T})\\
    &\ge \mathbb{E}\Big[\int_s^T f(r,X^{\alpha}_r,\alpha_r,\nu_r)\diff r\Big]-\mathbb{E}\Big[\int_s^Tf(r,X^{\alpha}_r,\alpha_r,\nu_r)\diff r\Big]+ V_\pi^c(s,\mu_s) -Ch.    
\end{split}
\end{equation*}
Since the above is true for any $\alpha \in \mathcal{A},$ taking the infimum over  $ {\alpha}\in \cA $ shows that for all $s\in [0,T]$,
\begin{equation*}
   V_\pi^c(s,\mu_s)-V(s,\mu_s)\le Ch,
\end{equation*}
which proves the upper bound of \eqref{mainresult1} for all $s \in [0,T]$ and $\mu_s \in \mathcal{P}_2(\mathbb{R}^n)$. The lower bound is a direct consequence of $\mathcal{A}_h \subseteq \mathcal{A},$ which concludes the proof.
\end{proof}
 
\section{Numerical experiments}
\label{numericalresults}

In this section, we illustrate the first-order convergence of piecewise constant control approximation using the controlled multi-dimensional stochastic Cucker--Smale (C--S) dynamics, as previously studied in, for instance, \cite{NouCaiMal, carmona2018a, ReiStoZha1}. 

\paragraph{Problem formulation.}
The C--S model was 
introduced  in \cite{CucSma} and describes  the self-organisation  behaviour of a group of agents, such as birds, that interact with each other through alignment and attraction forces. The state of each agent is determined by its position and velocity, and each agent evolves under similar dynamics. In its original uncontrolled formulation, see \cite{CucSma}, each agent updates its velocity based on the weighted velocities of all the other agents, where the weights depend on the distance between the agents. In the controlled   model, the controller aims to either bring consensus on a model that would diverge, or to accelerate the flocking process of an initial arrangement that would self-organise.

We fix a terminal time $T>0$ and consider a $d$-dimensional adapted control strategy $(\alpha_t)_{t\in[0,T]}$ and a $d$-dimensional Brownian motion $W$ on a fixed filtered probability space $(\Omega, \mathcal{F},\{\mathcal{F}_t\}_{t\in[0,T]},\mathbb{P})$. The state of each member evolves according to the following $2d$-dynamics with initial state $(x_0,v_0) \in L^2(\Omega;\mathbb{R}^d \times \mathbb{R}^d)$,
\begin{equation}\label{CS}
    \diff x_t = v_t\diff t, \quad \quad \diff v_t = \bigg(\alpha_t+ \iint_{\mathbb{R}^d \times \mathbb{R}^d} \kappa (x_t,v_t,x',v')\mathcal{L}_{(x_t,v_t)}(\diff x', \diff v')\bigg)\diff t + \sigma \diff W_t,
\end{equation}
for all $t\in [0,T]$, where $\sigma \in \mathbb{R}^{d \times d}$, and $\kappa:\mathbb{R}^{d} \times \mathbb{R}^{d} \times \mathbb{R}^{d} \times \mathbb{R}^{d} \to \mathbb{R}^{d}$ is defined by 
\begin{equation*}
  \kappa (x,v,x',v') = \frac{\Phi(v'-v)}{(1+|x'-x|^2)^{\beta}}, \text{ for some }\Phi,\beta \ge 0. 
\end{equation*}

The communication rate $\beta$ captures the rate of decay of the influence between agents that are positioned further apart. 
In this setting, the controller is assumed to minimize the cost functional $J$ over all adapted processes $\alpha \in \mathbb{R}^d$,
\begin{equation}\label{CSvaluefunction}
 J(\alpha;(x_0,v_0)) = \mathbb{E}\Big[\int_0^T \big( |v_t-\mathbb{E}[v_t]|^2+\gamma_1 |\alpha_t|^2\big)\diff t+|v_T-\mathbb{E}[v_T]|^2\Big],
\end{equation}
for given constant $\gamma_1 \ge 0.$ 

\paragraph{Experiment setup.}
 We now show numerical experiments for different $\beta \in \{0,1\}$, in which we fix $T=1$, $\Phi=1$, $\sigma=\gamma_1=0.1$, and the initial states $x_0$ and $v_0$ are uniformly distributed on the interval $[0,1)$. 

 We approximate the solution to the MFC problem by applying a policy gradient method to a neural network parametrisation
of the feedback map from the state $(x_t,v_t)$ to the control $\alpha_t$. The law of the process in the coefficients and objective is approximated by the empirical measure of a particle system. 

Specifically, 
using the same notation as in Algorithm 1 of \cite{carmona2019}, let $\mathbf{N}_{d_0,d_1,...,d_{l+1}}^{\psi}$ be the set of neural networks $\phi$ such that $\phi = \phi_l \circ \phi_{l-1} \circ ... \circ \phi_{0}$, with layer functions $\phi_{i}:\mathbb{R}^{i}\to \mathbb{R}^{i+1}$, $\phi_i(x)_k = \psi\big(\beta_k + \sum_{j=1}^{d_i}\omega_{k,j}x_j\big)$ for all $k\in \{1,...,i+1\}$, $\beta \in \mathbb{R}^{d_{i+1}},\omega \in \mathbb{R}^{d_{i+1}\times d_{i}}$, for $i \in \{0,...,l-1\}$ and an activation function $\psi:\mathbb{R}\to \mathbb{R}$. For the output layer, we use the identity $\psi(x)= x$ as the activation function. The network architecture consists of the number of hidden layers $l$, the number of units per layer, i.e., $d_0,d_1,...,d_{l+1}$, and the activation function $\psi$. Fixing the network architecture, the parameters of the network that we train are $\theta := (\beta_0,\omega_0,\beta_1,\omega_1,...,\beta_l,\omega_l)$. We denote this set by $\Theta$, so that $\theta \in \Theta.$ We also denote by $\phi_{\theta}$ the neural network $\phi$ computed using parameters $\theta$. At each time--step of the discretization, we keep the architecture of the network fixed. Specifically, we use $l=2$ hidden layers and set the dimensions of the hidden and output layers to $110$ and $d$ respectively. The dimension of the input layer is $d+1$ in the case when $\beta=0$, and $2d+1$ otherwise. To find the optimal weights of the network, we run the Adam algorithm \cite{KinBa}, with an initial learning rate equal to $0.001,$ and a scheduler that decreases the learning rate $\lambda$ by a factor of $0.617$ every $50$ epochs. At each time-step of the discretization, we approximate the joint law of the state and velocity processes by taking the empirical distribution of $N$ agents. 

Having parameterised the controls as neural networks, the optimization problem reduces to minimizing the following quantity
\begin{equation}\label{costnn}
\begin{split}
    J^{N}_{\phi}&=\frac{1}{N} \sum_{i=1}^{N}\bigg(\sum_{m=0}^{M-1}h \left(|v^i_{t_m}-\mathbb{E}^{\mu_{t_m}}[v_{t_m}]|^2+\gamma_1|\phi(t_m,(x^i_{t_m},v^i_{t_m}))|^2\right)+|v^i_{t_M}-\mathbb{E}^{\mu_{t_m}}[v_{t_M}]|^2\bigg),      
\end{split}
\end{equation}
over $\phi \in \mathbf{N}_{d_0,d_1,...,d_{l+1}}^{\psi}$, under the constraint
\begin{equation}\label{Eulerscheme}
\begin{split}
&x^{i}_{t_m+1} = x^{i}_{t_m} + hv^{i}_{t_m},\\
&v^{i}_{t_m+1} = v^{i}_{t_m}+h\bigg(\phi(t_m,(x^i_{t_m},v^i_{t_m}))+\frac{1}{N}\sum^{N}_{j=1} \frac{K(v^{j}_{t_m}-v^{i}_{t_m})}{(1+|x^{j}_{t_m}-x^{i}_{t_m}|^2)^{\beta}}\bigg)+\sigma\Delta W^i_m,       
\end{split}
\end{equation}
where $\mu_{t_m} = \frac{1}{N}\sum_{i=1}^{N}\delta_{v^{i}_{t_m}}, (x^{i}_{0})_{i \in \{1,...,N\}}$ and $(v^{i}_{0})_{i \in \{1,...,N\}}$ are i.i.d. with uniform distribution on $[0,1)$ and $(\Delta W^i_m)_{i,m}$ are i.i.d. random variables with distribution $\mathcal{N}(0,h)$, for all  $m \in \{0,...,M-1\}$ and $i\in \{1,...,N\}$. 

Here we summarize the policy gradient method for solving \eqref{CSvaluefunction}.

\begin{algorithm}
  \caption{Policy gradient method for controlled C-S models}
  \label{alg:PGalgorithm}
  \begin{algorithmic}[1]
  \STATE \textbf{Input:} Initialise weights $\theta_0 \in \Theta$,    the number of iterations $K$, 
  the number of agents  $N$,
  the number of timesteps  $M$. 
      \FOR{$k = 0,1,...,K-1$}
      \STATE Generate samples $\{x^i_{t_m},v^i_{t_m}\mid  1\le i\le N, 0\le  m \le M \}$ by \eqref{Eulerscheme} using policy $\phi_{\theta_k}$.
        \STATE Compute $\nabla J^{N}_{{\phi}_{\theta}}(\theta_k)$ for $J^{N}_{{\phi}_{\theta}}(\theta_k)$ given in \eqref{costnn}.
        \STATE Obtain $\theta_{k+1}$ with one Adam update.  
      \ENDFOR
      \STATE \textbf{return} $\theta_K$
  \end{algorithmic}
\end{algorithm} 

 \paragraph{Numerical results.}
We first set $\beta=0$ and $\gamma_1=0.1$, in the C-S dynamics \eqref{CS}--\eqref{CSvaluefunction}. The problem then becomes a $d$--dimensional Linear Quadratic (LQ) problem (as there is no $x$-dependence in the $v$ process) that can be solved explicitly via solving a Riccati ODE, see, e.g., \cite{JYong}, \cite{BasPha}, and references therein. 
The optimal feedback control can be found in \cite{ReiStoZha1}, Section 5.2.1, to be 
\begin{equation}\label{optimalcontrol}
    \alpha^*(t,x,v)= -\frac{\nu_t}{2\gamma_1}(v-\mathbb{E}[v^*_t]), \quad (t,x,v)\in [0,T]\times \mathbb{R}\times \mathbb{R},
\end{equation}
where $\nu:[0,T]\to \mathbb{R}$ solves the Riccati equation $\nu'_t -2\Phi \nu_t -\frac{1}{2\gamma_1}\nu^2_t+2=0,\, \nu_T=2,$ and $v^*$ the optimal velocity process. Fixing $N=10^3, d=1, M=128, K=800$, and using this explicit form of the optimal control, we compute the true value function to be $0.0236$ (to three correct digits). We use this as a benchmark and compare it with the one obtained when solving the problem numerically using Algorithm \ref{alg:PGalgorithm} above. In Figure \ref{fig:VF_convergence}, we observe that the approximated value function converges to $0.0239$ over the gradient iterations, close to the true solution. In this setting, we also illustrate the convergence in time of the discretized value function. Specifically, we measure the error in the approximation of the true solution when using Algorithm \ref{alg:PGalgorithm} above with different timesteps $\{4,8,16,32,64,128\}$. The log--log plot in Figure \ref{fig:time_convergence} shows a rate of order $1$ in time, which agrees with our findings in Theorem \ref{improvedorder}.

\begin{figure}[htbp]
  \centering
  \begin{minipage}[b]{0.45\textwidth} 
    \centering
    \includegraphics[width=\textwidth]{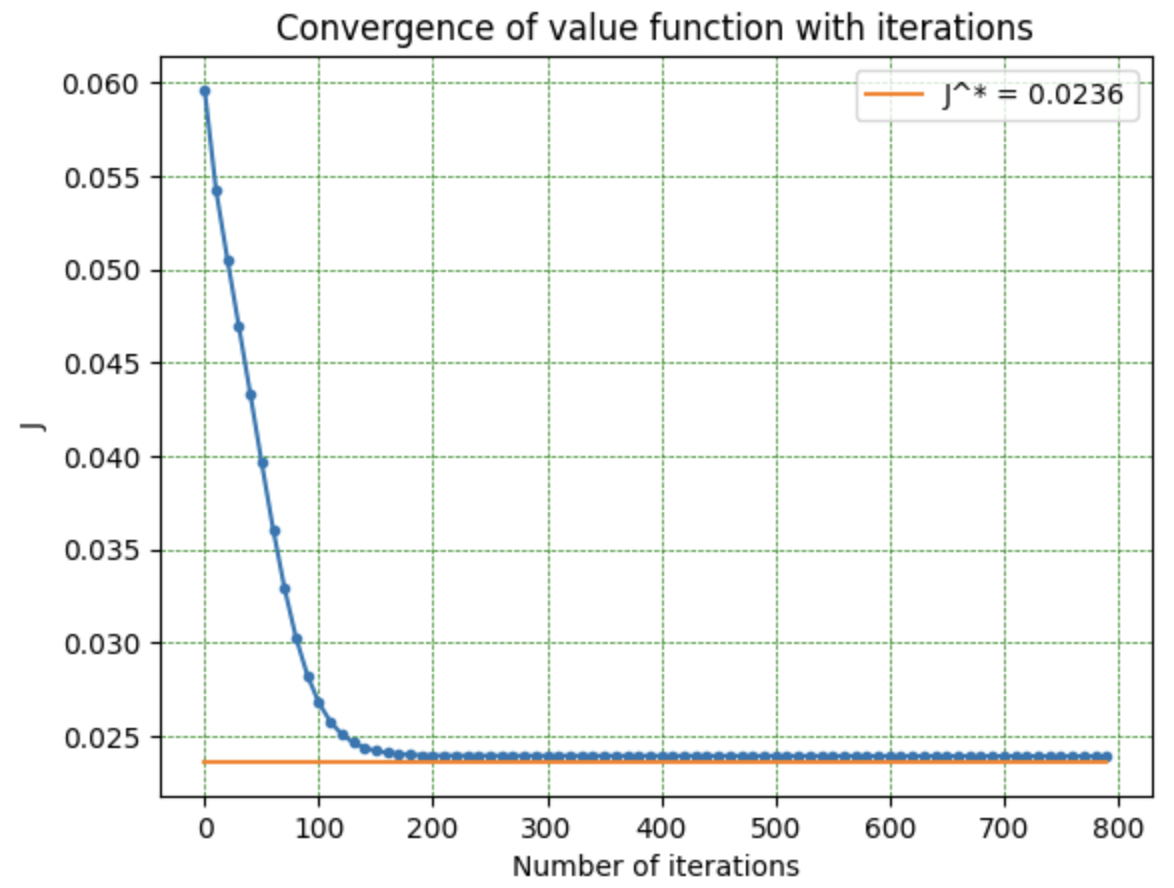}
    \caption{Value function convergence ($\beta=0,\gamma_1=0.1,d=1$)}
    \label{fig:VF_convergence}
  \end{minipage}
  \hfill
  \begin{minipage}[b]{0.45\textwidth} 
    \centering
    \includegraphics[width=\textwidth]{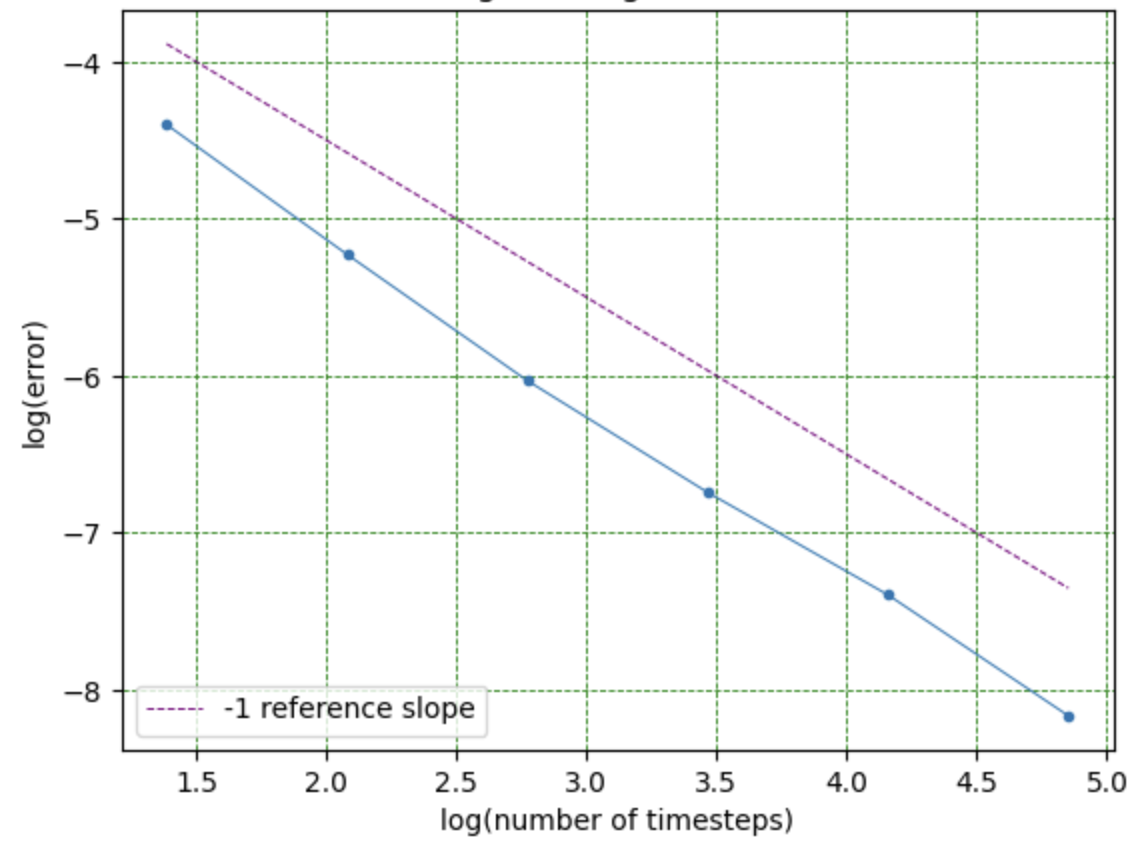}
    \caption{Convergence in time ($\beta=0,\gamma_1=0.1,d=1$)}
    \label{fig:time_convergence}
  \end{minipage}
\end{figure}

We now test how the PG method behaves in the higher-dimensional case when $d=3$, keeping the other parameters the same as above. Computing the optimal control via the Riccati ODE \eqref{optimalcontrol}, we find the true value function to be equal to $0.0664$. In Figure \ref{fig:3dimVF}, the approximated value function converges to $0.0669$, again close to the true value. 
Proceding for Figure \ref{fig:3dim}
analogously as for the $1$-dimensional case, shows again order $1$ convergence in time.
\begin{figure}[H]
  \centering
  \begin{minipage}[b]{0.45\textwidth} 
    \centering
\includegraphics[width=\textwidth]{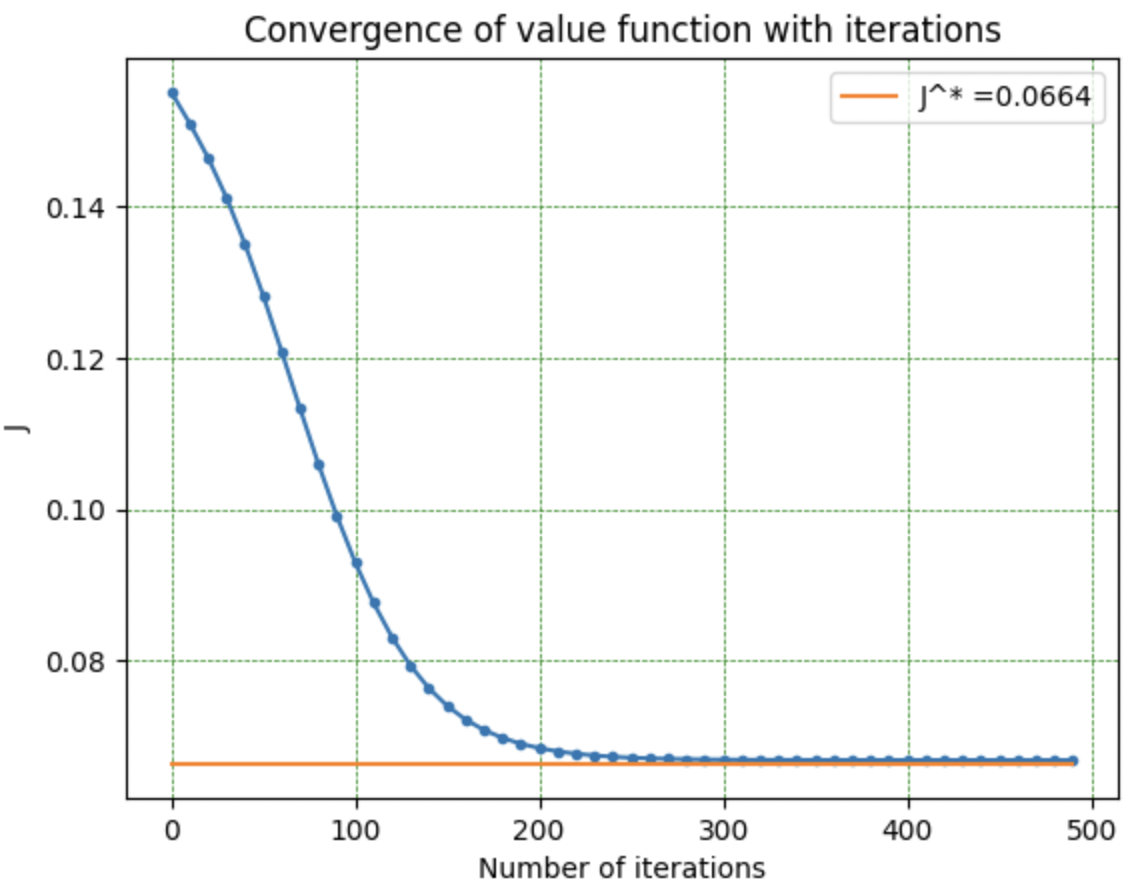}
    \caption{Value function convergence with iterations ($\beta=0$, $\gamma_1 = 0.1$, $d=3$)}
    \label{fig:3dimVF}
  \end{minipage}
  \hfill
  \begin{minipage}[b]{0.45\textwidth} 
    \centering
 \includegraphics[width=\textwidth]{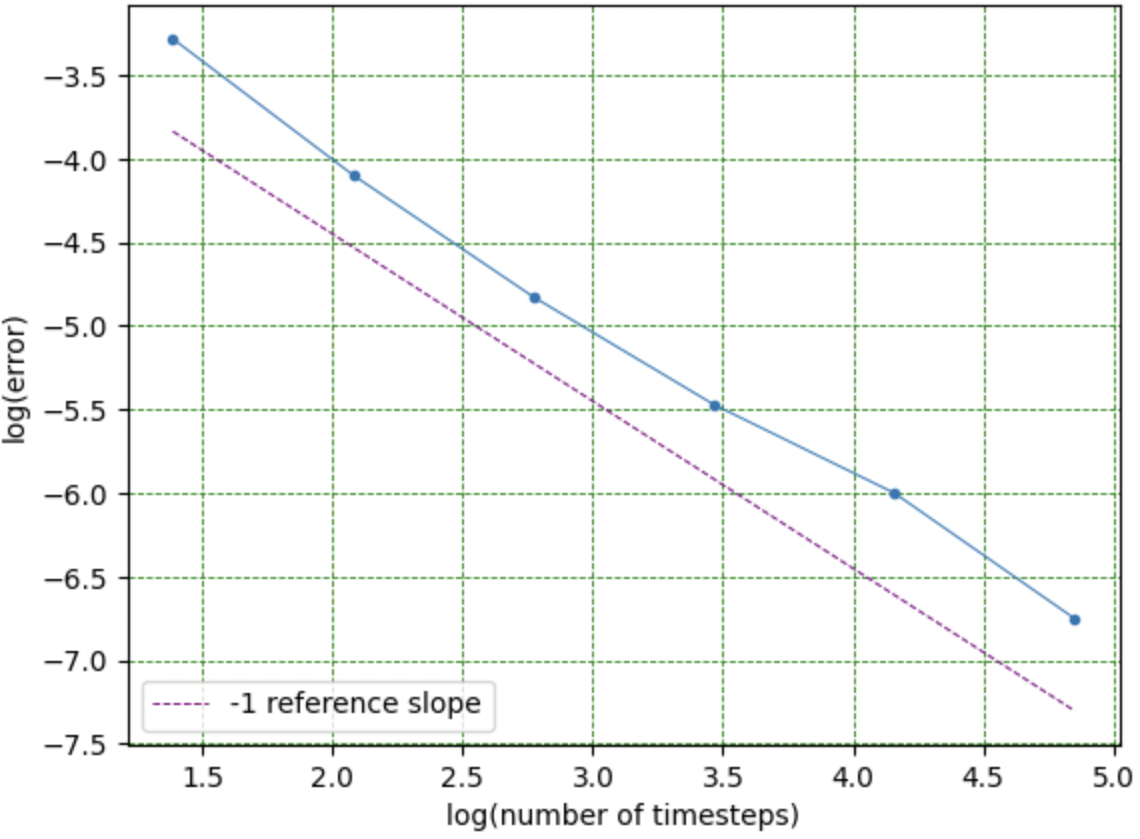}
    \caption{Convergence in time ($\beta=0$, $\gamma_1 = 0.1$, $d=3$)}
    \label{fig:3dim}
  \end{minipage}
\end{figure} 

Finally, we consider $\beta =1$, which gives a non-trivial coupling between $x$ and $v$ and a fully $2d$--dimensional models. We keep the rest of the parameters the same as before, while $K=400$. Recall that $\gamma_1=0.1$, $T=1, N= \mathcal{O}(10^3), d=1, M=128$, and $\sigma =0.1$. This is not a LQ model and therefore we do not have an explicit solution to the optimal feedback control. Nonetheless, in Figure \ref{fig:beta1VF}, we observe that the PG method appears to converge, to a value around 0.0229. The convergence in the time-step in Figure \ref{fig:beta1_t} is again consistent with first order.

\begin{figure}[H]

  \centering
  \begin{minipage}[b]{0.45\textwidth} 
    \centering
\includegraphics[width=\textwidth]{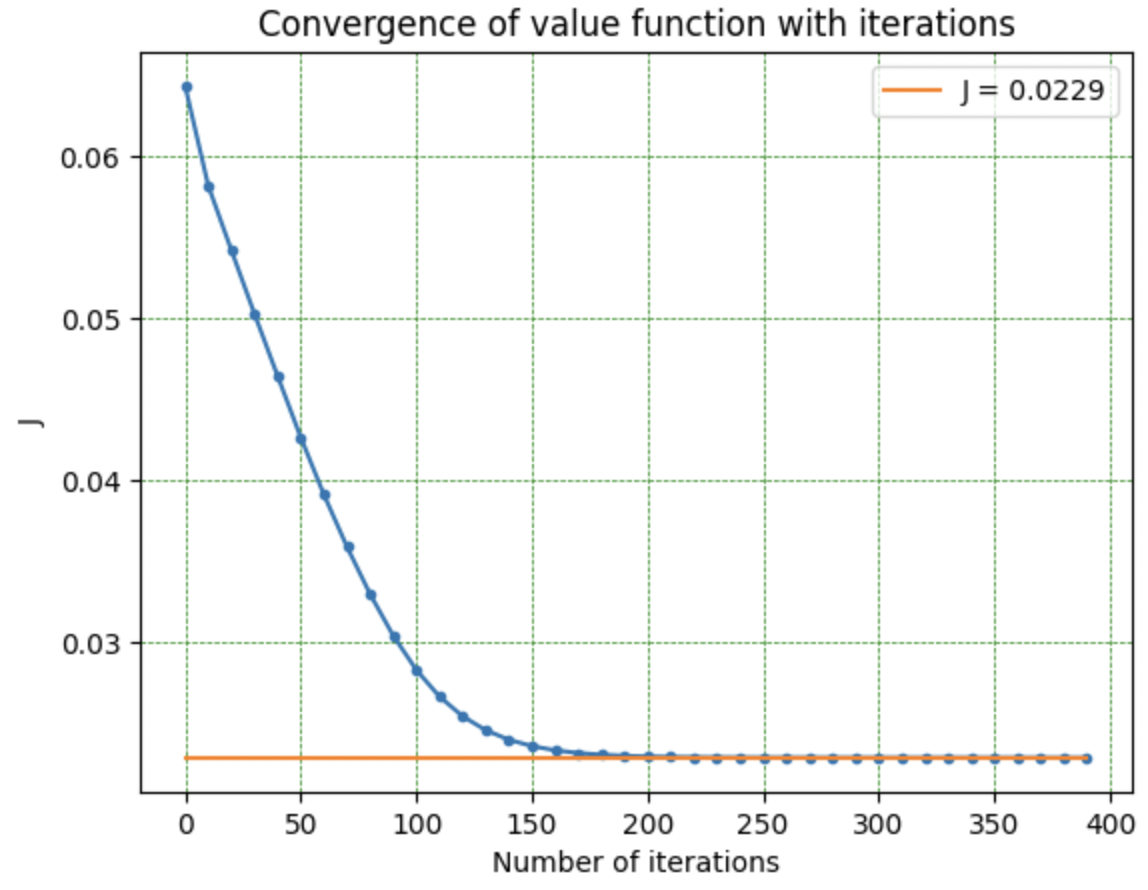}
    \caption{Value function convergence  with iterations ($\beta=1$, $\gamma_1 = 0.1$, $d=1$)}
    \label{fig:beta1VF}
  \end{minipage}
  \hfill
  \begin{minipage}[b]{0.45\textwidth} 
    \centering
 \includegraphics[width=\textwidth]{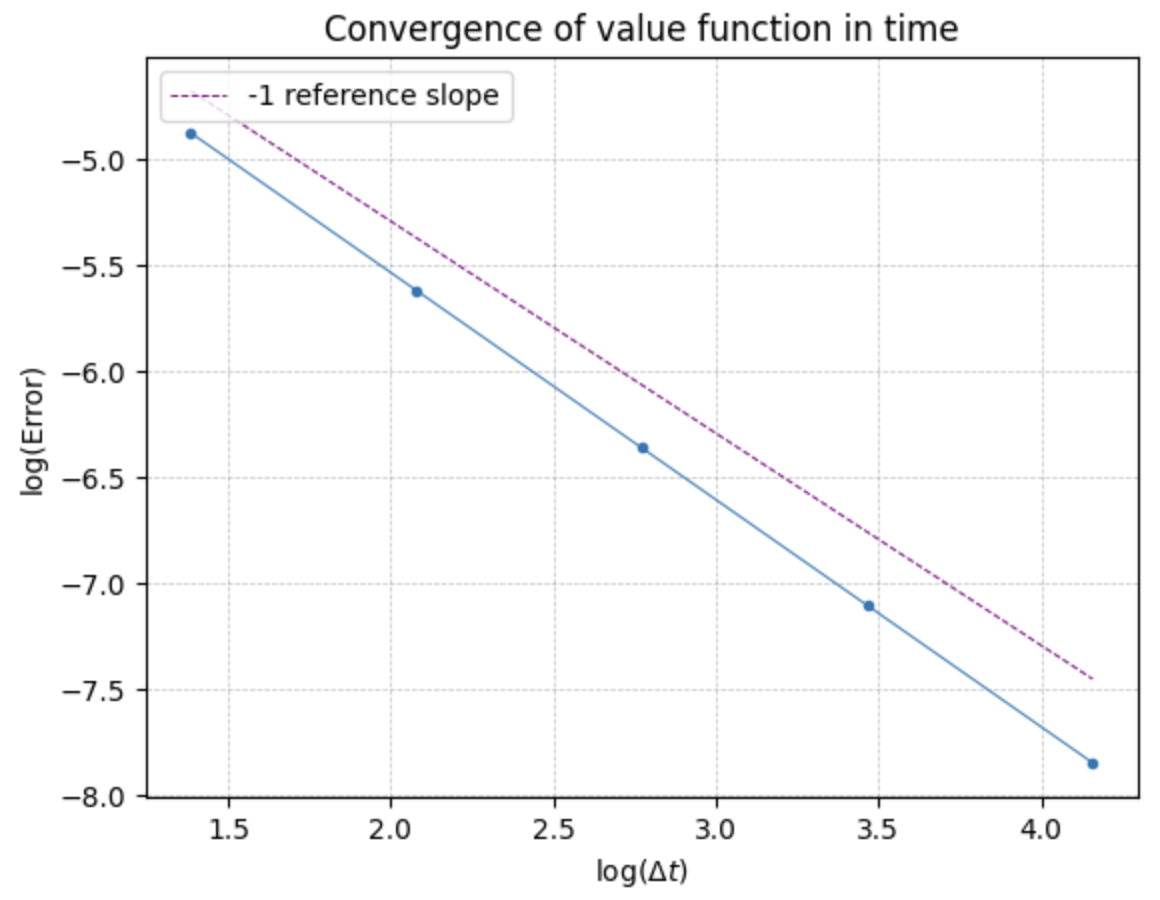}
    \caption{Convergence in time ($\beta=1$, $\gamma_1 = 0.1$, $d=1$)}
    \label{fig:beta1_t}
  \end{minipage}
\end{figure}

\appendix

\section{Proofs of Propositions \ref{prop:mcfE1} and  \ref{prop:mcfE2} and \ref{prop:fbsde_wellposedness}}
\l{appendix:hat_a_existence}

To prove Proposition \ref{prop:mcfE1},
we recall  
the following Kantorovich duality theorem,
which 
follows as a special case of \cite[Theorem 5.10]{Villani}.

\begin{Lemma}\label{lem:Kant}
Let $(\mathcal{X},\mu)$ and $(\mathcal{Y},\nu)$ be two Polish probability spaces and let $\omega: \mathcal{X} \times \mathcal{Y} \to [0,\infty)$ be a continuous function. Then
we have that
\begin{align*}
\inf_{\kappa \in \Pi(\mu,\nu)} \int_{\mathcal{X} \times \mathcal{Y}} \omega(x,y) \, \mathrm{d} \kappa(x,y) =
 \sup_{
 \substack{
 (\psi,\varphi)\in C_b(\cX)\t C_b(\cY),
 \\\varphi-\psi\le \om}}
   \left( \int_{\mathcal{Y}} \varphi(y)  \, \mathrm{d}\nu(y)  - \int_{\mathcal{X}} \psi(x)  \, \mathrm{d}\mu(x)  \right),
\end{align*}
where $\Pi(\mu,\nu)$ is the set of all couplings of $\mu$ and $\nu$,
and $ C_b(\cX)$ (resp.~$C_b(\cY)$) is the space of bounded continuous functions
$\cX\to \sR$ (resp.~$\cY\to \sR$).
\end{Lemma}

\begin{proof}[Proof of Proposition \ref{prop:mcfE1}]
We first show that the functions 
$(\hat{b},\sigma,\hat{f},\hat{g})$ 
satisfy the Lipschitz continuity.
The claim  obviously holds for the function $\sigma$ due to (H.\ref{assum:mfcE}(\ref{item:mfcE_lin})).
To show the Lipschitz continuity of $\hat{g}$,
 for any $(x,\mu), (x',\mu') \in \sR^n \t \cP_2(\sR^n)$ and any coupling $\kappa$ of $\mu$ and $\mu'$
(i.e., $\kappa\in \Pi(\mu,\mu')$),
we observe from (H.\ref{assum:mfcE}(\ref{item:mfcE_growth}))   that
\begin{align}\l{eq:g_hat_lip}
\begin{split}
&|\hat{g}(x,\mu)-\hat{g}(x',\mu')|^2 
\\
& \leq 2|\p_x g(x,\mu)- \p_x g(x',\mu')|^2  
 +  2 \left| \int_{\sR^n}\p_\mu g(\tilde{x},\mu)(x)\,\d \mu(\tilde{x}) - \int_{\sR^n}\p_\mu g(\tilde{x}',\mu')(x')\,\d \mu'(\tilde{x}')  \right|^2 \\
& \leq 
C (|x-x'|^2
+\cW^2_2(\mu,\mu'))
 + 2 \left( \int_{\sR^n \t \sR^n } |\p_\mu g(\tilde{x},\mu)(x) - \p_\mu g(\tilde{x}',\mu')(x')| \,\d \kappa(\tilde{x},\tilde{x}') \right)^2
\\
& \leq C 
\bigg\{
|x-x'|^2 + \cW^2_2(\mu,\mu')+ \left( \int_{\sR^n \t \sR^n } 
\Big(
|\tilde{x}-\tilde{x}'|+|x-x'|+\cW_2(\mu,\mu')
\Big) \,\d \kappa(\tilde{x},\tilde{x}') \right)^2
\bigg\},
\end{split}
\end{align}
where $C>0$ depends on the Lipschitz constant  in (H.\ref{assum:mfcE}(\ref{item:mfcE_growth})). 
Then, by applying Jensen's inequality to the above estimate
and taking the infimum   over all  $\kappa\in \Pi(\mu,\mu')$,
we can deduce that 
$|\hat{g}(x,\mu)-\hat{g}(x',\mu')| \leq C (|x-x'| + \cW_2(\mu,\mu'))$.

Before proceeding to show the Lipschitz continuity of $\hat{b}$ and $\hat{f}$, we first establish the Lipschitz continuity of $\phi(t,\chi)$ defined as in (H.\ref{assum:mfcE_hat}(\ref{item:ex})).
Let $\chi,\chi' \in \cP_2(\sR^{n} \t \sR^{n})$ be given. By applying Lemma \ref{lem:Kant} with $\mathcal{X}=\cY=\sR^{n} \t \sR^{k}$, $\nu=\phi(t,\chi)$, $\mu=\phi(t,\chi')$ and 
the function
$\omega\big((x',y'),(x,y)\big) \coloneqq |x-x'|^2 + |y-y'|^2$ for any $(x,y)\in \cY, (x',y')\in \cX$, 
we can obtain from the definition of $\phi$ that 
\begin{align*}
& \cW_2^{2}(\phi(t,\chi),\phi(t,\chi'))  \\
& = \sup \left( \int_{\sR^{n} \t \sR^{k}} h_1(x,y) \, \mathrm{d}\phi(t,\chi) (x,y) - \int_{\sR^{n} \t \sR^{k}} h_2(x',y') \, \mathrm{d}\phi(t,\chi') (x',y') \right)  \\
& =  \sup \left( \int_{\sR^{n} \t \sR^{n}} h_1(x,\hat{\a}(t,x,y,\chi)) \, \mathrm{d} \chi(x,y) - \int_{\sR^{n} \t \sR^{n}} h_2(x',\hat{\a}(t,x',y',\chi')) \, \mathrm{d} \chi'(x',y') \right),
\end{align*}
where the supremum is taken over all bounded continuous functions $h_1, h_2: \sR^{n} \t \sR^{k} \to \sR$ satisfying $h_1(x,y)- h_2(x',y') \leq |x-x'|^2 + |y-y'|^2$ for any $(x,y), (x',y') \in \sR^{n} \t \sR^{k}$. Note that for any given 
such functions $h_1,h_2$,  the Lipschitz continuity of $\hat{\a}$ in 
  (H.\ref{assum:mfcE_hat}(\ref{item:ex})) implies that
\begin{align*}
\begin{split}
&h_1(x,\hat{\a}(t,x,y,\chi))  - h_2(x',\hat{\a}(t,x',y',\chi')) 
\\
&\leq  (3L^2_{\a}+1) \left( |x-x'|^2 + |y-y'|^2 + \cW_2^{2}(\chi,\chi') \right)
\coloneqq \omega_2\big((x,y),(x',y')\big).
\end{split}
\end{align*}
Hence, another application of Lemma \ref{lem:Kant} with 
$\mathcal{X}=\cY=\sR^{n} \t \sR^{n}$, $\nu=\chi$, $\mu=\chi'$ and $\om=\om_2$ gives us that
\begin{align*}
\cW_2^{2}(\phi(t,\chi),\phi(t,\chi'))
&\le
\sup \left( \int_{\sR^{n} \t \sR^{n}} \tilde{h}_1(x,y) \, \mathrm{d} \chi(x,y) - \int_{\sR^{n} \t \sR^{n}} \tilde{h}_2(x',y') \, \mathrm{d} \chi'(x',y') \right) \\
& = \inf_{\kappa  \in \Pi(\chi',\chi)}  \int_{(\sR^n \t \sR^{n}) \t (\sR^n \t \sR^{n})}  \omega_2(x,y)
\, \mathrm{d} \kappa(x,y),
\end{align*}
where the supremum is 
 taken over all bounded continuous functions $\tilde{h}_1, \tilde{h}_2: \sR^{n} \t \sR^{n} \to\sR$ satisfying
 $\tilde{h}_1-\tilde{h}_2\le \om_2$. Thus, we readily deduce from the above estimate that
\begin{align}\label{eq:PF}
\cW_2(\phi(t,\chi),\phi(t,\chi')) \leq C \cW_2(\chi,\chi'),
\end{align}
with a constant  $C>0$ depending only on $L_{\a}$. 

Now  for any $(x,y,\chi), (x',y',\chi') \in \sR^n\t \sR^n \t \cP_2(\sR^n \t \sR^{n})$, 
we can obtain from \eqref{eq:mfc_coefficients2},  (H.\ref{assum:mfcE}(\ref{item:mfcE_lin})) and the Lipschitz continuity of $\hat{\a}$ in (H.\ref{assum:mfcE_hat}(\ref{item:ex}))  that
\begin{align*}
 &|\hat{b}(t,x,y,\chi) - \hat{b}(t,x',y',\chi')| 
 \\
 &= |b(t,x,\hat{\a}(t,x,y,\chi),\phi(t,\chi)) - b(t,x',\hat{\a}(t,x',y',\chi'),\phi(t,\chi'))| \\
& \leq C \left( |x-x'| + |y-y'| + \cW_2(\chi,\chi') \right),
\end{align*}
which shows the Lipschitz continuity of $\hat{b}$.
Finally, we shall establish the Lipschitz continuity of $\hat{f}$. 
Observe  that $\p_x H(t,\cdot)$  and $\p_\mu H(t,\cdot)(\cdot,\cdot)$ are Lipschitz continuous (uniformly in $t$), 
which follows from the definition  of the Hamiltonian $H$ in \eqref{eq:mfcE_hamiltonian} and 
(H.\ref{assum:mfcE}(\ref{item:mfcE_lin})(\ref{item:mfcE_growth})). Now, for any $(x,y,z,\rho), (x',y',z',\rho') \in \sR^n \t \sR^n \t \sR^{n \times d} \t \cP_2(\sR^n \t \sR^{n} \t  \sR^{n \t d})$,
let $\chi=\pi_{1,2} \sharp \rho$ (resp.~$\chi'=\pi_{1,2} \sharp \rho'$)
  the marginal of the measure $\rho$ (resp.~$\rho'$) on $\sR^n\t\sR^n$. Then,
  we can obtain from the definition of $\hat{f}$ in \eqref{eq:mfc_coefficients2} that 
\begin{align*}
& |\hat{f}(t,x,y,z,\rho) - \hat{f}(t,x',y',z',\rho')| \\
& \leq  | \p_x H(t,x,\hat{\a}(t,x,y,\chi),\phi(t,\chi),y,z) -  \p_x H(t,x',\hat{\a}(t,x',y',\chi'),\phi(t,\chi'),y',z')|  \\
& \quad +   \Bigg |\int_{\sR^n\t \sR^n\t \sR^{n\t d}}
\p_\mu H(t,\tilde{x},\hat{\a}(t,\tilde{x},\tilde{y},\chi),\phi(t,\chi),\tilde{y},\tilde{z})
(x,\hat{\a}(t,{x},{y},\chi))\,\d \rho(\tilde{x},\tilde{y},\tilde{z}) \\
& \quad - \int_{\sR^n\t \sR^n\t \sR^{n\t d}}
\p_\mu H(t,\tilde{x}',\hat{\a}(t,\tilde{x}',\tilde{y}',\chi'),\phi(t,\chi'),\tilde{y}',\tilde{z}')
(x',\hat{\a}(t,{x}',{y}',\chi'))\,\d \rho'(\tilde{x}',\tilde{y}',\tilde{z}') \Bigg| \\
& \coloneqq\Sigma_1 + \Sigma_2. 
\end{align*}
Note that one can easily deduce from Lemma \ref{lem:Kant}
that 
$\cW_2(\pi_{1,2} \sharp  \rho, \pi_{1,2} \sharp \rho') \leq \cW_2(\rho,\rho')$.
Then, by using 
the uniform Lipschitz continuity of $\p_x H(t,\cdot)$,
 (H.\ref{assum:mfcE_hat}(\ref{item:ex})) and
 (\ref{eq:PF}),
 we have the estimate that 
$\Sigma_1 \leq C \left( |x-x'|+ |y-y'|+ |z-z'| + \cW_2(\rho,\rho') \right)$.
 Furthermore, 
by using the same manipulations as in \eqref{eq:g_hat_lip} with  an arbitrary coupling of $\rho$ and  $\rho'$,
and employing  the Lipschitz continuity of $\p_\mu H(t,\cdot)(\cdot,\cdot)$
 along with (H.\ref{assum:mfcE_hat}(\ref{item:ex})) and (\ref{eq:PF}), 
 we can 
conclude the same upper bound for $\Sigma_2$, 
which leads to the desired Lipschitz continuity of $\hat{f}$.

It remains to show that the functions 
$(\hat{b},\sigma,\hat{f})$ 
satisfy the integrability conditions.
We can clearly see from
(H.\ref{assum:mfcE}(\ref{item:mfcE_lin}))
 that 
$\|\sigma(\cdot,0,\bm{\delta}_{0_{n}})\|_{L^\infty(0,T)}=
\|\sigma_0\|_{L^\infty(0,T)}<\infty$.
Moreover, 
 \eqref{eq:PushF} 
 and  (H.\ref{assum:mfcE_hat}(\ref{item:ex})) 
 imply that 
 $\|\phi(t,\bm{\delta}_{0_{n+n}})\|_2=\|\bm{\delta}_{(0,\hat{\a}(t,0,0,\bm{\delta}_{0_{n+n}}))}\|_2 
\le \|\hat{\a}(\cdot,0,0,\bm{\delta}_{0_{n+n}})\|_{L^\infty(0,T)}<\infty$
for all $t\in [0,T]$.
Hence, we can obtain from \eqref{eq:mfc_coefficients2}  and
 (H.\ref{assum:mfcE}(\ref{item:mfcE_lin})(\ref{item:mfcE_growth}))  
that
$ \|\hat{b}(\cdot,0,0,\bm{\delta}_{0_{n+n}})\|_{L^2(0,T)}+
 \| \hat{f}(\cdot,0,0,0,\bm{\delta}_{0_{n+n+nd}})\|_{L^\infty(0,T)}
 <\infty$,
 which completes the proof.
\end{proof}

\begin{proof}[Proof of Proposition \ref{prop:mcfE2}]
Throughout this proof,
let $t\in [0,T]$, 
for all $i\in \{1,2\}$
let 
$\Theta_i=(X_i,Y_i,Z_i)\in L^2(\Om;\sR^n\t \sR^n\t\sR^{n\t d})$
be a given random variable
and
 $\a_i=\hat{\a}(t,X_i,Y_i,\sP_{(X_i,Y_i)})$.

Let $(\tilde{X}_i,\tilde{Y}_i, \tilde{Z}_i)_{i=1}^2$ be an independent copy of 
$({X}_i,{Y}_i, {Z}_i)_{i=1}^2$ defined on the space 
$L^2(\tilde{\Om},\tilde{\cF},\tilde{\sP})$.
By applying the  convexity of $g$ in  (H.\ref{assum:mfcE}(\ref{item:mfcE_convex}))
with $(x',\mu')=(X_1(\om),\sP_{X_1})$, $(x,\mu)=(X_2(\om),\sP_{X_2})$ for each $\om$,
taking the expectation with respect to the measure $\sP$
and then exchanging the role of $X_1$ and $X_2$ in the estimates,
we obtain the desired monotonicity property of $\hat{g}$ 
in \eqref{eq:monotonicity} as follows:
\begin{align*}
0&\leq 
\sE[\la \p_x g(X_1,\sP_{X_1})-\p_x g(X_2,\sP_{X_2}),X_1-X_2 \ra]
\\
& \q +
\sE\big[ \tilde{\sE}[
\la \p_\mu g({X}_1,\sP_{X_1})(\tilde{X}_1)-\p_\mu g({X}_2,\sP_{X_1})(\tilde{X}_2),
\tilde{X}_1-\tilde{X}_2 \ra]\big]
\\
&=\sE[\la \hat{g}(X_1,\sP_{X_1})-\hat{g}(X_2,\sP_{{X}_2}), X_1-X_2\ra],
\end{align*}
where for the last equality we have used Fubini's theorem and the fact that 
$\sP_{{X}_i}=\tilde{\sP}_{\tilde{X}_i}$ for $i=1,2$.

To show monotonicity of $\hat{f}$, 
we first deduce from 
the definition of $\hat{b}$ (see \eqref{eq:mfc_coefficients2})
and
the linearity of $H$ in $(y,z)$ (see \eqref{eq:mfcE_hamiltonian})
that
\begin{align*}
\begin{split}
&\la \hat{b}(t,X_1,Y_1,\sP_{(X_1,Y_1)})-\hat{b}(t,X_2,Y_2,\sP_{(X_2,Y_2)}),  Y_1-Y_2\ra
\\
&\q +\la \sigma(t,X_1,\sP_{X_1})-\sigma(t,X_2,\sP_{X_2}),  Z_1-Z_2\ra
\\
&=\la {b}(t,X_1,\a_1,\sP_{(X_1,\a_1)}),  Y_1-Y_2\ra
+\la \sigma(t,X_1,\sP_{X_1}),  Z_1-Z_2\ra
\\
&\q -
\big({b}(t,X_2,\a_2,\sP_{(X_2,\a_2)}),  Y_1-Y_2\ra
+\la \sigma(t,X_2,\sP_{X_2}), Z_1- Z_2\ra
\big)
\\
&=H(t,X_1,\a_1,\sP_{(X_1,\a_1)},Y_1,Z_1)-H(t,X_1,\a_1,\sP_{(X_1,\a_1)},Y_2,Z_2)
\\
&\q
-\big(H(t,X_2,\a_2,\sP_{(X_2,\a_2)},Y_1,Z_1)-H(t,X_2,\a_2,\sP_{(X_2,\a_2)},Y_2,Z_2)\big).
\end{split}
\end{align*}
Moreover, by 
setting 
$\tilde{\a}_i=\hat{\a}(t,\tilde{X}_i,\tilde{Y}_i,\sP_{(X_i,Y_i)})$ for  all $i=1,2$
and
using the definition of $\hat{f}$ 
in \eqref{eq:mfc_coefficients2},
  we can obtain that
\begin{align*}
\begin{split}
&\sE\big[ \la -\hat{f}(t,\Theta_1,\sP_{\Theta_1})+\hat{f}(t,\Theta_2,\sP_{\Theta_2}), X_1-X_2\ra\big]
\\
&=-\sE\big[\la \p_x H(t,X_1,\a_1,\sP_{(X_1,\a_1)},Y_1,Z_1), {X}_1-{X}_2\ra\big]
\\
&\q
-\sE\big[ \tilde{\sE}[\la
\p_\mu H(t,X_1,\a_1,\sP_{(X_1,\a_1)},Y_1,Z_1)(\tilde{X}_1,\tilde{\a}_1), \tilde{X}_1-\tilde{X}_2\ra]
\big]
\\
&\q
+\sE\big[\la \p_x H(t,X_2,\a_2,\sP_{(X_2,\a_2)},Y_2,Z_2), {X}_1-{X}_2\ra\big]
\\
&\q
+\sE\big[\tilde{\sE}[
\la
\p_\mu H(t,X_2,\a_2,\sP_{(X_2,\a_2)},Y_2,Z_2)(\tilde{X}_2,\tilde{\a}_2), \tilde{X}_1-\tilde{X}_2\ra]
\big],
\end{split}
\end{align*}
where
we have also applied
 Fubini's theorem
and the fact
 that
 $\sP_{({X}_i,{Y}_i,{Z}_i,{\a}_i)}=\tilde{\sP}_{(\tilde{X}_i,\tilde{Y}_i,\tilde{Z}_i,\tilde{\a}_i)}$ for all $i=1,2$.

Therefore, we can conclude from \eqref{eq:mfcE_H_convex} that
\begin{align*}
\begin{split}
&\sE\big[\la \hat{b}(t,X_1,Y_1,\sP_{(X_1,Y_1)})-\hat{b}(t,X_2,Y_2,\sP_{(X_2,Y_2)}),  Y_1-Y_2\ra
\\
&\q+\la \sigma(t,X_1,\sP_{X_1})-\sigma(t,X_2,\sP_{X_2}),  Z_1-Z_2\ra
+ \la -\hat{f}(t,\Theta_1,\sP_{\Theta_1})+\hat{f}(t,\Theta_2,\sP_{\Theta_2}), X_1-X_2\ra\big]
\\
&=\sE\Big[H(t,X_1,\a_1,\sP_{(X_1,\a_1)},Y_1,Z_1)-H(t,X_2,\a_2,\sP_{(X_2,\a_2)},Y_1,Z_1)
%
\\
&\quad
-\la \p_x H(t,X_1,\a_1,\sP_{(X_1,\a_1)},Y_1,Z_1), {X}_1-{X}_2\ra
 \\
& \quad - \tilde{\sE}[\la
\p_\mu H(t,X_1,\a_1,\sP_{(X_1,\a_1)},Y_1,Z_1)(\tilde{X}_1,\tilde{\a}_1), \tilde{X}_1-\tilde{X}_2\ra]
\Big]
\\
&\q
-\sE\Big[H(t,X_1,\a_1,\sP_{(X_1,\a_1)},Y_2,Z_2) - H(t,X_2,\a_2,\sP_{(X_2,\a_2)},Y_2,Z_2)
\\
&\quad
-\la \p_x H(t,X_2,\a_2,\sP_{(X_2,\a_2)},Y_2,Z_2), {X}_1-{X}_2\ra
\\
& \quad -\tilde{\sE}[
\la
\p_\mu H(t,X_2,\a_2,\sP_{(X_2,\a_2)},Y_2,Z_2)(\tilde{X}_2,\tilde{\a}_2), \tilde{X}_1-\tilde{X}_2\ra]
\Big]
\\
&\le 
-2(\lambda_1 +\lambda_2) \sE[|\a_1-\a_2|^2], 
\end{split}
\end{align*}
where 
we have applied
Fubini's theorem,
\eqref{eq:opti_markov},
 and the definitions of 
$(\a_1,\a_2)$
  to derive the last estimate.
This shows the desired monotonicity property of $\hat{f}$ 
and completes the proof.
\end{proof}

To prove Proposition 
\ref{prop:fbsde_wellposedness}, we consider 
 the following family of MV-FBSDEs:
for $t\in [0,T]$,
\begin{align}\l{eq:moc}
\begin{split}
\d X_t&=(\lambda \hat{b}(t,X_t,Y_t,\sP_{(X_t,Y_t)})+\cI^{\hat{b}}_t)\,\d t +
(\lambda\sigma (t,X_t, \sP_{X_t})+\cI^\sigma_t)\, \d W_t, 
\\
\d Y_t&=-(\lambda \hat{f}(t,X_t,Y_t,Z_t, \sP_{(X_t,Y_t,Z_t)})+\cI^{\hat{f}}_t)\,\d t+Z_t\,\d W_t,
\\
X_0&=\xi,\q Y_T=\lambda \hat{g}(X_T,\sP_{X_T})+\cI^{\hat{g}}_T,
\end{split}
\end{align}
where  $\lambda\in [0,1]$, 
$\xi\in L^2(\cF_0;\sR^n)$,
$(\cI^{\hat{b}},\cI^\sigma,\cI^{\hat{f}})\in \cH^2(\sR^n\t\sR^{n\t d}\t \sR^n)$ and $\cI^{\hat{g}}_T\in L^2(\cF_T;\sR^n)$
are given.
The following lemma establishes the stability of \eqref{eq:moc}.
\begin{Lemma}\l{lemma:mono_stab}
Suppose (H.\ref{assum:mfcE}) and (H.\ref{assum:mfcE_hat}(\ref{item:ex})) hold,
and let the functions 
$(\hat{b},\hat{f},\hat{g})$ 
be defined as in \eqref{eq:mfc_coefficients2}.
 Then, there exists a constant $C>0$ such that,  
 for all $\lambda_0\in [0,1]$,
for every   
${\Theta}\coloneqq(X,Y, Z)\in  \cS^2(\sR^n) \t \cS^2(\sR^n) \t \cH^2(\sR^{n\t d})$
satisfying \eqref{eq:moc}
with 
$\lambda=\lambda_0$,
  functions $(\hat{b},\sigma,\hat{f},\hat{g})$
  and some
$(\cI^{\hat{b}},\cI^\sigma,\cI^{\hat{f}})\in \cH^2(\sR^n\t\sR^{n\t d}\t \sR^n)$,
 $\cI^{\hat{g}}_T\in L^2(\cF_T;\sR^n)$,
 $\xi\in L^2(\cF_0;\sR^n)$,
 and for every 
$ \bar{\Theta}\coloneqq(\bar{X},\bar{Y}, \bar{Z})\in  \cS^2(\sR^n) \t \cS^2(\sR^n) \t \cH^2(\sR^{n\t d})$
satisfying  \eqref{eq:moc}
with 
$\lambda=\lambda_0$,
another 4-tuple of Lipschitz functions $(\bar{b},\bar{\sigma},\bar{f},\bar{g})$ 
 and some 
$(\bar{\cI}^b,\bar{\cI}^\sigma,\bar{\cI}^f)\in \cH^2(\sR^n\t\sR^{n\t d}\t \sR^n)$,
 $\bar{\cI}^g_T\in L^2(\cF_T;\sR^n)$,
 $\bar{\xi}\in L^2(\cF_0;\sR^n)$, we have that 
 \begin{align}\l{eq:mono_stab}
 \begin{split}
 &\|X-\bar{X}\|_{\cS^2}^2+ \|Y-\bar{Y}\|_{\cS^2}^2+ \|Z-\bar{Z}\|_{\cH^2}^2
 \\
& \le C\bigg\{\|\xi-\bar{\xi}\|_{L^2}^2+
\|\lambda_0 (\hat{g}(\bar{X}_T,\sP_{\bar{X}_T})-\bar{g}(\bar{X}_T,\sP_{\bar{X}_T}))+\cI^{\hat{g}}_T-\bar{\cI}^g_T\|_{L^2}^2
\\
&\quad
+\|\lambda_0 (\hat{b}(\cdot,\bar{X}_\cdot,\bar{Y}_\cdot,\sP_{(\bar{X},\bar{Y})_\cdot})-\bar{b}(\cdot,\bar{X}_\cdot,\bar{Y}_\cdot,\sP_{(\bar{X},\bar{Y})_\cdot}))
+\cI^{\hat{b}}-\bar{\cI}^b\|_{\cH^2}^2
\\
&\quad
+\|\lambda_0 (\sigma(\cdot,\bar{X}_\cdot,\sP_{\bar{X}_\cdot})-\bar{\sigma}(\cdot,\bar{X}_\cdot,\sP_{\bar{X}_\cdot}))
+\cI^\sigma-\bar{\cI}^\sigma\|_{\cH^2}^2
\\
&\quad
+\|\lambda_0 (\hat{f}(\cdot,\bar{\Theta}_\cdot,\sP_{\bar{\Theta}_\cdot})-\bar{f}(\cdot,\bar{\Theta}_\cdot,\sP_{\bar{\Theta}_\cdot}))
+\cI^{\hat{f}}-\bar{\cI}^f\|_{\cH^2}^2
\bigg\}.
\end{split}
  \end{align} 
\end{Lemma}

\begin{proof}[Proof of Lemma \ref{lemma:mono_stab}]
{
For ease of notation, we will write $b,f,g$ instead of $\hat{b}, \hat{f}, \hat{g}$.
Also, throughout this proof, let 
$\delta \xi=\xi-\bar{\xi}$,
$\delta \cI^g_T=\cI^g_T-\bar{\cI}^g_T$,
  $g(X_T)=g(X_T,\sP_{X_T})$,
$g(\bar{X}_T)=g(\bar{X}_T,\sP_{\bar{X}_T})$ 
and $\bar{g}(\bar{X}_T)=\bar{g}(\bar{X}_T,\sP_{\bar{X}_T})$,
for each $t\in [0,T]$
 let 
 $\delta\cI^b_t=\cI^b_t-\bar{\cI}^b_t$,
  $\delta\cI^\sigma_t=\cI^\sigma_t-\bar{\cI}^\sigma_t$,
   $\delta\cI^f_t=\cI^f_t-\bar{\cI}^f_t$,
  $f(\Theta_t)=f(t,X_t,Y_t,\sP_{\Theta_t})$,
$f(\bar{\Theta}_t)=f(t,\bar{X}_t,\bar{Y}_t,\sP_{\bar{\Theta}_t})$
and $\bar{f}(\bar{\Theta}_t)=\bar{f}(t,\bar{X}_t,\bar{Y}_t,\sP_{\bar{\Theta}_t})$.
Similarly, we introduce the notation
$\sigma(X_t), \sigma(\bar{X}_t), \bar{\sigma}(\bar{X}_t)$ and $b(X_t,Y_t), b(\bar{X}_t,\bar{Y}_t), \bar{b}(\bar{X}_t,\bar{Y}_t)$ for $t\in [0,T]$.

By applying It\^{o}'s formula to $\la Y_t-\bar{Y}_t,X_t-\bar{X}_t\ra$,  we obtain that 
\begin{align*}
&\sE[\la \lambda_0(g(X_T)-\bar{g}(\bar{X}_T))+\delta I^g_T,X_T-\bar{X}_T \ra]
-\sE[\la Y_0-\bar{Y}_0, \delta \xi\ra]
\\
&=\sE\bigg[\int_0^T
\la \lambda_0(b(X_t,Y_t)-\bar{b}(\bar{X}_t,\bar{Y}_t))+\delta \cI^b_t,Y_t-\bar{Y}_t\ra 
+
\la \lambda_0(\sigma(X_t)-\bar{\sigma}(\bar{X}_t))+\delta \cI^\sigma_t,Z_t-\bar{Z}_t \ra
\\
&\quad 
+
\la -\big(\lambda_0(f(\Theta_t)-\bar{f}(\bar{\Theta}_t))+\delta \cI^f_t\big),X_t-\bar{X}_t \ra
\, \d t
\bigg].
\end{align*}
Then, 
by adding and subtracting the terms
$g(\bar{X}_T), b(\bar{X}_t,\bar{Y}_t),\sigma(\bar{X}_t),f(\bar{\Theta}_t)$
and applying the monotonicity property established in Proposition \ref{prop:mcfE2}, we can deduce that
\begin{align*}
&
\sE[\la \lambda_0(g(\bar{X}_T)-\bar{g}(\bar{X}_T))+\delta I^g_T,X_T-\bar{X}_T\ra]
-\sE[\la Y_0-\bar{Y}_0,\delta \xi\ra]
\\
&\le 
\sE\bigg[\int_0^T
\la \lambda_0(b(\bar{X}_t,\bar{Y}_t)-\bar{b}(\bar{X}_t,\bar{Y}_t)+\delta \cI^b_t, Y_t-\bar{Y}_t\ra 
+
\la \lambda_0(\sigma(\bar{X}_t)-\bar{\sigma}(\bar{X}_t))+\delta \cI^\sigma_t,Z_t-\bar{Z}_t \ra
\\
&\quad 
+
\la -\big(\lambda_0(f(\bar{\Theta}_t)-\bar{f}(\bar{\Theta}_t))+\delta \cI^f_t\big),X_t-\bar{X}_t\ra
\, \d t
\bigg]-2(\lambda_1 + \lambda_2)\lambda_0\int_0^T \phi_1(t,X_t,Y_t,\bar{X}_t,\bar{Y}_t) \, \d t,
\end{align*}
with $\phi_1(t,X_t,Y_t,\bar{X}_t,\bar{Y}_t) := \| \hat{\alpha}(t,X_t,Y_t,\mathbb{P}_{(X_t,Y_t)}) - \hat{\alpha}(t,\bar{X}_t, \bar{Y}_t,\mathbb{P}_{(\bar{X}_t,\bar{Y}_t)})  \|^2_{L_2}$,
which together with Young's inequality yields for each $\eps>0$ that
\begin{align}\l{eq:mono_stab_decouple}
\begin{split}
&
2(\lambda_1 + \lambda_2)\lambda_0 \int_0^T
 \phi_1(t,X_t,Y_t,\bar{X}_t,\bar{Y}_t) \, \d t
  \\
  &
  \le
\eps( \|X_T-\bar{X}_T\|_{L^2}^2+\|Y_0-\bar{Y}_0\|_{L^2}^2
 +\|\Theta-\bar{\Theta}\|_{\cH^2}^2)
 +C{\eps}^{-1}\textrm{RHS},
 \end{split}
\end{align}
where $\textrm{RHS}$ denotes the right-hand side of \eqref{eq:mono_stab}.

Now, by \eqref{eq:mono_stab_decouple} and the fact that $\lambda_1 + \lambda_2 >0$,
we have  for all $\eps>0$,
\begin{align}\l{eq:mono_stab_decouple_m<n}
\begin{split}
&
\lambda_0\int_0^T
  \phi_1(t,X_t,Y_t,\bar{X}_t,\bar{Y}_t)\, \d t
  \le
\eps( \|X-\bar{X}\|_{\cS^2}^2+\|Y-\bar{Y}\|_{\cS^2}^2
 +\|Z-\bar{Z}\|_{\cH^2}^2)
 +C{\eps}^{-1}\textrm{RHS}.
 \end{split}
\end{align}
Then, by using the Burkholder-Davis-Gundy  inequality, the definition of \eqref{eq:moc} and \eqref{eq:mfc_coefficients2}, Gronwall's inequality
and the fact that $\lambda_0\in [0,1]$, we can 
deduce that 
\begin{align*}
\|X-\bar{X}\|_{\cS^2}^2
&\le 
C\bigg(
\int_0^T \lambda_0\phi_1(t,X_t,Y_t,\bar{X}_t,\bar{Y}_t)\, \d t
+
\|\xi-\bar{\xi}\|_{L^2}^2
\\
&\quad
+\|\lambda_0 (b(\bar{X},\bar{Y})-\bar{b}(\bar{X},\bar{Y})
+\delta \cI^b\|_{\cH^2}^2
+\|\lambda_0 (\sigma(\bar{X})-\bar{\sigma}(\bar{X}))
+\delta \cI^\sigma\|_{\cH^2}^2
\bigg),
\end{align*}
which together with \eqref{eq:mono_stab_decouple_m<n} yields for all small enough $\eps>0$ that 
\begin{align*}
\begin{split}
&
\|X-\bar{X}\|_{\cS^2}^2
  \le
\eps(\|Y-\bar{Y}\|_{\cS^2}^2
 +\|Z-\bar{Z}\|_{\cH^2}^2)
 +C{\eps}^{-1}\textrm{RHS}.
 \end{split}
\end{align*}
Moreover,  by standard estimates for MV-BSDEs, we can obtain that
\begin{align*}
&\|Y-\bar{Y}\|_{\cS^2}^2
 +\|Z-\bar{Z}\|_{\cH^2}^2
 \\
&  \le
  C\bigg(\|X-\bar{X}\|_{\cS^2}^2
  +
  \|\lambda_0 (g(\bar{X}_T)-\bar{g}(\bar{X}_T))+\delta \cI^g_T\|_{L^2}^2
  +\|\lambda_0 (f(\bar{\Theta})-\bar{f}(\bar{\Theta}))
+\delta\cI^f\|_{\cH^2}^2
\bigg),
\end{align*}
which completes the desired estimate \eqref{eq:mono_stab}.}
\end{proof}

\begin{proof}[Proof of Proposition \ref{prop:fbsde_wellposedness}]
{We shall establish the  well-posedness, stability and \textit{a priori} estimates
 for \eqref{eq:mfc_fbsde2} with
an initial time
 $t=0$ and initial state $\xi_0\in L^2(\cF_0;\sR^n)$
 by applying Lemma \ref{lemma:mono_stab}. Similar arguments apply to a general 
 initial time 
 $t\in [0,T]$ and initial state $\xi\in L^2(\cF_t;\sR^n)$.

Let us start by proving 
the unique solvability of \eqref{eq:mfc_fbsde2} with a given $\xi_0\in L^2(\cF_0;\sR^n)$.
To simplify the notation, 
for every $\lambda_0\in [0,1]$, we say $(\cP_{\lambda_0})$ holds if 
 for any 
 $\xi\in  L^2(\cF_0;\sR^n)$,
$(\cI^{\hat{b}},\cI^\sigma,\cI^{\hat{f}})\in \cH^2(\sR^n\t\sR^{n\t d}\t \sR^n)$ and $\cI^{\hat{g}}_T\in L^2(\cF_T;\sR^n)$,
\eqref{eq:moc} with $\lambda=\lambda_0$ 
admits a unique solution in $\sB\coloneqq \cS^2(\sR^n) \t \cS^2(\sR^n) \t \cH^2(\sR^{n\t d})$.
It is clear that $(\cP_{0})$ holds since \eqref{eq:moc} is decoupled. 
Now we show 
 there exists a constant $\delta>0$, such that 
 if $(\cP_{\lambda_0})$   holds for some $\lambda_0\in [0,1)$,
 then $(\cP_{\lambda'_0})$ also holds for  all $\lambda_0'\in (\lambda_0,\lambda_0+\delta]\cap [0,1]$.
Note that this claim along with  the method of continuation
 implies the desired unique solvability of \eqref{eq:mfc_fbsde2} (i.e., \eqref{eq:moc} with $\lambda=1$,
$(\cI^{\hat{b}},\cI^\sigma,\cI^{\hat{f}},\cI^{\hat{g}}_T)=0$, $\xi=\xi_0$).

To establish the desired claim, let  $\lambda_0\in [0,1)$
be a constant for which $(\cP_{\lambda_0})$ holds,
$\eta\in [0,1]$ and 
$(\tilde{\cI}^{\hat{b}},\tilde{\cI}^\sigma,\tilde{\cI}^{\hat{f}})\in \cH^2(\sR^n\t\sR^{n\t d}\t \sR^n)$, $\tilde{\cI}^{\hat{g}}_T\in L^2(\cF_T;\sR^n)$, 
$\xi\in  L^2(\cF_0;\sR^n)$
be
arbitrarily given coefficients. Then,
we introduce the following mapping $\Xi:\sB\to \sB$
such that 
for all $\Theta=(X,Y,Z)\in \sB$, $\Xi(\Theta)\in \sB$ is the solution to \eqref{eq:moc} 
with $\lambda=\lambda_0$,
$\cI^{\hat{b}}_t=\eta \hat{b}(t,X_t,Y_t\sP_{(X_t,Y_t)})+\tilde{\cI}^{\hat{b}}_t$,
$\cI^\sigma_t=\eta \sigma (t,X_t, \sP_{X_t})+\tilde{\cI}^\sigma_t$,
$\cI^{\hat{f}}_t=\eta \hat{f}(t,\Theta_t, \sP_{\Theta_t})+\tilde{\cI}^{\hat{f}}_t$
and $\cI^{\hat{g}}_T=\eta \hat{g}(X_T,\sP_{X_T})+\tilde{\cI}^{\hat{g}}_T$,
which is well-defined due to the fact that $\lambda_0\in [0,1)$ satisfies the induction hypothesis.
Observe that
by setting $(\bar{b},\bar{\sigma},\bar{f},\bar{g})=(\hat{b},\sigma,\hat{f},\hat{g})$ in Lemma \ref{lemma:mono_stab},
 we see that there exists a constant $C>0$, independent of $\lambda_0$, such that it holds  for all $\Theta,\Theta'\in \sB$ that
 \begin{align*}
 \begin{split}
 &\|\Xi(\Theta)-\Xi({\Theta}')\|_{\sB}^2
 \\
& \le C\bigg\{\|\eta (\hat{g}({X}_T,\sP_{{X}_T})-\hat{g}({X}'_T,\sP_{{X}'_T}))\|_{L^2}^2
+\|\eta (\hat{b}(\cdot,{X}_\cdot,{Y}_\cdot,\sP_{{(X,Y)}_\cdot})-\hat{b}(\cdot,{X}'_\cdot,{Y}'_\cdot,\sP_{{X',Y'}_\cdot}))
\|_{\cH^2}^2
\\
&\quad
+\|\eta (\sigma(\cdot,{X}_\cdot,\sP_{{X}_\cdot})-{\sigma}(\cdot,{X}'_\cdot,\sP_{{X'}_\cdot}))
\|_{\cH^2}^2
+\|\eta (\hat{f}(\cdot,{\Theta}_\cdot,\sP_{{\Theta}_\cdot})-\hat{f}(\cdot,{\Theta}'_\cdot,\sP_{{\Theta}'_\cdot}))
\|_{\cH^2}^2
\bigg\}
\\
&\le C\eta^2\|\Theta-{\Theta}'\|_{\sB}^2,
\end{split}
  \end{align*}
which shows that $\Xi$ is a contraction when $\eta$ is sufficiently small (independent of $\lambda_0$),
and subsequently leads to  the desired claim due to Banach's fixed point theorem.

For any given $\xi, \xi'\in L^2(\cF_0;\sR^n)$,
the desired stochastic stability of \eqref{eq:mfc_fbsde2} follows directly from 
 Lemma \ref{lemma:mono_stab}
 by setting 
$\lambda=1$,
$(\bar{b},\bar{\sigma},\bar{f},\bar{g})=(\hat{b},\sigma,\hat{f},\hat{g})$,
$(\bar{\cI}^b,\bar{\cI}^\sigma,\bar{\cI}^f)=(\cI^{\hat{b}},\cI^\sigma,\cI^{\hat{f}})=0$,
 $\bar{\cI}^g_T=\cI^{\hat{g}}_T=0$ and
 $\bar{\xi}=\xi'$.
Moreover,  
for any given $\xi\in L^2(\cF_0;\sR^n)$,
 by setting 
$\lambda=1$,
$(\bar{b},\bar{\sigma},\bar{f},\bar{g})=0$,
$(\bar{\cI}^b,\bar{\cI}^\sigma,\bar{\cI}^f)=(\cI^{\hat{b}},\cI^\sigma,\cI^{\hat{f}})=0$,
 $\bar{\cI}^g_T=\cI^{\hat{g}}_T=0$,
 $\bar{\xi}=0$
 and $(\bar{X},\bar{Y},\bar{Z})=0$
 in  Lemma \ref{lemma:mono_stab},
we can deduce the 
 estimate that
 \begin{align*}
 \begin{split}
 &\|X\|_{\cS^2}^2+ \|Y\|_{\cS^2}^2+ \|Z\|_{\cH^2}^2
 \\
& \le C\bigg\{\|\xi\|_{L^2}^2+
| \hat{g}(0,\bm{\delta}_{{0}_{n}})|^2
+\| \hat{b}(\cdot, 0,\bm{\delta}_{{0}_{n+n}})\|_{L^2(0,T)}^2
+\| \sigma(\cdot, 0,\bm{\delta}_{{0}_{n}})\|_{L^2(0,T)}^2
\\
&\quad
+\| \hat{f}(\cdot, 0,\bm{\delta}_{{0}_{n+n+nd}})\|_{L^2(0,T)}^2
\bigg\}
\le C(1+\|\xi\|_{L^2}^2),
\end{split}
\end{align*}
which shows the desired moment bound of the processes $(X,Y,Z)$.
}
%
\end{proof}

\section{It\^{o}'s formula along a flow of measures}
\label{sec:ito_formula}

This section reviews the 
notion of the $L$--derivative 
for functions defined on the space of probability measures,
the related function spaces, 
and  the corresponding   It\^o's formula. 
 

The lift of a function $u:\mathcal{P}_2(\mathbb{R}^n) \to \mathbb{R}$ is a function $\Tilde{u}:L^2(\Omega, \mathcal{F},\mathbb{P};\mathbb{R}^n) \to \mathbb{R}$ such that $\Tilde{u}(\theta) = u(\mathbb{P}(\theta))$ for any $\theta \in L^2(\Omega, \mathcal{F},\mathbb{P};\mathbb{R}^n)$. $\Tilde{u}$ is called Fr\'echet differentiable at $\theta_0 \in L^2(\Omega, \mathcal{F},\mathbb{P};\mathbb{R}^n)$, if there exists a linear continuous mapping $D\Tilde{u}(\theta_0):L^2(\Omega, \mathcal{F},\mathbb{P};\mathbb{R}^n)\to \mathbb{R}$ such that as $\lVert \theta-\theta_0\rVert_{L^2}\to 0,$
\begin{equation*}
   \Tilde{u}(\theta)=\Tilde{u}(\theta_0) +\mathbb{E}\big[D\Tilde{u}(\theta_0)\cdot(\theta-\theta_0)\big]+o(\lVert \theta-\theta_0\rVert_{L^2}).
\end{equation*}
The Fr\'echet derivative $[D\Tilde{u}](\theta_0)$ is viewed as an element of $L^2(\Omega, \mathcal{F},\mathbb{P};\mathbb{R}^n)$ so that by Riesz Theorem, $[D\Tilde{u}](\theta)(\phi) = \mathbb{E}[D\Tilde{u}(\theta)\cdot(\phi)]$ and can be written as  
\begin{equation}\label{L-derivative}
    D\Tilde{u}(\theta_0) = \partial_{\mu}u(\mathbb{P}_{\theta_0})(\theta_0),
\end{equation}
for some Borel function $\partial_{\mu}u(\mathbb{P}_{\theta_0}):\mathbb{R}^n \to \mathbb{R}^n$, see for example \cite{carmona2018a}.
\begin{Definition}{\cite[Definition 2.1]{GuoPhaWei}}
$u$ is L-differentiable at $\mu_0 = \mathbb{P}_{\theta_0} \in \mathcal{P}_2(\mathbb{R}^n)$ if its lift function $\Tilde{u}$ is Fr\'echet differentiable at $\theta_0$. The representation of the L-derivative of $u$ at $\mu_0$ along variable $\theta_0$ is then function $\partial_{\mu}u(\mu_0)(.)$ as in \ref{L-derivative}. 
\end{Definition}


The following function space 
will be used in It\^{o}'s formula. 

\begin{Definition}{\cite[Definition 3.2]{CosGozKha}}\label{D1}
The space
$C^{1,2}_2([0,T] \times \mathcal{P}_2(\mathbb{R}^n))$ consists of functions $v:[0,T] \times \mathcal{P}_2(\mathbb{R}^n)\to\mathbb{R}$ that satisfy the following properties:
  \begin{enumerate}[(1)]
      \item The lifting $V$ of $v$ has a continuous Fr\'echet derivative $D_{\xi}V:[0,T]\times L^2(\Omega,\mathcal{F},\mathbb{P};\mathbb{R}^n)\to L^2(\Omega,\mathcal{F},\mathbb{P};\mathbb{R}^n)$, so that for all $(t,\mu)\in [0,T]\times \mathcal{P}_2(\mathbb{R}^n)$ there exists a measurable function $\partial_{\mu}v(t,\mu):\mathbb{R}^n \to \mathbb{R}^n,$ such that $D_{\xi}V(t,\xi) = \partial_{\mu}v(t,\mu)(\xi)$, for all $\xi \in L^2(\Omega,\mathcal{F},\mathbb{P};\mathbb{R}^n)$ with law $\mu$.
      \item The map $[0,T]\times \mathbb{R}^n \times \mathcal{P}_2(\mathbb{R}^n) \ni (t,x,\mu)\to \partial_{\mu}v(t,\mu)(x) \in \mathbb{R}^n$ is jointly continuous.
      \item The functions $\partial_t v$ and $\partial_x \partial_{\mu}v$ exist and the maps $[0,T]\times \mathcal{P}_2({\mathbb{R}}^n)\ni(t,\mu)\to \partial_tv(t,\mu)\in \mathbb{R}$ and $[0,T]\times \mathbb{R}^n \times \mathcal{P}_2(\mathbb{R}^n)\ni(t,x,\mu)\to \partial_x \partial_{\mu}v(t,\mu)(x)\in \mathbb{R}^{n\times n}$ are continuous.
      \item 
      There exists a constant $C\ge 0$
      such that
      for all $(t,x,\mu) \in [0,T]\times \mathbb{R}^n \times \mathcal{P}_2(\mathbb{R}^n)$, 
   $ 
    |\partial_{\mu}v(t,\mu)(x)|+|\partial_x\partial_{\mu}v(t,\mu)(x)|\le C(1+|x|^2).
$ 
  \end{enumerate}   
\end{Definition}

\begin{Theorem}{\cite[Proposition 5.102]{carmona2018a}}\label{timemeasureito} 
Let $t \in [0,T], \xi \in L^2(\Omega,\mathcal{F}_t,\mathbb{P};\mathbb{R}^n) $, and $u \in C^{1,2}_2([0,T] \times \mathcal{P}_2(\mathbb{R}^n))$. Also, for $0\le t \le s \le T$ let 
\[
X_s = \xi + \int_t^s b_r \diff r + \int_t^s \sigma_r \diff W_r,
\]
and $\mu_t$ denote the marginal distribution $\mu_t = \mathbb{P}_{X_t}$. Then for all $s\in[t,T],$ it holds that
\begin{equation}\label{Ito1}
\begin{split}
    u(s,\mu_s) = u(t,\mu_t)&+\int_t^s\partial_tu(r,\mu_r)\diff r + \int_t^s \mathbb{E}\Bigl[ \langle b_r, \partial_{\mu}u(r,\mu_r)(X^{\alpha}_r) \rangle\Bigr] \diff r \\
    & +\frac{1}{2} \int_t^s \mathbb{E}\Bigl[ \text{trace}\big( \sigma_r \sigma_r^\intercal \partial_x\partial_{\mu}u(r,\mu_r)(X^{\alpha}_r) \big)\Bigr] \diff r.
\end{split}
\end{equation}
\end{Theorem}


\section*{Acknowledgements}

Wolfgang Stockinger was supported by a special Upper Austrian  Government grant, and Maria Olympia Tsianni was supported by the Additional Funding Programme for Mathematical Sciences, delivered by EPSRC (EP/V521917/1) and the Heilbronn Institute for Mathematical Research.

\printbibliography

\end{document}